\documentclass[twoside]{irmaems}
\usepackage{amssymb} 
\usepackage{amsmath} 
\usepackage{makeidx}
\makeindex

\usepackage[all]{xy}
\SelectTips{cm}{}

\usepackage{pinlabel,yhmath}
\usepackage{a4}

\renewcommand\theequation{\arabic{section}.\arabic{equation}}

\theoremstyle{definition} 

 \newtheorem{definition}{Definition}[section]
 \newtheorem{remark}[definition]{Remark}
 \newtheorem{example}[definition]{Example}
 \newtheorem{problem}[definition]{Problem}

\theoremstyle{plain}      

 \newtheorem{proposition}[definition]{Proposition}
 \newtheorem{theorem}[definition]{Theorem}
 \newtheorem{corollary}[definition]{Corollary}
 \newtheorem{lemma}[definition]{Lemma}
\newtheorem{conjecture}[definition]{Conjecture}

\newcommand{\ad}{\operatorname{ad}}
\newcommand{\Arf}{\operatorname{Arf}}
\newcommand{\Aut}{\operatorname{Aut}}
\newcommand{\Der}{\operatorname{Der}}
\newcommand{\diff}{\operatorname{d}}
\newcommand{\End}{\operatorname{End}}
\newcommand{\IAut}{\operatorname{IAut}}           
\newcommand{\IOut}{\operatorname{IOut}}           
\newcommand{\incl}{\operatorname{incl}}

\newcommand{\interior}{\operatorname{int}}
\newcommand{\Id}{\operatorname{Id}}
\newcommand{\Gr}{\operatorname{Gr}}
\newcommand{\Hom}{\operatorname{Hom}}
\newcommand{\Ker}{\operatorname{Ker}}
\newcommand{\ODer}{\operatorname{ODer}}      
\newcommand{\Out}{\operatorname{Out}}
\newcommand{\sgn}{\operatorname{sgn}}
\newcommand{\Sp}{\operatorname{Sp}}
\newcommand{\Spin}{\operatorname{Spin}}
\newcommand{\T}{\operatorname{T}}

\newcommand{\GLike}{\operatorname{GLike}}      
\newcommand{\Prim}{\operatorname{Prim}}        

\newcommand{\Q}{\mathbb{Q}}
\newcommand{\R}{\mathbb{R}} 
\newcommand{\Z}{\mathbb{Z}} 

\newcommand{\A}{\mathcal{A}}
\newcommand{\Lie}{\mathfrak{L}}                
\newcommand{\MalcevLie}{\mathfrak{m}}          

\newcommand{\Torelli}{\mathcal{I}}            
\newcommand{\mcg}{\mathcal{M}}                
\newcommand{\mcyl}{\mathbf{c}}                
\newcommand{\cyl}{\mathcal{IC}}               
\newcommand{\cob}{\mathcal{C}}                
\newcommand{\gcob}{\mathcal{H}}               
\newcommand{\gcyl}{\mathcal{IH}}              

\newcommand{\osqcup}{\hphantom{}^<_\sqcup}

\newcommand{\centereddot}{\stackrel{\centerdot}{}}

\newcommand{\clocase}[1]{\wideparen{#1}}      

\newcommand{\closure}[1]{\overline{#1}}      

\newcommand{\set}[1]{\lfloor #1\rceil}

\newcommand{\ie}{i$.$e$.$ }

\newcommand{\Ygraphbottoptop}[3]
{  \begin{array}{c} 
\hphantom{.}\\
\labellist  \hair 2pt 
\pinlabel {$#1$} [t] at 56 0
\pinlabel {$#2$} [br] at 40 165
\pinlabel {$#3$} [bl] at 90 165
\endlabellist
\includegraphics[scale=0.13]{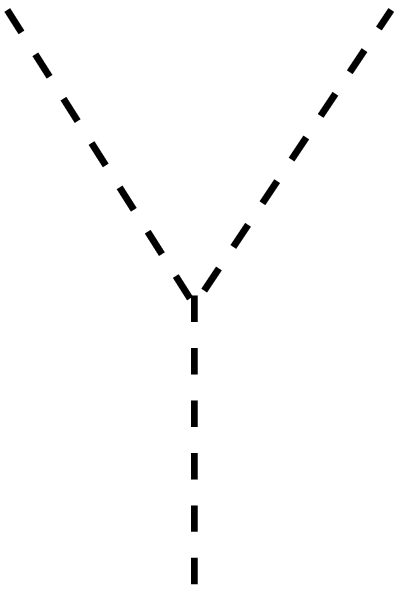}\\
\hphantom{.}
\end{array}  }

\newcommand{\figtotext}[3]{\begin{array}{c}\includegraphics[width=#1pt,height=#2pt]{#3}\end{array}}

\newcommand{\thetagraph}
{\hspace{-0.2cm} \figtotext{18}{18}{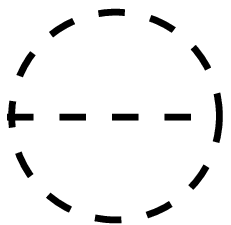} \hspace{-0.2cm}}

\begin{document}

\title{From mapping class groups\\ to  monoids of homology cobordisms: a survey}

\date{February 8, 2012}

\author{Kazuo Habiro and Gw\'ena\"el Massuyeau}

\address{Research Institute for Mathematical Sciences\\ Kyoto University\\ Kyoto 606-8502, Japan\\
email:\,\tt{habiro@kurims.kyoto-u.ac.jp}
\\[4pt]
Institut de Recherche Math\'ematique Avanc\'ee\\ Universit\'e de Strasbourg \& CNRS\\
7 rue Ren\'e Descartes\\ 67084 Strasbourg, France\\
email:\,\tt{massuyeau@math.unistra.fr}
}

\maketitle

\begin{abstract}
Let $\Sigma$ be a compact oriented surface.
A homology cobordism of $\Sigma$ is a cobordism $C$ between two copies of $\Sigma$, 
such that  both the ``top'' inclusion and the ``bottom'' inclusion $\Sigma \subset C$ induce isomorphisms in homology.
Homology cobordisms of $\Sigma$ form a monoid,
into which the mapping class group of $\Sigma$ embeds by the mapping cylinder construction.
In this chapter, we survey recent works on the structure of the monoid of homology cobordisms,
and we outline their relations with the study of the mapping class group.
We are mainly interested in the cases where $\partial \Sigma$ is empty or connected. 
\end{abstract}



\tableofcontents   

\section{Introduction}

It is well-known that $3$-manifolds can be presented in terms of self-homeomorphisms of surfaces.
In particular, any closed  oriented $3$-manifold $M$ can be presented by a \emph{Heegaard splitting}:
then $M$ is obtained by gluing  two copies of the same handlebody $H$
with a self-homeomorphism of the closed surface $\partial H$.
In that way, the study of mapping class groups interplays with $3$-dimensional topology.
For example, Lickorish has used Heegaard splittings to show 
that any closed oriented $3$-manifold can be obtained from $S^3$ by surgery \cite{Lickorish} ---
a theorem deduced from the fact that the mapping class group  is generated by Dehn twists. 
This is  a contribution of mapping class groups to $3$-dimensional topology.
In the other direction and in a more recent context, the TQFT approach to $3$-manifolds 
has produced a variety of representations for mapping class groups: see Masbaum's survey \cite{Masbaum}.
As an alternative to Heegaard splittings, \emph{open book decompositions} with connected binding
can also be used to present all closed oriented $3$-manifolds:
one then needs to consider surfaces with connected nonempty boundary instead of closed surfaces.

Adding homological considerations to $3$-manifolds, 
$3$-dimensional topology becomes intertwined with the study of the Torelli group.
Let $\Sigma$ be a compact connected oriented surface with at most one boundary component.
Then, the Torelli group $\Torelli(\Sigma)$ of the surface $\Sigma$ is defined as the subgroup
of the mapping class group $\mcg(\Sigma)$ acting trivially on the homology of $\Sigma$.
The study of the Torelli group, from a topological point of view, started with works of Birman \cite{Birman_Siegel}
and was followed by Johnson through a series of paper, 
which notably resulted in a finite generating set for $\Torelli(\Sigma)$ \cite{Johnson_finite_generation}
and an explicit computation of its abelianization \cite{Johnson_abelianization}. 
The reader is referred to Johnson's survey \cite{Johnson_survey} for an account of his work on $\Torelli(\Sigma)$.
A key ingredient in his paper \cite{Johnson_abelianization} is the use
of the Birman--Craggs homomorphisms \cite{BC},
whose definition involves the Rochlin invariant of closed spin $3$-manifolds
and can be stated in a simple way using open book decompositions (with connected binding).

The study of the structure of the Torelli group has been pursued by Morita.
He discovered a deep relation with the Casson invariant of homology $3$-spheres \cite{Morita_Casson_1,Morita_Casson_2,Morita_Casson_3}
and he reinforced the use of algebraico-topological methods in that study 
(via the action of $\Torelli(\Sigma)$ on $\pi_1(\Sigma)$) \cite{Morita_Abelian,Morita_linear}.
We refer to his survey \cite{Morita_ICM} and to \cite{Morita_prospect} for more recent developments.
Morita's work on the Casson invariant prefigures 
the use of finite-type invariants of homology $3$-spheres in the study of the Torelli group. 
This approach of $\Torelli(\Sigma)$ has been developed 
in subsequent works of Garoufalidis and Levine \cite{GL_FTI_Torelli,GL_blinks}.

Works of Goussarov \cite{Goussarov,Goussarov_clovers,GGP} 
and the first author \cite{Habiro} breathed new life into this approach to the Torelli group.
In their works, new surgery techniques, known as calculus of claspers,
are developed for the study of finite-type invariants. 
Besides, the study of the mapping class group is tied to $3$-dimensional topology in the following way.
Goussarov and the first author consider \emph{homology cobordisms} of $\Sigma$, \ie cobordisms $C$ (with corners if $\partial \Sigma \neq \varnothing$)
such that both the ``top'' inclusion and the ``bottom'' inclusion $\Sigma \subset C$
induce isomorphisms at the level of homology.
With the usual ``stacking'' operation, the set of homology cobordisms forms a monoid $\cob(\Sigma)$,
into which the group $\mcg(\Sigma)$ embeds by the mapping cylinder construction:
$$
\mcyl: \mcg(\Sigma) \longrightarrow \cob(\Sigma).
$$
Calculus of claspers applies  to the monoid $\cyl(\Sigma)$ of \emph{homology cylinders} over $\Sigma$, 
\ie cobordisms of $\Sigma$ that can not be distinguished from the trivial cylinder by homology.
Since $\Torelli(\Sigma)$ is mapped into $\cyl(\Sigma)$ by $\mcyl$, homology cylinders have opened new perspectives for the Torelli group.\\

In this chapter, we shall survey recent developments in that direction.
The  monoid $\cob(\Sigma)$ of homology cobordisms is presented, 
with special attention given to the submonoid $\cyl(\Sigma)$ of homology cylinders.
We outline how the structure of the monoid $\cyl(\Sigma)$ can be studied by means of finite-type invariants and claspers.
At the same time, we explain how this study is related to more classical results and constructions for the Torelli group.
We are mainly interested in compact connected oriented surfaces $\Sigma$ without boundary 
(referred to, below, as the \emph{closed case}) or with connected boundary (the \emph{bordered case}).
Indeed, when the surface $\Sigma$ has more than one boundary component,
the study of homology cobordisms of $\Sigma$ by means of finite-type invariants is still possible, but it gets a little bit more complicated.
Since closed oriented $3$-manifolds can be presented by Heegaard splittings 
(or, alternatively, by open book decompositions with connected binding),
the interactions between $3$-dimensional topology and mapping class groups 
are somehow contained in the closed case (or in the bordered case).\\

The survey is organized as follows.
Section \ref{sec:cob_cyl} introduces the monoids $\cob(\Sigma)$ and $\cyl(\Sigma)$, starting with their precise definitions.
We give a few elementary facts about these monoids. 
For instance, it is shown that the groups of invertible elements of $\cob(\Sigma)$ and $\cyl(\Sigma)$
coincide with $\mcg(\Sigma)$ and $\Torelli(\Sigma)$, respectively.
We also discuss the passage from the bordered case to the closed case.

Section \ref{sec:Johnson_Morita} reviews constructions based on the Dehn--Nielsen representation, 
which is given by the canonical action of  $\mcg(\Sigma)$ on the fundamental group  $\pi_1(\Sigma)$.
Thus, we survey Johnson's homomorphisms (in the bordered and closed cases)
and their extensions by Morita (in the bordered case). Originally defined on subgroups of the mapping class group, 
these homomorphisms encode the action of $\mcg(\Sigma)$ on nilpotent quotients of $\pi_1(\Sigma)$.
Since, by virtue of Stallings' theorem, a homology equivalence between groups induces isomorphisms at the level of their nilpotent quotients,
Johnson's homomorphisms have natural extensions to  the monoid $\cob(\Sigma)$.
By using the Malcev Lie algebra of $\pi_1(\Sigma)$, 
we define ``infinitesimal'' versions of the Dehn--Nielsen representation and
we use them to reformulate Johnson's and Morita's homomorphisms.

Section \ref{sec:LMO} introduces the LMO homomorphism, 
which is a diagrammatic representation of the monoid $\cyl(\Sigma)$.
Derived from the Le--Murakami--Ohtsuki invariant of closed $3$-manifolds, 
this invariant of homology cylinders is universal among $\Q$-valued finite-type invariants.
Thus, the LMO homomorphism seems very appropriate to the study of the Torelli group from the point of view of finite-type invariants.
The \emph{tree-reduction} of the LMO homomorphism (where all looped diagrams are ``killed'')
encodes the action of $\cyl(\Sigma)$ on the Malcev Lie algebra of $\pi_1(\Sigma)$.
It follows that the LMO homomorphism is injective on the Torelli group, 
and that it determines the Johnson homomorphisms as well as the Morita homomorphisms.
(The results of this section apply to the bordered and closed cases.)

Section \ref{sec:Y} gives an overview of claspers 
and its applications to the study of the monoid $\cyl(\Sigma)$.
We recall the definition of the $Y_k$-equivalence relations ($k\geq 1$), which is based on clasper surgery,
and a few of their properties. These relations define a filtration 
$$
\cyl(\Sigma) = Y_1 \cyl(\Sigma) \supset Y_2 \cyl(\Sigma) \supset Y_3 \cyl(\Sigma) \supset \cdots 
$$
of the monoid $\cyl(\Sigma)$ by submonoids. 
Although $\cyl(\Sigma)$ by itself is only a monoid, 
this filtration shares several properties with the lower central series of a group.
In particular, there is a graded Lie ring $\Gr^Y \cyl(\Sigma)$ associated with this filtration.
When restricted to the Torelli group, the $Y$-filtration contains the lower central series,
and the two filtrations are expected to coincide (in stable genus). 

Section \ref{sec:Lie_ring_homology_cylinders} is an exposition of  results 
on the graded Lie ring $\Gr^Y \cyl(\Sigma)$ for a bordered or a closed surface $\Sigma$.
On one hand, the graded Lie algebra $\Gr^Y \cyl(\Sigma) \otimes \Q$ 
has an explicit diagrammatic description. The proof, which is only sketched here,
is based on the LMO homomorphism and claspers. 
This result is connected to Hain's ``infinitesimal'' presentation of the Torelli group.
On the other hand, the abelian group $\cyl(\Sigma)/Y_2$ 
(\ie the degree $1$ part of $\Gr^Y \cyl(\Sigma)$) is explicitly described 
using the first Johnson homomorphism and the Birman--Craggs homomorphism.
This is similar to Johnson's result on the abelianization of $\Torelli(\Sigma)$,
which was mentioned above.

Finally, following Garoufalidis and Levine,
we consider in Section \ref{sec:group} homology cobordisms up to the \emph{relation} $\sim_H$ of homology cobordism. 
The quotient $\gcob(\Sigma):= \cob(\Sigma)/\!\sim_H$ is a group, 
in which the mapping class group $\mcg(\Sigma)$ still embeds.
We outline the extent to which the previous constructions apply to the study of the group $\gcob(\Sigma)$.
This group has been the subject of recent works in other directions: those developments are only evoked here,
while they are presented in more detail in Sakasai's chapter.\\

\noindent
\textbf{Conventions.}
Homology groups are assumed to be with integer coefficients unless otherwise specified.
On pictures, the blackboard framing convention is used to represent $1$-dimensional objects with framing.
Equivalence classes are denoted by curly brackets $\{\centereddot\}$, 
except for homology/homotopy/isotopy classes which are denoted by square brackets $[\centereddot]$.
We denote by the same symbol a graded vector space  and its degree completion.

\section{Homology cobordisms and homology cylinders}

\label{sec:cob_cyl}

In the sequel, we denote by $\Sigma_{g,b}$ a compact connected oriented surface of genus $g$ with $b$ boundary components.
We define the monoid of cobordisms of $\Sigma_{g,b}$,
and we recall how the mapping class group of $\Sigma_{g,b}$ embeds into this monoid.
At the end of this section, we restrict ourselves to the \emph{closed} surface $\Sigma_g$ ($b=0$)
and to the \emph{bordered} surface $\Sigma_{g,1}$ ($b=1$).

\subsection{Definition of the monoids}

\label{subsec:definitions}

A \emph{cobordism}\index{cobordism} of $\Sigma_{g,b}$ is a pair $(M,m)$ where  $M$ is a compact connected oriented $3$-manifold 
and $m: \partial\left(\Sigma_{g,b} \times [-1,1] \right) \to \partial M$ is an orientation-preserving homeomorphism.
(We will usually denote the cobordism $(M,m)$ simply by $M$, 
the convention being that the boundary parameterization is denoted by the lower-case letter $m$.)
Two cobordisms $M,M'$ are \emph{homeomorphic} if there is an orientation-preserving homeomorphism
$f:M \to M'$ such that $f|_{\partial M} \circ m = m'$.
The inclusion $\Sigma_{g,b} \to M$ defined by $s \mapsto m(s,\pm 1)$ is denoted by $m_\pm$.
We call $m_+(\Sigma_{g,b})$  the \emph{top} surface of $M$ and  $m_-(\Sigma_{g,b})$ the \emph{bottom} surface.

Two cobordisms $M,M'$ of $\Sigma_{g,b}$ can be \emph{composed} by gluing the bottom of $M'$ to the top of $M$:
$$
M \circ M'  := M \cup_{m_+ \circ (m'_-)^{-1}} M'.
$$
The resulting $3$-manifold is a cobordism of $\Sigma_{g,b}$ with the obvious parameterization of its boundary.
When cobordisms are considered up to homeomorphisms, the operation $\circ$ is associative
and the \emph{trivial} cobordism $\Sigma_{g,b} \times [-1,1] := \left(\Sigma_{g,b} \times [-1,1], \Id \right)$
is an identity element for that operation.

\begin{definition}
A \emph{homology cobordism}\index{homology cobordism}\index{cobordism!homology}
of $\Sigma_{g,b}$ is a cobordism $(M,m)$ such that both  inclusions $m_+$ and $m_-$ induce isomorphisms $H_*(\Sigma_{g,b}) \to H_*(M)$.
A \emph{homology cylinder}\index{homology cylinder}\index{cylinder!homology}
over $\Sigma_{g,b}$ is a cobordism $(M,m)$ for which we can find an isomorphism
$h: H_*(\Sigma_{g,b} \times [-1,1]) \to H_*(M)$ such that the following diagram commutes:
$$
\xymatrix{
H_*\left(\Sigma_{g,b} \times [-1,1]\right) \ar[r]_-\simeq^-h & H_*(M)\\
H_*\left(\partial(\Sigma_{g,b} \times [-1,1])\right) \ar[r]^-\simeq_-{m_*} \ar[u]^-{\incl_*}
& H_*(\partial M) \ar[u]_-{\incl_*}.
}
$$
In other words, a homology cylinder $(M,m)$ 
has the same homology type as the trivial cobordism  $\left(\Sigma_{g,b} \times [-1,1], \Id \right)$.
\end{definition}

\noindent
The set of homology cobordisms contains the set of homology cylinders,
and an application of the Mayer--Vietoris theorem shows that both sets are stable by composition. 
In the sequel, homology cobordisms and homology cylinders are considered up to homeomorphisms.
Thus, they form monoids\index{homology cobordism!monoid of}\index{homology cylinder!monoid of}
which we denote by 
$$
\cob\left(\Sigma_{g,b}\right) \supset \cyl\left(\Sigma_{g,b}\right).
$$

\begin{example}[Genus 0]
\label{ex:genus_0_cobordisms}
For $b>1$, a \emph{framed string-link} in $D^2 \times [-1,1]$ 
on $(b-1)$ strands in the sense of \cite{HL_link-homotopy,HL_concordance} 
can be regarded as a cobordism of $\Sigma_{0,b}$ by taking its complement.
More generally, the monoid $\cob(\Sigma_{0,b})$ can be identified with the monoid of framed string-links in homology $3$-balls.  
The monoid $\cyl(\Sigma_{0,b})$ corresponds by this identification 
to the monoid of framed string-links whose linking matrix is trivial.
For $b=0$ or $b=1$, homology cobordisms of $\Sigma_{0,b}$ can be transformed into homology $3$-spheres
by gluing balls along their boundary components:  thus,
the monoid $\cob(\Sigma_{0,b}) = \cyl(\Sigma_{0,b})$ is canonically isomorphic
to the monoid of homology $3$-spheres with the connected sum operation.

\end{example}

\subsection{The mapping class group and the Torelli group}

\label{subsec:groups}

We shall now identify the invertible elements of the monoid $\cob\left(\Sigma_{g,b}\right)$.

\begin{definition}
The \emph{mapping class group}\index{mapping class group}\index{group!mapping class} of $\Sigma_{g,b}$ is the group
$$
\mcg\left(\Sigma_{g,b}\right) := \operatorname{Homeo}_{+}\left(\Sigma_{g,b},\partial \Sigma_{g,b}\right)/ \cong
$$
of isotopy classes of homeomorphisms $\Sigma_{g,b} \to \Sigma_{g,b}$ which preserve the orientation and fix the boundary pointwise.
\end{definition}

A homeomorphism $f:\Sigma_{g,b} \to \Sigma_{g,b}$ (which preserves the orientation and fixes the boundary pointwise)  gives rise to a cobordism
$$
\big(\Sigma_{g,b} \times [-1,1], 
(\Id \times (-1)) \cup (\partial \Sigma_{g,b} \times \Id) \cup (f \times 1)\big),
$$
called  the \emph{mapping cylinder}\index{mapping cylinder}\index{cylinder!mapping} of $f$.  
This construction defines a monoid map
$$
\mcyl: \mcg\left(\Sigma_{g,b}\right)  \longrightarrow \cob\left(\Sigma_{g,b}\right).
$$
By a classical result of Baer \cite{Baer1,Baer2}, any two homeomorphisms   $\Sigma_{g,b} \to \Sigma_{g,b}$
of the above type are homotopic relatively to the boundary if
and only if they are isotopic relatively to the boundary. It follows that the homomorphism $\mcyl$ is injective. 
The following result is folklore, it appears in \cite{HL_concordance} for the genus $0$ case.

\begin{proposition}
An element of the monoid $\cob\left(\Sigma_{g,b}\right)$ is left-invertible if and only if it is right-invertible, 
and the group of invertible elements of  $\cob\left(\Sigma_{g,b}\right)$ coincides with 
the image of $\mcyl: \mcg\left(\Sigma_{g,b}\right)  \to \cob\left(\Sigma_{g,b}\right)$.
\end{proposition}

\begin{proof}
When $g=0$ and $b=0$ or $1$,  the monoid $\cob\left(\Sigma_{g,b}\right)$
can be interpreted as the monoid of homology $3$-spheres. (See Example \ref{ex:genus_0_cobordisms}.)
The proposition then amounts to Alexander's theorem \cite{Alexander}:  the  $3$-manifold $S^3$ is irreducible.
Consequently, we can assume that $\Sigma_{g,b}\neq D^2 , S^2$.  

Let $M,N \in \cob\left(\Sigma_{g,b}\right)$ be such that $M\circ N=\Sigma_{g,b} \times [-1,1]$.
We are asked to show that $M$ is a mapping cylinder.

\begin{quote}
\textbf{Claim.}
The inclusion $m_-: \Sigma_{g,b} \to M$ induces an isomorphism at the level of the fundamental group.
\end{quote}

\noindent
Moreover, the manifold $M$ is irreducible because $M\circ N =  \Sigma_{g,b} \times [-1,1]$ is irreducible.
Then, it follows from \cite[Theorem 10.2]{Hempel}
that there is a homeomorphism between $M$ and $\Sigma_{g,b}\times [-1,1]$ 
which sends $m_-(\Sigma_{g,b})$ to $\Sigma_{g,b} \times (-1)$
and, so,  $m_+(\Sigma_{g,b})$ to $\Sigma_{g,b} \times 1$. The conclusion follows.

To prove the claim, we denote by $i:M \to M \circ N$ and $j: N \to M \circ N$ the inclusions.
The map $i \circ m_-$ is equivalent to the inclusion of $\Sigma_{g,b} \times (-1)$ in 
$\Sigma_{g,b} \times [-1,1]$ and, so, induces an isomorphism at the level of the fundamental group.
Thus, it is enough to prove that $i_*:\pi_1(M) \to \pi_1(M\circ N)$ is injective.
Since $M \circ N$ is the trivial cylinder, the image of $j_* n_{+,*}$ in $\pi_1(M\circ N)$
contains the image of $i_* m_{-,*}$. We deduce the following inclusion of subgroups of $\pi_1(M)$:
$$
m_{-,*}\left(\pi_1(\Sigma_{g,b})\right) \subset {i_*}^{-1}\left( j_*\left(\pi_1(N)\right) \right).
$$
Besides, an application of the van Kampen theorem shows that
\begin{equation}
\label{eq:inverse_image} 
{i_*}^{-1}\left(j_*\left(\pi_1(N)\right)\right) \subset m_{+,*}\left(\pi_1(\Sigma_{g,b})\right)
\end{equation}
so that  $m_{-,*}\left(\pi_1(\Sigma_{g,b})\right) \subset m_{+,*}\left(\pi_1(\Sigma_{g,b})\right)$.
Since the homomorphism $i_* m_{-,*}: \pi_1(\Sigma_{g,b}) \to \pi_1(M\circ N)$ is surjective, we deduce that
$i_* m_{+,*}: \pi_1(\Sigma_{g,b}) \to \pi_1(M\circ N)$ is surjective and, so, is an isomorphism.
(Here, we use the fact that the fundamental group of a surface is Hopfian \cite{Hopf}.)
Another application of (\ref{eq:inverse_image}) then shows that $i_*$ is injective, and we are done.
\end{proof}

Usually in the literature, the Torelli group is only defined when the surface is closed or has a single boundary component.
In our situation, it is natural to define the Torelli group $\Sigma_{g,b}$ for any $b\geq 0$ in the following way.
(When $b=0$ or $b=1$, our definition is equivalent to the usual one recalled in \S \ref{subsec:bordered_to_closed}.)

\begin{definition}
The \emph{Torelli group}\index{Torelli group}\index{group!Torelli} of $\Sigma_{g,b}$ is the subgroup 
$$
\Torelli\left(\Sigma_{g,b}\right) := \mcyl^{-1}\left(\cyl(\Sigma_{g,b})\right) \ \subset \mcg\left(\Sigma_{g,b}\right).
$$
\end{definition}

\begin{example}[Genus $0$]
For $b>1$, the mapping class group of $\Sigma_{0,b}$
can be identified with the group of \emph{framed pure braids} on $(b-1)$ strands over the disk $D^2$ \cite{Birman_book}.
The Torelli group then corresponds to those pure braids with trivial linking matrix.
For $b=0$ or $b=1$, the groups $\mcg(\Sigma_{0,b})$ and $\Torelli(\Sigma_{0,b})$ are trivial. 
\end{example}

\subsection{The closure construction}

\label{subsec:closure}

There is a systematic way of transforming a homology cobordism
into a $3$-manifold \emph{without} boundary: we simply ``close'' a
homology cobordism by identifying its top and bottom surfaces.
This basic construction will be used several times in the following sections 
to define invariants of homology cylinders. So, we would like to define it carefully.

\begin{definition}
The \emph{closure}\index{closure of a homology cobordism}\index{homology cobordism!closure of}
of an $M \in \cob(\Sigma_{g,b})$ is the closed connected oriented $3$-manifold
\begin{equation}
\label{eq:closure}
\closure{M} := (M /\! \sim) \cup \left(S^1 \times D^2\right)_1 \cup \cdots \left(S^1 \times D^2\right)_b.
\end{equation}
Here, the equivalence relation $\sim$ identifies $m_+(s)$ with $m_-(s)$ for all $s\in \Sigma_{g,b}$.
The resulting $3$-manifold $M /\! \sim$ has a toroidal boundary which is recapped by gluing $b$ solid tori.
More precisely, we number the boundary components $\partial_1 \Sigma_{g,b},\dots,\partial_b \Sigma_{g,b}$ 
of $\Sigma_{g,b}$ from $1$ to $b$, and we choose a base point $*_i$ on each of them.
Then, the meridian $1 \times \partial D^2$ of  $\left(S^1 \times D^2\right)_i$ 
is glued along the circle $m(*_i \times [-1,1]) /\! \sim$
while its longitude $S^1 \times 1$ is glued along $m(\partial_i \Sigma_{g,b} \times 0)$. 
\end{definition}

In the special case where $M=\mcyl(f)$ is a mapping cylinder, 
the closure of $M$ is obtained from the mapping torus of $f$ by gluing $b$ solid tori.
If $b=0$, the $3$-manifold $\closure{\mcyl(f)}$ is the total space of a surface bundle over $S^1$ and,
so, one only reaches a special (although interesting) class of closed oriented $3$-manifolds.
If $b>0$, the $3$-manifold $\closure{\mcyl(f)}$ comes with an open book decomposition
whose binding consists of the cores of the glued solid tori. 
Any closed connected oriented $3$-manifold can be presented in that way \cite{Alexander_open_book},
and one can even require the binding to be connected \cite{GA,Myers}, \ie we can assume $b=1$.

\begin{example}
The closure of the trivial cobordism $\Sigma_{g,b} \times [-1,1]$ 
is the product $\Sigma_g \times S^1$ if $b=0$ and the connected sum $\sharp^{2g+b-1}( S^2 \times S^1)$ if $b >0$.
For $g=0$ and $b=1$, the closure operation is the same as the ``recapping'' operation 
described at the end of Example \ref{ex:genus_0_cobordisms}.
\end{example}

\subsection{The closed case and the bordered case}

\label{subsec:bordered_to_closed}

As mentioned in the introduction, we are specially interested
in the closed surface $\Sigma_g$ and the bordered surface $\Sigma_{g,1}$.
The monoids defined in \S \ref{subsec:definitions} are  denoted by
$$
\cob_{g}  \supset \cyl_{g} 
\quad \hbox{and} \quad
\cob_{g,1}  \supset \cyl_{g,1}
$$
in the closed case and the bordered case, respectively. 
Observe that, for $b=0$ and $b=1$, a homology cobordism $(M,m)$ of $\Sigma_{g,b}$
is a homology cylinder if and only if $m_+$ and $m_-$ induce 
the \emph{same} isomorphism $H_*\left(\Sigma_{g,b}\right) \to H_*(M)$.
Similarly, the groups introduced in \S \ref{subsec:groups} are simply denoted by
$$
\mcg_{g}  \supset \Torelli_{g} 
\quad \hbox{and} \quad
\mcg_{g,1}  \supset \Torelli_{g,1}.
$$ 
Observe that, for $b=0$ and $b=1$, an $f\in \mcg\left(\Sigma_{g,b}\right)$
belongs to $\Torelli\left(\Sigma_{g,b}\right)$  if and only if $f$ induces the identity in homology.
This is the way the Torelli group is usually defined in the literature \cite{Johnson_survey}.

The study of $\mcg_g$ can be somehow ``reduced'' to the study of $\mcg_{g,1}$ thanks to Birman's exact sequence, which we now recall.
For this, we need to fix a closed disk $D \subset \Sigma_g$. Then, we can think of the closed surface $\Sigma_g$ as the union of $\Sigma_{g,1}$ with $D$.

\begin{theorem}[Birman's exact sequence \cite{Birman_exact_sequence}]
\label{th:Birman}
Give $\Sigma_{g}$ an arbitrary smooth structure as well as a Riemannian metric,
and let $\operatorname{U}(\Sigma_{g})$ be the total space of the unit tangent bundle of $\Sigma_{g}$.
Then, there is an exact sequence of groups
$$
\xymatrix{
\pi_1\left(\operatorname{U}(\Sigma_{g})\right) \ar[r]^-{\ \operatorname{Push}\ }
& \mcg_{g,1} \ar[r]^-{\ \centereddot \cup \Id_D\ } & \mcg_g \ar[r] & 1
}
$$
where the ``Push'' map is defined below.
\end{theorem}

\begin{proof}[Sketch of the proof]
Let $\operatorname{Diffeo}_{+}(\Sigma_{g})$ be the group of orientation-preserving 
diffeomorphisms $\Sigma_g \to \Sigma_g$.
Since ``diffeotopy groups'' coincide with ``homeotopy groups'' in dimension two, we have
\begin{equation}
\label{eq:Birman_one}
\mcg_g = \pi_0\left(\operatorname{Diffeo}_{+}(\Sigma_g)\right).
\end{equation}
Let $v$ be a  ``unit tangent vector'' of $D$: $v\in T_p\Sigma_g$ with $\Vert v \Vert =1$ and $p\in D$.
Then, we can consider the subgroup $\operatorname{Diffeo}_{+}(\Sigma_g,v)$ 
consisting of diffeomorphisms whose differential fixes $v$. One can show that
\begin{equation}
\label{eq:Birman_two}
\mcg_{g,1} \simeq \pi_0\left(\operatorname{Diffeo}_{+}(\Sigma_g,v)\right).
\end{equation}
The map $\operatorname{Diffeo}_{+}(\Sigma_g) \to \operatorname{U}(\Sigma_g)$ defined by
$f \mapsto \operatorname{d}_p\! f(v)$ is a fiber bundle with fiber $\operatorname{Diffeo}_{+}(\Sigma_g,v)$.
According to (\ref{eq:Birman_one}) and (\ref{eq:Birman_two}), 
the long exact sequence for homotopy groups induced by this fibration terminates with
$$
\pi_1\left(\operatorname{Diffeo}_{+}(\Sigma_g),\Id\right) \longrightarrow
\pi_1\left(\operatorname{U}(\Sigma_g),v\right) \longrightarrow \mcg_{g,1} \longrightarrow \mcg_g \longrightarrow 1.
$$
The map $\pi_1\left(\operatorname{U}(\Sigma_g),v\right) \to \mcg_{g,1}$ is called 
the ``Push'' map because it has the following description. 
A loop $\gamma$ in $\operatorname{U}(\Sigma_g)$ based at $v$ can be
regarded as an isotopy $I:D \times [0,1] \to \Sigma_g$
of the disk $D$ in $\Sigma_g$ such that  $I(\centereddot,0)=I(\centereddot,1)$ is the inclusion $D \subset \Sigma_g$.
This isotopy can be extended to an ambient isotopy $\overline{I}:\Sigma_g \times [0,1] \to \Sigma_g$
starting with $\overline{I}(\centereddot,0)=\Id_{\Sigma_g}$. Then, 
$$
\operatorname{Push}([\gamma]) := \left[ \hbox{restriction of $\overline{I}(\centereddot,1)$ 
to $\Sigma_{g,1}=\Sigma_g \setminus \operatorname{int}(D)$} \right]
$$
is the image of the loop $\gamma$ in $\mcg_{g,1}$.
\end{proof}

For cobordisms, we have the following analogue of Birman's exact sequence.
This lemma will be used in the following sections to define homomorphisms on $\cob_g$
by first defining them on $\cob_{g,1}$.

\begin{lemma}
\label{lem:closing}
Gluing a $2$-handle along $\partial \Sigma_{g,1} \times [-1,1]$ defines a surjection
$$
\cob_{g,1} \longrightarrow \cob_g, \ M \longmapsto \clocase{M}.
$$
Moreover, two cobordisms $M$ and $M'$ satisfy $\clocase{M} = \clocase{M'}$ 
if and only if $M'$ can be  obtained from $M$ by surgery
along a $2$-component framed link $K_0 \cup K_1$ in $M$ as shown in Figure \ref{fig:link_K}.
\end{lemma}

\begin{figure}[h!]
\begin{center}
{\labellist \small \hair 0pt 
\pinlabel {$K_0$} [tl] at 102 98
\pinlabel {$K_1$} [t] at 24 77
\endlabellist}
\includegraphics[scale=0.5]{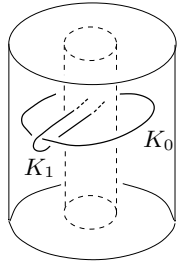}
\end{center}
\caption{The framed link $K_0 \cup K_1$ in the vicinity of the ``vertical'' boundary of $M$,
\ie the part $m\left(\partial \Sigma_{g,1} \times [-1,1]\right)$ of $\partial M$.}
\label{fig:link_K}
\end{figure}

\begin{proof}
We define the cobordism $\clocase{M}$ of $\Sigma_g$ 
by attaching a $2$-handle along the ``vertical'' boundary of $M$: 
$$
\clocase{M} := M \cup_{m|_{\partial \Sigma_{g,1} \times [-1,1]}} \left(D \times [-1,1]\right).
$$
The cobordism $\clocase{M'}$ is similarly defined from $M'$.
The $2$-handle $D \times [-1,1]$ can be regarded as a framed string-knot $X$
in the $3$-manifold $\clocase{M}$ and, similarly, we have a framed string-knot $X' \subset \clocase{M'}$.
Assume that there is a homeomorphism $f: \clocase{M'} \to \clocase{M}$. 
If the image $\tilde{X} := f(X')$ happens to be isotopic to $X$, 
then $f$ restricts to a homeomorphism $M' \to M$ and we are done.
Otherwise, we can apply the ``slam dunk'' move: 

\begin{equation}
\label{eq:slam_dunk}
\labellist \small \hair 2pt
\pinlabel {$K_0$} [l] at 217 260
\pinlabel {$K_1$} [l] at 243 429
\pinlabel {$X$} [l] at 257 10
\pinlabel {$\tilde{X}$} [l] at 809 15
\endlabellist
\includegraphics[scale=0.15]{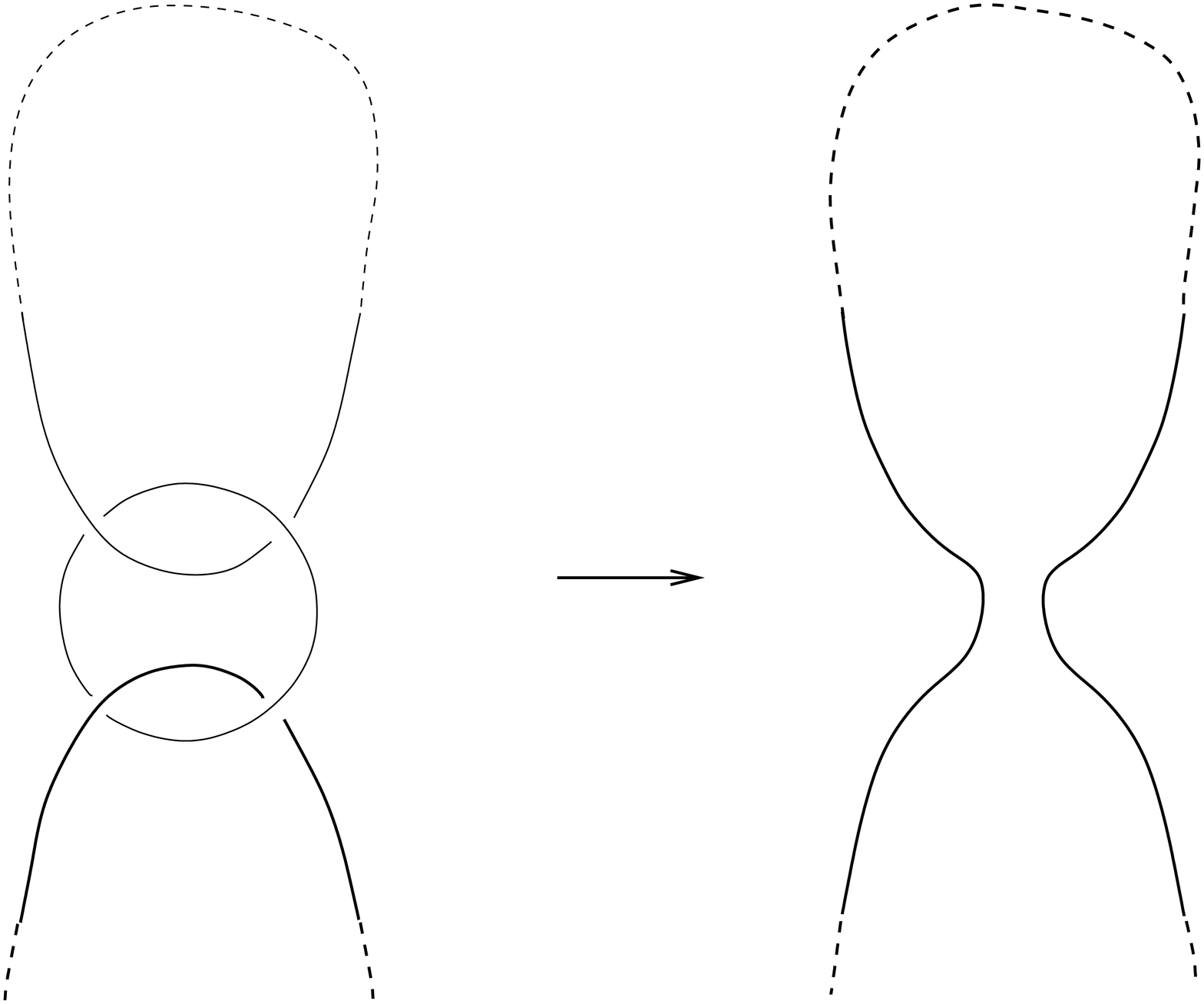}
\end{equation}

\noindent
Here, surgery is performed along the 2-component framed link $K_0 \cup K_1$;
it produces a homeomorphic manifold and the corresponding homeomorphism changes
the framed string-knot $X$ to $\tilde{X}$. Therefore, $M'$ is (up to homeomorphism)
the result of the surgery in $M$ along $K_0 \cup K_1$.
\end{proof}

\section{Johnson homomorphisms and Morita homomorphisms}

\label{sec:Johnson_Morita}

We review the Johnson homomorphisms and their extensions by Morita.
Originally defined for the mapping class group of a surface, those homomorphisms
have natural extensions to the monoid of homology cobordisms. 
In this section, we restrict ourselves to the bordered surface $\Sigma_{g,1}$ and to the closed surface $\Sigma_g$. 
We also refer to Morita's survey in Volume I of this handbook \cite{Morita_chapter}.

\subsection{The Dehn--Nielsen representation}

\label{subsec:Dehn-Nielsen}

The definitions of Johnson's and Morita's homomorphisms are based
on the action of the mapping class group on the fundamental group, which we shall first review.
We start with the case of the bordered surface $\Sigma_{g,1}$.
We denote
$$
\pi:=\pi_1(\Sigma_{g,1},\star)
$$
where $\star \in \partial \Sigma_{g,1}$ is a fixed base point.  
The homotopy class of the oriented boundary curve defines a special  element in the group $\pi$,
namely $\zeta := \left[\partial \Sigma_{g,1}\right]$.

\begin{definition}
The \emph{Dehn--Nielsen representation}\index{Dehn--Nielsen representation}\index{representation!Dehn--Nielsen}
is the group homomorphism
$$
\rho: \mcg_{g,1} \longrightarrow \Aut(\pi), \ f \longmapsto f_*.
$$
\end{definition}

\noindent
A classical result by Dehn, Nielsen and Baer asserts that $\rho$ is injective,
and that its image consists of the automorphisms that fix $\zeta$ \cite{ZVC}. 

There is no obvious way to extend the map $\rho$ to the monoid $\cob_{g,1}$.
Nonetheless, the ``nilpotent'' versions of $\rho$ can be extended from $\mcg_{g,1}$ to $\cob_{g,1}$.
For this, we consider the lower central series\index{lower central series} of $\pi$
$$
\Gamma_1 \pi \supset \Gamma_2\pi \supset
\cdots \supset \Gamma_k \pi \supset \Gamma_{k+1} \pi \supset \cdots
$$
which is inductively defined by $\Gamma_1 \pi := \pi$ and $\Gamma_{i+1} \pi := \left[\Gamma_i \pi, \pi\right]$. 
For each $k \geq 0$, we  define a monoid homomorphism 
$$
\rho_k: \cob_{g,1} \longrightarrow \Aut\left(\pi/\Gamma_{k+1} \pi\right), 
\ M \longmapsto  {m_{-,* }}^{-1} \circ m_{+,*}.
$$
Here the fundamental group of $M$ is based at $\star_0 := m(\star,0)$,
which is identified (homotopically) with $m_+(\star)=m(\star,+1)$ and to $m_-(\star)=m(\star,-1)$ 
by the ``vertical'' segments $m(\star \times [0,1])$ and $m(\star \times [-1,0])$. 
The map $m_{\pm,*}:\pi/\Gamma_{k+1} \pi \to \pi_{1}(M,\star_0)/\Gamma_{k+1}\pi_1(M,\star_0)$
is a group isomorphism according  to Stallings' theorem \cite{Stallings}. 
(The hypothesis that $m_{\pm}$ is an isomorphism in homology is needed here.)

\begin{example}[Degree $1$]
The abelian group $\pi/\Gamma_2 \pi$ is $H:= H_1\left(\Sigma_{g,1}\right)$.
The homomorphism $\rho_1: \cob_{g,1} \to \Aut(H)$ gives the homology type of homology cobordisms. 
So, the kernel of $\rho_1$ is the monoid of homology cylinders $\cyl_{g,1}$.
The image of $\rho_1$ is the subgroup $\Sp(H)$ of $\Aut(H)$ consisting
of the elements preserving 
the intersection pairing $\omega: H \times H \to \Z$.
\end{example}

The following filtration is introduced in \cite{Habiro,GL}.

\begin{definition} 
The \emph{Johnson filtration}\index{Johnson filtration}
of $\cob_{g,1}$ is the  sequence of submonoids
$$
\cob_{g,1} = \cob_{g,1}[0] \supset \cob_{g,1}[1] \supset  \cdots \supset \cob_{g,1}[k] \supset \cob_{g,1}[k+1] \supset \cdots
$$
where $\cob_{g,1}[k]$ is the kernel of  $\rho_k$ for all $k\geq 1$. 
\end{definition}

This filtration restricts to a decreasing sequence of subgroups of the mapping class group, 
which was introduced by Johnson in \cite{Johnson_survey} and studied by Morita in \cite{Morita_Abelian}. We denote it by
$$
\mcg_{g,1} = \mcg_{g,1}[0] \supset \mcg_{g,1}[1] \supset  \cdots \supset \mcg_{g,1}[k] \supset \mcg_{g,1}[k+1] \supset \cdots
$$
This series of subgroups is an \emph{$N$-series}, in the sense that
the commutator subgroup of $\mcg_{g,1}[k]$ and $\mcg_{g,1}[l]$ is contained in  $\mcg_{g,1}[k+l]$ for all $k,l\geq 0$.
It is also \emph{separating}, in the sense that its intersection $\bigcap_{k\geq 0} \mcg_{g,1}[k]$ is trivial.
This fact follows directly from the injectivity of the Dehn--Nielsen representation,
and from the fact that $\pi$ is residually nilpotent.
Note that the Johnson filtration is far from being separating in the case of homology cobordisms.
(For instance, every $\rho_k$ is trivial for $g=0$.
This fact, according to Example \ref{ex:genus_0_cobordisms}, contrasts with the richness of the monoid of homology $3$-spheres.)\\

Let us now consider the case of a closed surface $\Sigma_g$. We denote 
$$
\clocase{\pi} := \pi_1(\Sigma_g).
$$
We have implicitly chosen a base point $\star \in \Sigma_g$,
but we do not include it in the notation since the definitions below will not depend on it.

\begin{definition}
The \emph{Dehn--Nielsen representation}\index{Dehn--Nielsen representation}\index{representation!Dehn--Nielsen}
is the group homomorphism
$$
\rho: \mcg_{g} \longrightarrow \Out(\clocase{\pi} ), \ f \longmapsto f_*.
$$
\end{definition}

\noindent
In the closed case, the map $\rho$ is injective and
its image consists of classes of automorphisms that fix
$[\Sigma_g] \in H_2(\Sigma_g) \simeq H_2(\clocase{\pi})$ \cite{ZVC}.

Similarly to the bordered case, the ``nilpotent'' versions of $\rho$ can be extended to the monoid $\cob_g$.
Thus, we define for each $k \geq 0$ a monoid homomorphism 
$$
\rho_k: \cob_{g} \longrightarrow \Out\left(\clocase{\pi} /\Gamma_{k+1} \clocase{\pi} \right), 
\ M \longmapsto  {m_{-,* }}^{-1} \circ m_{+,*}.
$$
To define $\rho_k(M)$, we need to choose a base point in $M$ and 
some paths joining it to $m_+(\star)$ and to $m_-(\star)$ respectively,
but the resulting outer automorphism of $\clocase{\pi} /\Gamma_{k+1} \clocase{\pi} $ is independent of these choices.

\begin{definition}
The \emph{Johnson filtration}\index{Johnson filtration} of $\cob_{g}$ is the sequence of submonoids
$$
\cob_{g} = \cob_{g}[0] \supset \cob_{g}[1] \supset  \cdots \supset \cob_{g}[k] \supset \cob_{g}[k+1] \supset \cdots
$$
where $\cob_{g}[k]$ is the kernel of  $\rho_k$ for all $k\geq 1$. 
\end{definition}

This filtration restricts to a decreasing sequence of subgroups of the mapping class group
$$
\mcg_{g} = \mcg_{g}[0] \supset \mcg_{g}[1] \supset  \cdots \supset \mcg_{g}[k] \supset \mcg_{g}[k+1] \supset \cdots
$$
which, again, is an $N$-series.
By using the Dehn--Nielsen representation and more delicate arguments than in the bordered case \cite{BL},
one can prove that the Johnson filtration is also separating in the closed case (at least for $g\neq 2$).

\subsection{Johnson homomorphisms}

\label{subsec:Johnson}

We start by defining the Johnson homomorphisms in the bordered case.
For this, we fix a few identifications between groups, which will be implicit in the sequel. First of all, we denote
\begin{equation}
\label{eq:H_pi}
H:= H_1(\Sigma_{g,1}) \simeq \pi/\Gamma_2 \pi.
\end{equation}
The left-adjoint of the  intersection pairing $\omega$ of the surface $\Sigma_{g,1}$ 
defines an isomorphism between $H$ and its dual $H^* := \Hom(H,\Z)$:
$$
H \stackrel{\simeq}{\longrightarrow} H^*,\ h \longmapsto \omega(h,-).
$$
(In the literature, the right-adjoint of $\omega$ is also used to identify 
$H$ with $H^*$. This seems, for instance, to be the case in Johnson's and Morita's papers.)
Let $\Lie:=\Lie(H)$ denote the Lie ring freely generated by $H$: 
this is a graded Lie ring whose degree is given by the lengths of brackets.
Recall from \cite{Bourbaki,MKS} that there is a canonical isomorphism 
\begin{equation}
\label{eq:iso_free}
\Gr \pi = \bigoplus_{k\geq 1} \frac{\Gamma_k \pi}{\Gamma_{k+1} \pi} 
\stackrel{\simeq}{\longrightarrow} \bigoplus_{k\geq 1} \Lie_k = \Lie
\end{equation}
between the graded Lie ring associated with the lower central series of $\pi$ and $\Lie$. 
Both $\Gr \pi$ and $\Lie$ are generated by their degree $1$ parts,
and the isomorphism (\ref{eq:iso_free}) is simply given in degree $1$ by (\ref{eq:H_pi}).

\begin{definition}
For all $k\geq 1$, the \emph{$k$-th Johnson homomorphism}\index{Johnson homomorphism}\index{homomorphism!Johnson}
is the monoid map 
$$
\tau_k: \cob_{g,1}[k] \longrightarrow \Hom\left(H,\Gamma_{k+1} \pi / \Gamma_{k+2} \pi \right)
= H^* \otimes \Gamma_{k+1} \pi / \Gamma_{k+2} \pi \simeq H \otimes \Lie_{k+1}
$$
which sends an $M \in \cob_{g,1}[k]$ to the map defined by 
$\{x\} \mapsto \rho_{k+1}(M)(\{x\})\cdot \{x\}^{-1}$ for all $\{x\} \in \pi/\Gamma_2 \pi$.
\end{definition}

The Johnson homomorphisms are defined for homology cobordisms in \cite{GL}
as a natural generalization of the homomorphisms introduced in \cite{Johnson_survey,Morita_Abelian} for mapping class groups.
Recall that $\partial \Sigma_{g,1}$ defines an element $\zeta \in \pi$ and, clearly,
the automorphism $\rho_{k+1}(M)$ of $\pi/\Gamma_{k+2}\pi$ lifts 
to an endomorphism of $\pi/\Gamma_{k+3}\pi$ which fixes $\{\zeta\}$ (namely the automorphism $\rho_{k+2}(M)$).
It can be checked that such a ``boundary condition'' implies that $\tau_k(M)$ belongs to the kernel of the bracket map:
$$
\operatorname{D}_{k+2}(H) := 
\Ker\left([\centereddot,\centereddot]: H \otimes \Lie_{k+1}(H) \longrightarrow \Lie_{k+2}(H)\right).
$$
Thus, the $k$-th Johnson homomorphism is a monoid homomorphism
$$
\tau_k: \cob_{g,1}[k] \longrightarrow \operatorname{D}_{k+2}(H)
$$
whose kernel is the submonoid $\cob_{g,1}[k+1]$.

The first Johnson homomorphism $\tau_1$ on the Torelli group was introduced by Johnson himself in \cite{Johnson_first_homomorphism}
and is a key ingredient in the abelianization of the Torelli group \cite{Johnson_abelianization} --- see \S \ref{subsec:degree_one}. 
Its target can be identified with the third exterior power of $H$ by the map
\begin{equation}
\label{eq:trivectors}
\Lambda^3 H \stackrel{\simeq}{\longrightarrow} \operatorname{D}_{3}(H), \  
a \wedge b \wedge c \longmapsto a\otimes [b,c]+ c\otimes [a,b] + b \otimes [c,a].
\end{equation}
Thus, the first Johnson homomorphism is a monoid homomorphism
$$
\tau_1: \cyl_{g,1} \longrightarrow \Lambda^3 H.
$$
The resulting representation of the Torelli group $\Torelli_{g,1}$ 
already appears in the background of a paper by Sullivan \cite{Sullivan} 
in connection with the cohomology rings of closed oriented $3$-manifolds.
To understand this relation,  
recall from \S \ref{subsec:closure} that an $M \in \cyl_{g,1}$ can be ``closed'' 
to a $3$-manifold $\overline{M}$ without boundary. Then, the inclusion $m_{\pm}: \Sigma_{g,1} \to M \subset \overline{M}$
induces an isomorphism in homology so that $H$ can be identified with $H_1(\overline{M})$.

\begin{proposition}[Johnson \cite{Johnson_survey}]
\label{prop:triple-cup}
For all $M \in \cyl_{g,1}$, $\tau_1(M)$ coincides
with the triple-cup product form of the closure $\overline{M}$, namely the form
$$
H^1(\overline{M}) \times H^1(\overline{M}) \times H^1(\overline{M}) \longrightarrow \Z, 
\ (x,y,z) \longmapsto \left\langle x\cup y \cup z, \left[\overline{M}\right] \right\rangle
$$
which we regard as an element of 
$\Hom\left(\Lambda^3 H^1(\overline{M}), \Z\right) \simeq \Lambda^3 H_1(\overline{M}) \simeq \Lambda^3 H$.
\end{proposition}

\noindent
As claimed by Johnson \cite{Johnson_survey}, 
this relationship can be generalized to the higher Johnson homomorphisms. 
Thus, the $k$-th Johnson homomorphism  of an $M \in \cob_{g,1}[k]$
can be computed from the length $k+1$ Massey products of the closure $\overline{M}$.
This has been proved by Kitano in the case of the mapping class group \cite{Kitano}
and by Garoufalidis--Levine in the case of homology cobordisms \cite{GL}.

Johnson proved that  $\tau_1: \Torelli_{g,1} \to \Lambda^3 H$ is surjective (for $g\geq 2$)
by computing its values on some explicit elements of the Torelli group \cite{Johnson_first_homomorphism}.
In general, the image of $\tau_k: \mcg_{g,1}[k] \to \operatorname{D}_{k+2}(H)$ 
is not known (except for small values of $k$) --- see \cite{Morita_prospect} for an account of this important problem.
Nevertheless, we have the following result for homology cobordisms.

\begin{theorem}[Garoufalidis--Levine \cite{GL}]
\label{th:surjectivity_Johnson_homomorphisms}
For all $k\geq 1$, the monoid map $\tau_{k}: \cob_{g,1}[k] \to \operatorname{D}_{k+2}(H)$ is surjective.
\end{theorem}

\begin{proof}[Sketches of the proof]
Let $\Aut_\zeta\left(\pi/\Gamma_{k+1} \pi\right)$ 
be the group of those automorphisms $\Psi$ of $\pi/\Gamma_{k+1} \pi$ 
which lift to an endomorphism $\widetilde{\Psi}$ of $\pi/\Gamma_{k+2} \pi$ 
such that $\widetilde{\Psi}(\{\zeta\})= \{\zeta\}$.  
By using cobordism theory and surgery techniques, Garoufalidis and Levine first prove that the monoid homomorphism
$$
\rho_k: \cob_{g,1} \longrightarrow \Aut_\zeta\left(\pi/\Gamma_{k+1} \pi\right)
$$
is surjective \cite[Theorem 3]{GL}. 
Similar techniques are used by Turaev in \cite{Turaev_nilpotent} to prove a realization result
for nilpotent homotopy types of closed oriented $3$-manifolds --- see \S \ref{subsec:Morita} in this connection.
Next, they consider the following commutative diagram:
\begin{equation}
\label{eq:D_Aut_Aut}
\xymatrix{
1 \ar[r] & \cob_{g,1}[k] \ar[r] \ar[d]_-{\tau_k} & \cob_{g,1} \ar[r]^-{\rho_k} \ar[d]^-{\rho_{k+1}}  & 
\Aut_\zeta\left(\pi/\Gamma_{k+1} \pi\right) \ar[r] \ar@{=}[d]& 1  \\
1 \ar[r] & \operatorname{D}_{k+2}(H) \ar[r] &  \Aut_\zeta\left(\pi/\Gamma_{k+2} \pi\right) \ar[r] &
\Aut_\zeta\left(\pi/\Gamma_{k+1} \pi\right)  \ar[r] &1.
}
\end{equation}
Clearly, the first row is exact.
In the second row, the right-hand side map is induced by the canonical projection 
$\End(\pi/\Gamma_{k+2} \pi) \to \End(\pi/\Gamma_{k+1} \pi)$, 
and the surjectivity of $\rho_k$ implies that it is surjective. 
The left-hand side map in the second row sends any 
$f \in \operatorname{D}_{k+2}(H) \subset H \otimes \Lie_{k+1} 
\simeq \Hom(H, \Gamma_{k+1} \pi / \Gamma_{k+2} \pi)$ to the automorphism $\Psi$ of 
$\pi/\Gamma_{k+2} \pi$ defined by $\Psi(\{x\}) = f(\{x\})\cdot \{x\}$ for all $x\in \pi$.
It can be checked that the second row is exact \cite[Proposition 2.5]{GL},
and the surjectivity of $\rho_{k+1}$ then implies that $\tau_k$ is surjective.

Habegger gives in \cite{Habegger} a different proof of Theorem \ref{th:surjectivity_Johnson_homomorphisms}.
For this, he defines a correspondence between the set $\cob_{g,1}[1]=\cyl(\Sigma_{g,1})$ 
and the set $\cyl(\Sigma_{0,2g+1})$ of framed string-links in homology $3$-balls on $2g$ strands with trivial linking matrix.
(See Example \ref{ex:genus_0_cobordisms}.)
Although not multiplicative, this bijection makes the Johnson filtration correspond to the Milnor filtration,
and the Johnson homomorphisms correspond to the Milnor invariants. 
Thus, Habegger calls it the ``Milnor--Johnson correspondence'' and the surjectivity of $\tau_k$ 
then follows from the surjectivity of the Milnor invariants for string-links in homology $3$-balls \cite{HL_concordance}.
\end{proof}

In the closed case, the Johnson homomorphisms are defined for the mapping class group by Morita in \cite{Morita_linear}.
The definition is similar to the bordered case, but it is also more technical.
Since any embedding $\Sigma_{g,1} \to \Sigma_g$ induces an isomorphism in homology, 
we use the same notation as in the bordered case:
$$
H := H_1(\Sigma_g) \simeq \clocase{\pi}/\Gamma_2  \clocase{\pi}.
$$
The intersection pairing $\omega: H \times H \to \Z$ defines an isomorphism
$H \simeq H^*$ and, so, is dual to a bivector $\omega \in \Lambda^2 H$.
By a theorem of Labute \cite{Labute}, the graded Lie ring 
associated with the lower central series of $\clocase{\pi}$ is canonically isomorphic to the quotient
$\clocase{\Lie} := \Lie  /  \langle \omega \rangle_{\operatorname{ideal}}$,
where $\omega$ is regarded as an element of $\Lie_2$.

Each $\{y\} \in \Gamma_k \clocase{\pi}/ \Gamma_{k+1} \clocase{\pi}$
induces  a group homomorphism $H \to \Gamma_{k+1} \clocase{\pi}/ \Gamma_{k+2} \clocase{\pi}$ 
defined by $\{x\} \mapsto \{[y,x]\}$. 
In that way, $\Gamma_k \clocase{\pi}/ \Gamma_{k+1} \clocase{\pi}$ is mapped to a subgroup
of $\Hom\left(H,\Gamma_{k+1} \clocase{\pi} / \Gamma_{k+2} \clocase{\pi} \right)$
denoted by $\left[ \Gamma_k \clocase{\pi}/ \Gamma_{k+1} \clocase{\pi}, \centereddot\right]$.
Besides, each element $y \in \clocase{\Lie}_k$ defines by antisymmetrization 
of $\omega \in \Lambda^2 H $ in $H \otimes H$ 
an element $\omega \otimes y$ of $H \otimes H \otimes  \clocase{\Lie}_k$
and, so, by applying the bracket $H \otimes  \clocase{\Lie}_k \to \clocase{\Lie}_{k+1}$,
we get an element $\left(\Id \otimes [\centereddot,\centereddot]\right)(\omega \otimes y)$ 
of $H \otimes  \clocase{\Lie}_{k+1}$. 
Thus, $\clocase{\Lie}_k$ is mapped to a subgroup of $H \otimes  \clocase{\Lie}_{k+1}$
denoted by $\left(\Id \otimes [\centereddot,\centereddot]\right)(\omega \otimes \clocase{\Lie}_k)$.
It is easily checked that the canonical isomorphism
$$
\Hom\left(H,\Gamma_{k+1} \clocase{\pi} / \Gamma_{k+2} \clocase{\pi} \right)
= H^* \otimes \Gamma_{k+1} \clocase{\pi} / \Gamma_{k+2} \clocase{\pi}
\simeq H \otimes  \clocase{\Lie}_{k+1}
$$
makes those two subgroups correspond.

\begin{definition}
For all $k\geq 1$, the \emph{$k$-th Johnson homomorphism}\index{Johnson homomorphism}\index{homomorphism!Johnson}  
is the monoid map 
$$
\tau_k: \cob_{g}[k] \longrightarrow 
\frac{\Hom\left(H,\Gamma_{k+1} \clocase{\pi} / \Gamma_{k+2} \clocase{\pi} \right)}
{\left[ \Gamma_k \clocase{\pi}/ \Gamma_{k+1} \clocase{\pi}, \centereddot\right]}
\simeq \frac{H \otimes \clocase{\Lie}_{k+1}}{\left(\Id \otimes [\centereddot,\centereddot]\right)(\omega \otimes \clocase{\Lie}_k)}
$$
which sends an $M \in \cob_{g}[k]$ to the map defined by 
$\{x\} \mapsto \rho_{k+1}(M)(\{x\})\cdot \{x\}^{-1}$ for all $\{x\} \in \clocase{\pi}/\Gamma_2 \clocase{\pi}$.
\end{definition}

\noindent
Recall that, in the closed case, $\rho_{k+1}(M)$ is an automorphism of $\clocase{\pi}/\Gamma_{k+2} \clocase{\pi}$
which is only defined up to some inner automorphism. This ambiguity disappears by taking
the quotient with $\left[ \Gamma_k \clocase{\pi}/ \Gamma_{k+1} \clocase{\pi}, \centereddot\right]$.
For the same reason, the choice of the base point  $\star\in \Sigma_g$ has no incidence on the definition of $\tau_k(M)$.

As explained in \S \ref{subsec:bordered_to_closed}, we may think of the closed surface $\Sigma_g$ as the union of $\Sigma_{g,1}$ with a  disk $D$. 
We can also choose a system of meridians and parallels $(\alpha,\beta)$ on $\Sigma_{g,1}$ as shown in Figure \ref{fig:surface}. 
Then, for a base point $\star \in \partial \Sigma_{g,1}$ and some appropriate basings of those meridians and parallels, we have
$$
\zeta = \left[\partial \Sigma_{g,1}\right]
=  \left[ \alpha_{g},\beta_{g}^{-1}\right]\cdots   \left[ \alpha_{1},\beta_{1}^{-1}\right] \ \in \pi.
$$
Since $\zeta=[\partial D]$ is trivial in $\clocase{\pi}$,
any representative $r \in \Aut\left(\clocase{\pi}/\Gamma_{k+3} \clocase{\pi}\right)$ of $\rho_{k+1}(M)$
satisfies the identity 
$$
\left[ r(\alpha_{g}),r(\beta_{g}^{-1})\right] \cdots \left[ r(\alpha_{1}),r(\beta_{1}^{-1})\right] =1 \ \in \clocase{\pi}/\Gamma_{k+3} \clocase{\pi}.
$$
This implies that a representative of $\tau_k(M)$ in $H \otimes \clocase{\Lie}_{k+1}$ actually belongs to the subgroup
$$
\clocase{\operatorname{D}}_{k+2}(H)
:= \Ker\left([\centereddot,\centereddot]: H \otimes \clocase{\Lie}_{k+1} \longrightarrow \clocase{\Lie}_{k+2}\right).
$$
Thus, the $k$-th Johnson homomorphism is a monoid homomorphism
$$
\tau_k: \cob_{g}[k] \longrightarrow 
\frac{\clocase{\operatorname{D}}_{k+2}(H)}{\left(\Id \otimes [\centereddot,\centereddot]\right)(\omega \otimes \clocase{\Lie}_k)}
$$
whose kernel is the submonoid $\cob_{g}[k+1]$.

\begin{figure}[h]
\begin{center}
{\labellist \small \hair 0pt 
\pinlabel {\large $\star$} at 567 3
\pinlabel {$\alpha_1$} [l] at 310 105
\pinlabel {$\alpha_g$} [l] at 568 65
\pinlabel {$\beta_1$} [b] at 216 80
\pinlabel {$\beta_g$} [b] at 472 69
\endlabellist}
\includegraphics[scale=0.52]{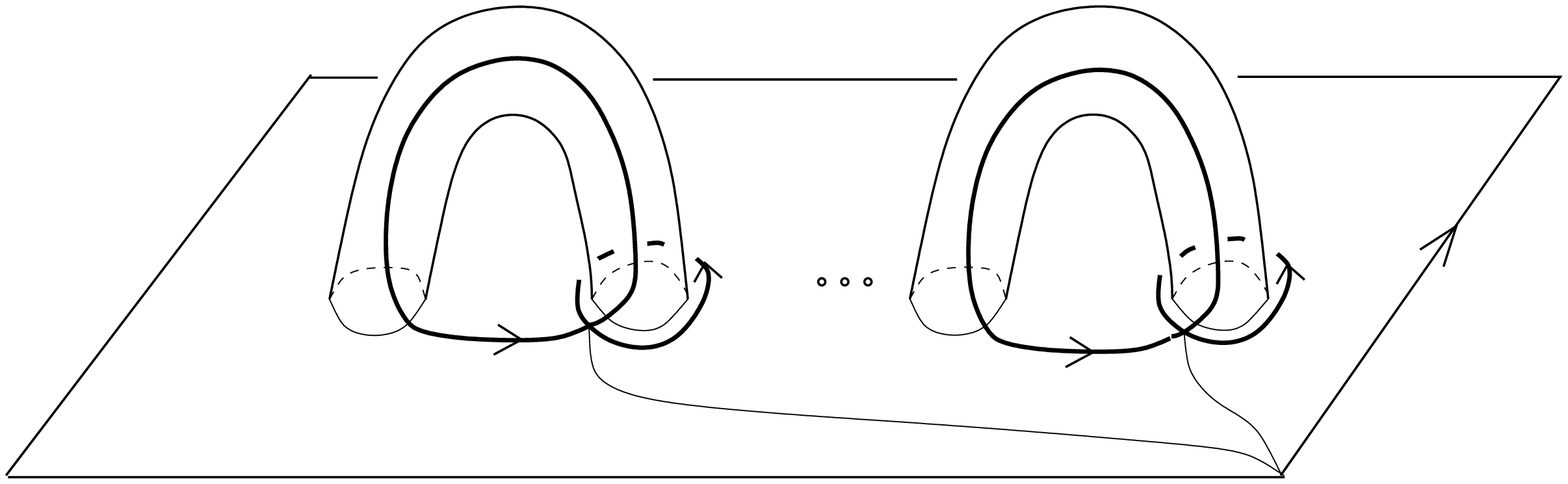}
\end{center}
\caption{The surface $\Sigma_{g,1}$ and a system of meridians and parallels $(\alpha,\beta)$.}
\label{fig:surface}
\end{figure}

Similarly to the bordered case, $\tau_1$ was introduced by Johnson himself \cite{Johnson_first_homomorphism},
and it plays a key role in the abelianization of $\Torelli_g$ --- see \S \ref{subsec:degree_one}.
Note that  we have an isomorphism 
$$
 \Lambda^3 H \stackrel{\simeq}{\longrightarrow} \operatorname{D}_3(H) \simeq \clocase{\operatorname{D}}_3(H)
 $$
through which the subgroup $\omega \wedge H$ goes to $\left(\Id \otimes [\centereddot,\centereddot]\right)(\omega \otimes H)$.
Thus, in the closed case, the first Johnson homomorphism is a monoid map
$$
\tau_1: \cyl_{g} \longrightarrow \frac{\Lambda^3 H}{\omega \wedge H}.
$$
As in the bordered case, $\tau_1$ can be interpreted in terms of the cohomology rings of closed oriented $3$-manifolds. 
Indeed, let $\overline{M}$ be the closure of an $M \in \cyl_g$ as defined in \S \ref{subsec:closure}.
We have a short exact sequence of abelian groups
\begin{equation}
\label{eq:ses_homology}
\xymatrix{
0 \ar[r] & H \ar[r]^-{m_*} &  H_1(\overline{M}) \ar[r]^-p & \Z \ar[r] & 0 
}
\end{equation}
where $m_*$ is induced by the inclusion $m_\pm: \Sigma_g \to M \subset \overline{M}$
and $p$ takes the homological intersection with $\left[m_\pm(\Sigma_g)\right] \in H_2(\overline{M})$.

\begin{proposition}
\label{prop:triple-cup_closed}
For all $M \in \cyl_{g}$, $\tau_1(M)$ is the image
of the triple-cup product form of the closure $\overline{M}$ under the projection
$$
\xymatrix{
\Lambda^3 H_1(\overline{M}) \ar@{->>}[r]^-{\Lambda^3 s\ } &
\Lambda^3 H \ar@{->>}[r] & \Lambda^3 H/ \omega \wedge H
}
$$
where $s$ is a left inverse of $m_*$ in (\ref{eq:ses_homology}).
\end{proposition}

The map $\clocase{\ \centereddot\ }:\cob_{g,1} \to \cob_g$ defined in Lemma \ref{lem:closing} 
(when $\Sigma_g$ is decomposed into a closed disk and a copy of $\Sigma_{g,1}$) 
sends the Johnson filtration to the Johnson filtration. Moreover, we have a commutative square:
\begin{equation}
\label{eq:Johnson_bordered_to_closed}
\xymatrix{
\cob_{g,1}[k] \ar[r]^-{\tau_k} \ar[d]_-{\clocase{\ \centereddot\ }} & \operatorname{D}_{k+2}(H) \ar@{->>}[d] \\
\cob_{g}[k] \ar[r]_-{\tau_k} &
\frac{\clocase{\operatorname{D}}_{k+2}(H)}{\left(\Id \otimes [\centereddot,\centereddot]\right)(\omega \otimes \clocase{\Lie}_k)}
}
\end{equation}
We deduce from Theorem \ref{th:surjectivity_Johnson_homomorphisms} that $\tau_k$ is surjective also in the closed case.
However, computing the image of $\mcg_g[k]$ by $\tau_k$ is still an unsolved and hard problem \cite{Morita_prospect}.
Besides, (\ref{eq:Johnson_bordered_to_closed}) is also useful to deduce  Proposition \ref{prop:triple-cup_closed} from Proposition \ref{prop:triple-cup}.

\subsection{Morita homomorphisms}

\label{subsec:Morita}

In the case of the bordered surface $\Sigma_{g,1}$,
Morita introduced in \cite{Morita_Abelian} some refinements of the Johnson homomorphisms. 
From the point of view of $3$-dimensional topology, 
Morita's homomorphisms correspond to Turaev's nilpotent homotopy types of closed $3$-manifolds \cite{Turaev_nilpotent},
which we shall now recall.

Let $k \geq 1$ be an integer and let $G$ be a group of nilpotency class $k$ (\ie the subgroup $\Gamma_{k+1} G$ is trivial).
Let $N$ be a closed connected oriented $3$-manifold 
whose $k$-th nilpotent quotient of the fundamental group is parameterized by the group $G$:
$$
n: G \stackrel{\simeq}{\longrightarrow} \pi_1(N) / \Gamma_{k+1} \pi_1(N).
$$
Then, Turaev defines the \emph{$k$-th nilpotent homotopy type} of the pair $(N,n)$ to be 
$$
\mu_k(N,n) := f_{n,*}\left([N]\right) \in H_3(G)
$$
where $f_n: N \to \operatorname{K}(G,1)$ gives the composition
$\pi_1(N) \to \pi_1(N)/\Gamma_{k+1} \pi_1(N) \stackrel{n^{-1}}{\to} G$ at the level of fundamental groups.
This terminology is justified by the fact due to Thomas \cite{Thomas} and Swarup \cite{Swarup} that
the oriented homotopy type of $N$ is determined by the group $\pi_1(N)$ and the image of $[N]$ in $H_3\left(\pi_1(N)\right)$.

\begin{remark}
\label{rem:cobordism}
It is well-known that the cobordism group $\Omega_3^{\operatorname{SO}}( \operatorname{K}(G,1))$ 
is isomorphic to $H_3(G)$ by the map $\{(L,l)\} \mapsto l_*([L])$,
for any closed connected oriented $3$-manifold $L$ and any map $l:L \to  \operatorname{K}(G,1)$.
Thus, an equivalent definition for the $k$-th nilpotent homotopy type is
$\mu_k(N,n) := \left\{(N,f_n)\right\} \ \in \Omega_3^{\operatorname{SO}}( \operatorname{K}(G,1))$.
\end{remark}

Let us now consider the closure $\overline{C}$ of a cobordism $C \in \cob_{g,1}[k]$ defined in \S \ref{subsec:closure}. 
Observe that the map $\Sigma_{g,1} \to \overline{C}$, which is
obtained by composing $c_\pm: \Sigma_{g,1} \to  C$ with the inclusion $C \subset \overline{C}$,
induces an isomorphism $c_{*}: \pi / \Gamma_{k+1} \pi \to  \pi_1(\overline{C})/ \Gamma_{k+1}  \pi_1(\overline{C})$.

\begin{definition}
For all $k\geq 1$, the {\it $k$-th Morita homomorphism}\index{Morita homomorphism}\index{homomorphism!Morita} 
is the map
$$
M_k: \cob_{g,1}[k] \longrightarrow H_3\left( \pi / \Gamma_{k+1} \pi \right)
$$
that sends a $C \in \cob_{g,1}[k]$ to $\mu_k\left(\overline{C},c_{*}\right)$.
\end{definition}

\noindent
This map is studied by Heap in \cite{Heap} at the mapping class group level,
and by Sakasai in \cite{Sakasai} for homology cobordisms. 
It also appears implicitly in \cite{GL}. By considering the simplicial model of $\operatorname{K}\left(\pi/\Gamma_{k+1} \pi,1\right)$,
Heap shows that $M_k$ is equivalent to Morita's refinement of the $k$-th Johnson homomorphism  \cite{Morita_Abelian}.
Morita's original definition is more algebraic: for any $f \in \mcg_{g,1}[k]$, 
the homology class $M_k(\mcyl(f))$ is defined in \cite{Morita_Abelian}
from the Dehn--Nielsen representation $\rho(f) \in \Aut(\pi)$ using the bar resolution of the group $\pi$.
Similarly, one can define $M_k(C)$ from $\rho_{2k}(C) \in \Aut(\pi/\Gamma_{2k+1} \pi)$
using the fact  (proved by Igusa and Orr \cite{IO}) that the canonical homomorphism $H_3(\pi/\Gamma_{2k+1}\pi) \to H_3(\pi/\Gamma_{k+1}\pi)$ is trivial.

However, it is not clear from the above definition that $M_k$ is a monoid homomorphism.
The additivity of $M_k$ is a consequence of the following ``variation formula''
for the $k$-th nilpotent homotopy type.

\begin{proposition}
\label{prop:variation_Morita}
Let $N$ be a closed connected oriented $3$-manifold.
Given a cobordism $C \in \cob_{g,1}[k]$ and an embedding $j:\Sigma_{g,1} \to N$,
consider the $3$-manifold $N'$ obtained by ``cutting'' $N$ along $j(\Sigma_{g,1})$ and by ``inserting'' $C$. 
Then, any isomorphism $n: G \to \pi_1(N) / \Gamma_{k+1} \pi_1(N)$ induces, in a canonical way,
 an isomorphism $n': G \to \pi_1(N') / \Gamma_{k+1} \pi_1(N')$ such that
\begin{equation}
\label{eq:variation}
\mu_k(N',n') - \mu_k(N, n) =  n_*^{-1} j_{*}\left(M_k(C)\right) \ \in H_3(G).
\end{equation}
\end{proposition}

\noindent
Similar formulas are shown in \cite[Theorem 2]{GL} and \cite[Theorem 5.2]{Heap} by cobordism arguments --- 
see Remark \ref{rem:cobordism}.

\begin{proof}[Proof of Proposition \ref{prop:variation_Morita}]
Denote by  $J$ the surface $j(\Sigma_{g,1})$ in $N$. The $3$-manifold $N'$ is defined by
$
N' := \left(N \setminus \interior\left(J \times [-1,1]\right)\right) \cup_{j' \circ c^{-1}} C
$
where $J \times [-1,1]$ denotes a closed regular neighborhood of $J$ in $N$
and $j'$ is the restriction to the boundary of the homeomorphism 
$j \times \Id: \Sigma_{g,1} \times [-1,1] \to J \times [-1,1]$.
The van Kampen theorem shows the existence of a canonical isomorphism
between $\pi_1(N)/\Gamma_{k+1} \pi_1(N)$ and $\pi_1(N')/\Gamma_{k+1} \pi_1(N')$ 
which is defined by the following commutative diagram:
$$
\xymatrix{
&{\frac{\pi_1\left(N \setminus \interior( J \times [-1,1])\right)}
{\Gamma_{k+1}\pi_1\left(N \setminus \interior(J \times [-1,1])\right)}} 
\ar@{->>}[ld] \ar@{->>}[rd] & \\
{\frac{\pi_1(N)}{\Gamma_{k+1} \pi_1(N)}} \ar@{-->}[rr]_-\simeq^-{\exists !}
& & {\frac{\pi_1(N')}{\Gamma_{k+1} \pi_1(N')}}.
}
$$
By composing it with $n$, we obtain an isomorphism $n': G \to \pi_1(N') / \Gamma_{k+1} \pi_1(N')$.

Let us now prove formula (\ref{eq:variation}).
The closed $3$-manifolds $N$, $N'$ and $\overline{C}$ can be seen all together inside the following singular $3$-manifold:
$$
\widetilde{N} :=  \big(N \setminus \interior\left(J \times [-1,1]\right)\big) \cup_{j' \sqcup j' \circ c^{-1}} 
\left( \Sigma_{g,1} \times [-1,1] \bigsqcup C\right).
$$
Denote the corresponding inclusions by $i:N \to \widetilde{N}$,
$i':N' \to \widetilde{N}$ and $\ell:\overline{C} \to \widetilde{N}$.
The homomorphism $i_*:\pi_1(N)/\Gamma_{k+1} \pi_1(N) \to \pi_1(\widetilde{N})/\Gamma_{k+1} \pi_1(\widetilde{N})$
is an isomorphism, and we denote by  $\widetilde{n}$ its pre-composition with $n$. 
Let $\widetilde{f}: \widetilde{N} \to \operatorname{K}(G,1)$ be a map which 
induces the composition
$\pi_1(\widetilde{N}) \to \pi_1(\widetilde{N})/\Gamma_{k+1} \pi_1(\widetilde{N}) 
\stackrel{\widetilde{n}^{-1}}{\to} G$ at the level of fundamental groups.
The restrictions of $\widetilde{f}$ to $N$ and $N'$ are denoted by $f$ and $f'$ respectively,
and they induce $n^{-1}$ and $(n')^{-1}$ at the level of fundamental groups. So, we have
$$
\mu_k(N',n') - \mu_k(N, n)  =  f'_*([N']) - f_*([N])
=  \widetilde{f}_*\left(i'_*([N'])-i_*([N])\right)
= \widetilde{f}_* \ell_*\left(\left[\overline{C}\right]\right).
$$
But, we have the following commutative square in the category of groups:
$$
\xymatrix{
\pi_1\left(\overline{C}\right) \ar@{->>}[r] \ar[dd]_-{\ell_*} & 
\frac{\pi_1(\overline{C})}{\Gamma_{k+1} \pi_1(\overline{C})} \ar[rr]^-{c_*^{-1}}_-\simeq \ar@{-->}[d]_-{\ell_*} & &
\frac{\pi}{\Gamma_{k+1} \pi} \ar[d]^-{j_*} \ar@{=}[r] & 
\pi_1\left(\operatorname{K}(\pi/\Gamma_{k+1} \pi,1)\right)\\
& \frac{\pi_1(\widetilde{N})}{\Gamma_{k+1} \pi_1(\widetilde{N})}  \ar@{-->}[rrd]^-{\widetilde{n}^{-1}}_-\simeq& &
\frac{\pi_1(N)}{\Gamma_{k+1} \pi_1(N)} \ar[d]^-{n^{-1}}_-\simeq \ar@{-->}[ll]_-\simeq^-{i_*}  \\
\pi_1\left(\widetilde{N}\right) \ar[rrr]_-{\widetilde{f}_*} \ar@{-->>}[ur]
& & & G \ar@{=}[r] &  \pi_1(\operatorname{K}(G,1))
}
$$
(The dashed arrows are shown here to check the commutativity.)
Choose some maps $a: \overline{C} \to \operatorname{K}(\pi/\Gamma_{k+1} \pi,1)$
and $b:\operatorname{K}(\pi/\Gamma_{k+1} \pi,1) \to \operatorname{K}(G,1)$ that induce 
the top homomorphism and the right-hand side homomorphism, respectively, at the level of fundamental groups.
Thus,  $b \circ a$ is homotopic to $\widetilde{f} \circ \ell$ so that
$$
\mu_k(N',n') - \mu_k(N, n)  
=  (\widetilde{f}\circ \ell)_*\left(\left[\overline{C}\right]\right)
=  (b \circ a)_*\left(\left[\overline{C}\right]\right)\\
=  n^{-1}_* j_* \left(  \mu_k\left(\overline{C}, c_*\right)\right).
$$
We conclude from the definition of $M_k$.
\end{proof}

The $k$-th Morita homomorphism is a refinement of the $k$-th Johnson homomorphism in the following sense.
Consider the (central) extension of groups
$$
\xymatrix{
0 \ar[r] & \Lie_{k+1} \ar[r] & \pi/\Gamma_{k+2} \pi \ar[r] & \pi/\Gamma_{k+1} \pi \ar[r] &1
}
$$
and the corresponding Hochschild--Serre spectral sequence:
$$
E^2_{p,q} = H_p\left(\pi/\Gamma_{k+1} \pi; H_q\left(\Lie_{k+1}\right) \right)
\Longrightarrow H_{p+q}\left(\pi/\Gamma_{k+2} \pi\right).
$$
We need the differential $d^2 := d^2_{3,0}: E^2_{3,0} \to E^2_{1,1}$,
where $E^2_{3,0}\simeq H_3\left(\pi/\Gamma_{k+1} \pi\right)$ 
and $E^2_{1,1} \simeq H\otimes \Lie_{k+1}$.
Morita uses this spectral sequence in \cite{Morita_Abelian} to prove
that the image of $d^2$ is precisely the subgroup $\operatorname{D}_{k+2}(H)$.

\begin{theorem}[Morita \cite{Morita_Abelian}]
\label{th:Morita_to_Johnson}
For all integer $k\geq 1$, we have the following commutative triangle:
$$
\xymatrix{
\cob_{g,1}[k] \ar[r]^-{M_k} \ar[rd]_-{-\tau_k} & H_3\left(\pi/\Gamma_{k+1} \pi\right) \ar[d]^-{d^2} \\
& \operatorname{D}_{k+2}(H).
}
$$
\end{theorem}

\noindent
This result is stated and proved in \cite{Morita_Abelian} for mapping class groups, but the proof can be adapted to homology cobordisms.
From a $3$-dimensional viewpoint, this statement translates 
the fact that the $k$-th nilpotent homotopy type of a closed connected oriented $3$-manifold 
determines its Massey products of length $k+1$.
See \cite{Heap} for a topological proof of  Theorem \ref{th:Morita_to_Johnson}. Similar arguments appear in \cite{GL}.

\begin{example}[Degree $1$]
The Pontrjagin product for the homology of the abelian group $H$ defines  an isomorphism:
$$
\Lambda^3 H \stackrel{\simeq}{\longrightarrow}   H_3(H)  \simeq H_3(\pi/\Gamma_2 \pi).
$$
Through that isomorphism, the differential $d^2: H_3(\pi/\Gamma_2 \pi) \to \operatorname{D}_{3}(H)$ 
coincides with the isomorphism (\ref{eq:trivectors}). Thus, $M_1$ is equivalent to $-\tau_1$.
\end{example}

For all integers $l\geq k \geq 1$, the following square is commutative:
$$
\newdir{ >}{{}*!/-5pt/\dir{>}}
\xymatrix{
{\cob_{g,1}[l]} \ar@{ >->}[d] \ar[r]^-{M_l} & H_3\left(\pi/\Gamma_{l+1} \pi\right) \ar[d] \\
\cob_{g,1}[k]  \ar[r]_-{M_k} & H_3\left(\pi/\Gamma_{k+1} \pi\right).
}
$$
The right-hand side map is induced by the canonical map $\pi/\Gamma_{l+1} \pi \to \pi/\Gamma_{k+1} \pi$ and, as already mentioned,
Igusa and Orr proved  it  to be trivial for $l \geq 2k$ \cite{IO}. So, $M_k$ vanishes on $\cob_{g,1}[2k]$.
In fact, the computation of $H_3\left(\pi/\Gamma_{k+1} \pi\right)$ performed  in \cite{IO}
and the surjectivity of the Johnson homomorphisms (Theorem \ref{th:surjectivity_Johnson_homomorphisms}) 
can be used to prove the following stronger statement.

\begin{theorem}[Heap \cite{Heap}, Sakasai \cite{Sakasai}]
\label{th:kernel_Morita}
For all integer $k \geq 1$, we have the following short exact sequence of monoids:
$$
\xymatrix{
1 \ar[r] & \cob_{g,1}[2k] \ar[r] & \cob_{g,1}[k]  \ar[r]^-{M_k} & H_3\left(\pi/\Gamma_{k+1} \pi\right) \ar[r]  & 1.
}
$$
\end{theorem}

\noindent
More generally, two cobordisms $C,C' \in \cob_{g,1}[k]$ satisfy $M_k(C)=M_{k}(C')$ if and only if we have $\rho_{2k}(C)= \rho_{2k}(C')$.
(This can be deduced from  Theorem \ref{th:kernel_Morita} 
by considering group quotients of the monoid $\cob$: see \S \ref{subsec:Y_cylinders} and \S \ref{subsec:definition_group}.)

\subsection{Infinitesimal versions of the Dehn--Nielsen representation}

\label{subsec:total_Johnson}

We now  give  ``infinitesimal'' versions of the Dehn--Nielsen representation,
which are defined on the monoid of homology cobordisms and contain all
the Johnson homomorphisms.
Here, the word ``infinitesimal'' means that we are going to replace fundamental groups by their Malcev Lie algebras.
The passage from groups to Lie algebras will be useful in \S \ref{subsec:LMO_tree}
to connect the Dehn--Nielsen representation to finite-type invariants.

Let $G$ be a group and let $\Q[G]$ be the group algebra of $G$, with augmentation ideal $I$.
The $I$-adic completion of $\Q[G]$
$$
\widehat{\Q[G]} := \varprojlim_{k} \Q[G]/ I^k
$$
is a  Hopf algebra equipped with a complete filtration. If $G$ is residually torsion-free nilpotent 
(which will always be the case in our situation), then the canonical map $G \to \widehat{\Q[G]}$ is injective, 
so that we can write $G \subset \widehat{\Q[G]}$.
Following Quillen \cite{Quillen} and Jennings \cite{Jennings},
we define the \emph{Malcev Lie algebra}\index{Malcev Lie algebra}\index{Lie algebra!Malcev}
of $G$ as the primitive part of $\widehat{\Q[G]}$:
$$
\MalcevLie(G) := \Prim\big(\widehat{\Q[G]}\big).
$$
This is a filtered Lie algebra into which $G$ embeds by the logarithmic series:
$$
\log: G \longrightarrow \MalcevLie(G), \ 
x \longmapsto \sum_{n \geq 1} \frac{(-1)^{n+1}}{n} \cdot \left( x - 1 \right)^n.
$$
By a theorem of Quillen \cite{Quillen_graded}, this map induces an isomorphism
$$
(\Gr \log) \otimes \Q: \Gr G \otimes \Q \stackrel{\simeq}{\longrightarrow} \Gr \MalcevLie(G)
$$
between the graded Lie algebras associated with the lower central series of $G$
and with the filtration of $\MalcevLie(G)$ respectively.
This way of defining the Malcev Lie algebra is clearly functorial and,
if $G$ is residually torsion-free nilpotent, then the  map
$\MalcevLie: \Aut(G) \to \Aut(\MalcevLie(G))$ is injective.

We now consider the case of the free group $\pi = \pi_1(\Sigma_{g,1},\star)$.
We abbreviate $H_\Q := H \otimes \Q = (\pi/\Gamma_2 \pi)\otimes \Q$,
and $\T(H_\Q)$ denotes the tensor algebra over $H_\Q$. Following our convention,
the same notation is used for the degree completion of the tensor algebra.

\begin{definition}
An \emph{expansion}\index{expansion} of the free group $\pi$ is a map $\theta: \pi \to \T(H_\Q)$ 
which is multiplicative and satisfies
$$
\forall x\in\pi, \ \theta(x) = 1 + [x] + (\deg \geq 2) \ \in \T(H_\Q).
$$
The expansion is \emph{group-like}\index{group-like expansion}\index{expansion!group-like} if it takes group-like values.
\end{definition}

\noindent
Expansions of free groups have been considered by Lin \cite{Lin} and Kawazumi \cite{Kawazumi} 
in connection with Milnor's invariants and Johnson's homomorphisms, respectively. 
An expansion  $\theta$ of $\pi$ extends in a unique way to a filtered  algebra isomorphism
$\theta: \widehat{\Q[\pi]} \to \T(H_\Q)$, which induces the canonical isomorphism
\begin{equation}
\label{eq:graded_iso}
\Gr \widehat{\Q[\pi]} = \bigoplus_{k \geq 0} \frac{I^k}{I^{k+1}}
\stackrel{\simeq}{\longrightarrow} \T(H_\Q), \ 
\frac{I}{I^2} \ni \{x-1\} \longmapsto [x] \in H_\Q
\end{equation}
at the graded level \cite{Bourbaki,MKS}. 
Group-like expansions $\theta$ correspond to Hopf algebra isomorphisms $\theta: \widehat{\Q[\pi]} \to \T(H_\Q)$
which induce the canonical isomorphism (\ref{eq:graded_iso}) at the graded level.
Equivalently, by restricting to primitive elements,
group-like expansions $\theta$ of $\pi$ correspond to filtered Lie algebra isomorphisms
$\theta: \MalcevLie(\pi) \to \Lie_\Q$
which induce the canonical isomorphism
$$
\xymatrix{
\Gr \MalcevLie(\pi) \ar[rr]_-\simeq^{(\log \otimes \Q)^{-1}} & &
\Gr \pi \otimes \Q \ar[r]_-\simeq^-{(\ref{eq:iso_free})} & \Lie_\Q
}
$$
at the graded level. 
Here $\Lie_\Q := \Lie \otimes \Q$ denotes the Lie algebra freely generated by $H_\Q$ or, depending on the context, its degree completion.

\begin{example}
If a basis of $\pi$ is specified --- for instance, the basis $(\alpha,\beta)$
defined by a system of meridians and parallels as in Figure \ref{fig:surface} ---
then there is a unique  expansion $\theta_{(\alpha,\beta)}$ of $\pi$ defined by
\begin{equation}
\label{eq:exponential_expansion}
\forall i\in \{1,\dots,g\}, \
\theta_{(\alpha,\beta)}(\alpha_i) := \exp_\otimes\left([\alpha_i]\right)
\ \hbox{and} \ 
\theta_{(\alpha,\beta)}(\beta_i) := \exp_\otimes\left([\beta_i]\right).
\end{equation}
By construction, the expansion  $\theta_{(\alpha,\beta)}$ is group-like.
\end{example}

Let $\theta$ be a group-like expansion of $\pi$.
The following composition can be regarded as an \emph{infinitesimal} version of the Dehn--Nielsen representation,
which depends on $\theta$:
\begin{equation}
\label{eq:infinitesimal_DN}
\xymatrix{
\mcg_{g,1}\ \ar@{>->}[r]^-\rho \ar@/_2pc/@{>-->}[rrrr]^-{\varrho^\theta} & \Aut(\pi)\ \ar@{>->}[r]^-\MalcevLie & 
\Aut\left(\MalcevLie(\pi)\right) \ar[rr]^-{ \theta \circ \centereddot \circ \theta^{-1} }_-\simeq & &\Aut(\Lie_{\Q}).
}
\end{equation}
We denote it by $\varrho_\theta$.
Since any automorphism of the complete free Lie algebra $\Lie_\Q$ is determined 
by its restriction to $H_\Q$, we can replace $\varrho^\theta$  by the map
$$
\tau^\theta: \mcg_{g,1} \longrightarrow \Hom(H_\Q,\Lie_\Q), \
f \longmapsto \varrho^\theta(f)|_{H_\Q} - \Id_{H_\Q} = \theta\circ \MalcevLie(f_*)\circ \theta^{-1}|_{H_\Q} - \Id_{H_\Q}
$$
without loss of information.
This is essentially the map introduced in \cite{Kawazumi} 
under the name of the\index{total Johnson map}\index{map!total Johnson} \emph{total Johnson map}. 
Indeed, it can be checked from its definition that $\tau^\theta$
contains all the Johnson homomorphisms.

\begin{theorem}[Kawazumi \cite{Kawazumi}]
\label{th:total_Johnson_to_Johnson}
The degree $k$ part of the total Johnson map, restricted to the $k$-th term of the Johnson filtration,
coincides with the $k$-th Johnson homomorphism (with rational coefficients):
$$
\tau^\theta_k = \tau_k : \mcg_{g,1}[k] \longrightarrow  \Hom(H_\Q,\Lie_{\Q,k+1}) 
\simeq H_\Q\otimes \Lie_{\Q,k+1}.
$$
\end{theorem}

\noindent
Using Fox's free differential calculus, 
Perron defines another map on $\mcg_{g,1}$ which contains all the Johnson homomorphisms \cite{Perron}.

The extension of $\varrho^\theta$ to the monoid of homology cobordisms is straightforward.
Indeed, a group-like expansion $\theta$ induces a filtered Lie algebra isomorphism
$\theta: \MalcevLie(\pi/\Gamma_{k+1} \pi) \to \Lie_\Q/\Lie_{\Q,\geq k+1}$
for each integer $k\geq 1$. So, we can consider the monoid homomorphism $\varrho_k^\theta$ defined by the composition
$$
\xymatrix{
\cob_{g,1} \ar@{->}[r]^-{\rho_k} \ar@/_2pc/@{-->}[rrr]^-{\varrho_k^\theta} 
& \Aut(\pi/\Gamma_{k+1} \pi)\ \ar@{>->}[r]^-\MalcevLie & 
\Aut\left(\MalcevLie(\pi/\Gamma_{k+1} \pi)\right) \ar[r]^-{ \ \theta \circ \centereddot \circ \theta^{-1}  }_-\simeq 
 & \Aut\left(\Lie_{\Q}/\Lie_{\Q,\geq k+1}\right)  
}
$$
or, equivalently, we can consider the map
$$
\tau^\theta_{\leq k}: \cob_{g,1} \longrightarrow \Hom\left(H_\Q,\Lie_\Q/\Lie_{\Q,\geq k+1}\right),\ 
f \longmapsto \varrho^\theta_k(f)|_{H_\Q} - \Id_{H_\Q}.
$$
Passing to the limit $k\to +\infty$, we get a map 
$$
\tau^\theta: \cob_{g,1} \longrightarrow \Hom(H_\Q,\Lie_\Q) \simeq H_\Q \otimes \Lie_\Q
$$
whose restriction to $\mcg_{g,1}$ is compatible with the preceding definition. 
Theorem \ref{th:total_Johnson_to_Johnson} works with $\mcg_{g,1}[k]$ replaced by $\cob_{g,1}[k]$.

\begin{definition}
The \emph{infinitesimal Dehn--Nielsen representation}\index{Dehn--Nielsen representation!infinitesimal}\index{representation!infinitesimal Dehn--Nielsen}
(induced by the group-like expansion $\theta$ of $\pi$)
is the monoid homomorphism
$$
\varrho^\theta: \cob_{g,1} \longrightarrow \Aut(\Lie_\Q)
$$
which sends a homology cobordism $C$ to the unique filtered automorphism
of $\Lie_\Q$ whose restriction to $H_\Q$ is $\Id_{H_\Q} + \tau^\theta(C)$.
\end{definition}

The following statement can be proved from the definitions.

\begin{proposition}[See \cite{Massuyeau_tree}]
\label{prop:truncated_total_Johnson}
The degree $[k,2k[$ truncation of the total Johnson map $\tau^\theta$,
restricted to the $k$-th term of the Johnson filtration,
$$
\tau^\theta_{[k,2k[} := \sum_{m=k}^{2k-1} \tau^\theta_m:
\cob_{g,1}[k] \longrightarrow \bigoplus_{m=k}^{2k-1} H_\Q \otimes \Lie_{\Q,m+1} 
$$
is a monoid homomorphism, and its kernel is $\cob_{g,1}[2k]$.
\end{proposition}

Among group-like expansions, we prefer those which have the following property.
Recall that $\omega \in \Lambda^2 H \subset H\otimes H$ is the dual of the intersection pairing.

\begin{definition}
An expansion $\theta: \pi \to \T(H_\Q)$ is \emph{symplectic}\index{symplectic expansion}\index{expansion!symplectic}
if it is group-like and if it sends $\zeta$ to $\exp_\otimes(-\omega)$.
\end{definition}

\noindent
It is not difficult, using the Baker--Campbell--Hausdorff formula
and starting from the expansion (\ref{eq:exponential_expansion}),   
to construct degree-by-degree a symplectic expansion \cite{Massuyeau_tree}. 
If $\theta$ is symplectic, the infinitesimal Dehn--Nielsen representation $\varrho^\theta$ has values in
$$
\Aut_\omega\left(\Lie_\Q\right) := 
\left\{a \in \Aut\left(\Lie_\Q\right) : a(\omega) = \omega \right\}
$$
and it can be checked that $\tau^\theta_{[k,2k[}$ then takes values in the kernel of the bracket map:
$$
\operatorname{D}(H_\Q) := \operatorname{D}(H) \otimes \Q =
\Ker \left([\centereddot,\centereddot]: H_\Q \otimes \Lie_\Q \longrightarrow \Lie_\Q \right).
$$

\begin{theorem}[See \cite{Massuyeau_tree}]
\label{th:total_Johnson_to_Morita}
If $\theta$ is a symplectic expansion of $\pi$,
then there is a commutative diagram of the following form:
$$
\xymatrix @R=1em {
& H_3\left(\pi/\Gamma_{k+1}\pi\right)\ \ar@{>->}[r] & H_3\left(\pi/\Gamma_{k+1}\pi;\Q\right) \ar[dr]_-\simeq &\\
\cob_{g,1}[k] \ar[ru]^-{-M_k} \ar[dr]_-{\tau^\theta_{[k,2k[}} & & & H_3\left(\MalcevLie(\pi/\Gamma_{k+1}\pi);\Q\right). \\ 
&\displaystyle{\bigoplus_{m=k}^{2k-1} \operatorname{D}_{m+2}(H_\Q)} \ar[r]^-\simeq & 
H_3\left(\frac{\Lie_\Q}{\Lie_{\Q,\geq k+1}};\Q\right) \ar[ru]_-{\theta_*^{-1}}^\simeq & 
}
$$
Thus, the degree $[k,2k[$ truncation of the total Johnson map $\tau^\theta$
is equivalent to the $k$-th Morita homomorphism.
\end{theorem}

\noindent
The top map is injective because $H_3(\pi/\Gamma_{k+1}\pi)$ is torsion-free \cite{IO}.
The right-hand side top map is an application of the following result by Pickel:
with rational coefficients, the homology of a torsion-free finitely-generated nilpotent group is isomorphic
to the homology of its Malcev Lie algebra \cite{Pickel}. 
The bottom map is defined using diagrammatic descriptions of the spaces $\operatorname{D}(H_\Q)$
and $H_3(\Lie_\Q/\Lie_{\Q,\geq k+1};\Q)$. (See \S \ref{subsec:LMO_tree} for the diagrammatic description of $\operatorname{D}(H_\Q)$.) 
The diagram is shown to be commutative using an ``infinitesimal'' version of $M_k$.\\

We now sketch how to use symplectic expansions to deal with the case of a closed surface. 
If we fix a closed disk $D \subset \Sigma_g$ (as we did in \S \ref{subsec:bordered_to_closed}),
then we have a decomposition $\Sigma_{g} = \Sigma_{g,1} \cup D$ so that
$$
\clocase{\pi} = \pi_1(\Sigma_g,\star) \simeq \pi/\langle \zeta \rangle.
$$
Thus, a symplectic expansion $\theta$ induces a filtered Lie algebra isomorphism
$$
\clocase{\theta}: \MalcevLie(\clocase{\pi}) \stackrel{\simeq}{\longrightarrow} \clocase{\Lie}_\Q
$$
between the Malcev Lie algebra of $\clocase{\pi}$ and the complete Lie algebra
$\clocase{\Lie}_\Q := \Lie_\Q/ \langle \omega \rangle_{\operatorname{ideal}}$.
Then, the following composition gives an \emph{infinitesimal} 
version\index{representation!infinitesimal Dehn--Nielsen}\index{Dehn--Nielsen representation!infinitesimal} of the Dehn--Nielsen representation:
\begin{equation}
\label{eq:infinitesimal_DN_closed}
\xymatrix{
\mcg_{g}\ \ar@{>->}[r]^-\rho \ar@/_2pc/@{>-->}[rrrr]^-{\varrho^{\theta}} 
& \Out\left(\clocase{\pi}\right)\ \ar@{>->}[r]^-\MalcevLie & 
\Out\left(\MalcevLie\left(\clocase{\pi}\right)\right) \ar[rr]^-{\clocase{\theta}\circ \centereddot \circ \clocase{\theta}^{-1}}_-\simeq && \Out\left(\clocase{\Lie}_{\Q}\right).
}
\end{equation}
We denote it by $\varrho^\theta$.
Here, for a Lie algebra $\mathfrak{g}$ equipped with a complete filtration,
$\Out(\mathfrak{g})$ denotes the group of filtered automorphisms of $\mathfrak{g}$
modulo \emph{inner automorphisms}, \ie exponentials of inner derivations of $\mathfrak{g}$.
Equivalently, we can consider the \emph{total Johnson map}\index{total Johnson map}\index{map!total Johnson} defined by 
$$
\tau^{\theta}: \mcg_{g} \longrightarrow 
\frac{\Hom\left(H_\Q,\clocase{\Lie}_\Q\right)}{\exp_\circ(\ad \clocase{\Lie}_\Q)(\centereddot) - (\centereddot)}, \
f \longmapsto \left\{ \clocase{\theta} \circ \MalcevLie(f_*) \circ \clocase{\theta}^{-1}\vert_{H_\Q}- \Id_{H_\Q} \right\}
$$
where the space $\Hom(H_\Q,\clocase{\Lie}_\Q)$ is quotiented by the subspace
of homomorphisms of the form $\left(h \mapsto \exp_\circ(\ad u)(h) - h\right)$ where $u \in \clocase{\Lie}_\Q$.

As in the bordered case, we can extend the map $\tau^{\theta}$ 
and the homomorphism $\varrho^{\theta}$ to the monoid $\cob_g$ 
by considering, first, the nilpotent quotients $\clocase{\Lie}_\Q/\clocase{\Lie}_{\Q,\geq k+1}$
and by passing, next, to the limit $k\to +\infty$. 
There are analogues of Theorem \ref{th:total_Johnson_to_Johnson} 
and Proposition \ref{prop:truncated_total_Johnson} for the monoid $\cob_g$.
Clearly, the following square is commutative:
\begin{equation}
\label{eq:rho_rho}
\xymatrix{
\cob_{g,1} \ar[r]^-{\varrho^{\theta}} \ar@{->>}[d]_-{\clocase{\ \centereddot\ }} &  \ar@{->>}[d] \Aut_\omega\left(\Lie_\Q\right)\\
\cob_{g} \ar[r]_-{\varrho^{\theta}} & \Out\left(\clocase{\Lie}_\Q\right) .
}
\end{equation}

\section{The LMO homomorphism}

\label{sec:LMO}

We present the LMO homomorphism, which is a diagrammatic representation of the monoid of homology cylinders.
It is derived from the Le--Murakami--Ohtsuki invariant of closed oriented $3$-manifolds.
The LMO homomorphism dominates all the Johnson/Morita homomorphisms, 
and it will play a key role in Section \ref{sec:Lie_ring_homology_cylinders}.
In this section, we restrict ourselves to the surfaces $\Sigma_{g,1}$ and $\Sigma_g$.

\subsection{The algebra of symplectic Jacobi diagrams}

\label{subsec:symplectic_Jacobi_diagrams}

We start by defining the target of the LMO homomorphism.
For this, we need to define some kind of Feynman diagrams 
which  appear in the theory of finite-type invariants \cite{Bar-Natan,Ohtsuki}.

A \emph{Jacobi diagram}\index{Jacobi diagram} is a finite graph 
whose vertices have valence $1$ (\emph{external} vertices) or $3$ (\emph{internal}
vertices).  Each internal vertex is \emph{oriented}, in the sense that its incident edges are cyclically ordered.  
A Jacobi diagram is \emph{colored} by a set $S$ if a map from the set of its external vertices to $S$ is specified.
A \emph{strut} is a Jacobi diagram with only two external vertices and no internal vertex.  
Examples of Jacobi diagrams are shown in Figure \ref{fig:diagrams_examples}:
the custom is to draw such diagrams with dashed lines 
and the vertex orientations are given by the counter-clockwise orientation.
\begin{figure}[h]
\begin{center}
\includegraphics[scale=0.45]{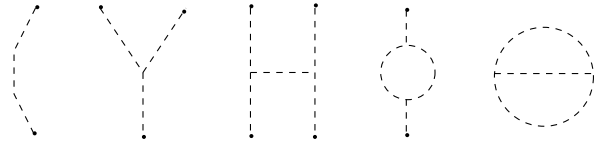}
\end{center}
\caption{Some examples of Jacobi diagrams: the strut, the Y graph, the H graph, the Phi graph and the Theta graph.}
\label{fig:diagrams_examples}
\end{figure}

As in the previous section, we start with the bordered case.
We denote $H:=H_1(\Sigma_{g,1})$ and $H_\Q := H\otimes \Q$.
We define the $\Q$-vector space\index{symplectic Jacobi diagram}\index{Jacobi diagram!symplectic}
$$
\A(H_\Q) := 
\frac{\Q\cdot \left\{ \begin{array}{c} \hbox{Jacobi diagrams without strut component}\\
\hbox{and with external vertices colored by } H_\Q  \end{array} \right\}}
{\hbox{AS, IHX, multilinearity}}.
$$
The ``AS'' and ``IHX'' relations are diagrammatic analogues of the antisymmetry and Jacobi identities  in Lie algebras:\\

\begin{center}
\labellist \small \hair 2pt
\pinlabel {AS} [t] at 102 -5
\pinlabel {IHX} [t] at 543 -5
\pinlabel {$= \ -$}  at 102 46
\pinlabel {$-$} at 484 46
\pinlabel {$+$} at 606 46
\pinlabel {$=0$} at 721 46 
\endlabellist
\centering
\includegraphics[scale=0.4]{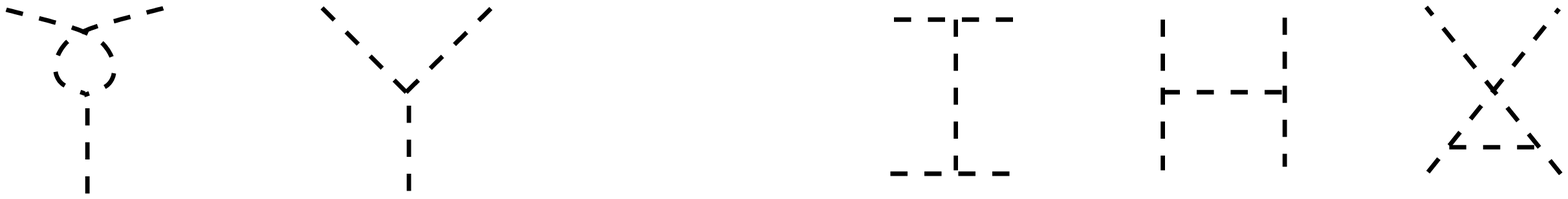}
\end{center}
\vspace{0.5cm}

\noindent
The ``multilinearity'' relation simply states that a Jacobi diagram $D$ with one external vertex $v$ colored
by $q_1\cdot h_1 + q_2\cdot h_2$ (with $q_1,q_2 \in \Q$ and $h_1,h_2 \in H_\Q$)
is equivalent to the  linear combination $q_1\cdot D_1 + q_2 \cdot D_2$
where $D_i$ is the Jacobi diagram $D$ with the vertex $v$ colored by $h_i$. 
The {\em degree} of a Jacobi diagram is the number of its internal vertices.
Thus, $\A(H_\Q)$ is  a graded vector space
$$
\A(H_\Q) = \bigoplus_{d=0}^\infty \A_d(H_\Q)
$$
where $\A_0(H_\Q)\simeq \Q$ is spanned by the empty diagram $\varnothing$.
Following our convention, the degree completion of $\A(H_\Q)$ is denoted in the same way.

\begin{example}[Degree $1$]
\label{ex:degree_one}
The spaces $\A_1(H_\Q)$ and $\Lambda^3 H_\Q$ are isomorphic by
$$
\Ygraphbottoptop{x_1}{x_2}{x_3} \mapsto x_1 \wedge x_2 \wedge x_3.
$$
\end{example}

As in the previous section, $\omega: H_\Q \otimes H_\Q \to \Q$ denotes the intersection pairing of $\Sigma_{g,1}$.
The group of the automorphisms of $H_\Q$ that preserve $\omega$,
namely the \emph{symplectic group}\index{symplectic group}\index{group!symplectic}
of $(H_\Q,\omega)$, is denoted by $\Sp(H_\Q)$.
It acts on $\A(H_\Q)$ in the obvious way. We shall now define on this space an $\Sp(H_\Q)$-equivariant structure of Hopf algebra. 
If we think of $H_\Q$-colored Jacobi diagrams as a kind of tensors, 
the multiplication $\star$ in  $\A(H_\Q)$ is defined as the contraction by the pairing $\omega/2$.
In more detail, let $D$ and $E$ be $H_\Q$-colored Jacobi diagrams, 
whose sets of external vertices  are denoted by $V$ and $W$ respectively. 
Then, we define
$$
D \star E := \sum_{\substack{V' \subset V,\ W' \subset W\\ 
\beta\ :\ V' \stackrel{\simeq}{\longrightarrow} W'}}\
\frac{1}{2^{|V'|}} \cdot \prod_{v\in V'} \omega\Big(\hbox{color}\big(v\big),\hbox{color}\big(\beta(v)\big)\Big)\ \cdot (D \cup_\beta E).
$$
Here, the sum is taken over all ways of identifying a subset $V'$ of $V$ with a subset $W'$ of $W$,
and $D \cup_\beta E$ is obtained from $D \sqcup E$ by gluing each vertex $v \in V'$ to $\beta(v) \in W'$. 
The comultiplication is given on an  $H_\Q$-colored  Jacobi diagram $D$ by
\begin{gather*}
\Delta(D) :=\sum_{D=D'\sqcup D''} D'\otimes D'',
\end{gather*}
where the sum runs over all ways of dividing the connected components of $D$ into two parts. 
Thus, the primitive part of $\A(H_\Q)$ is the subspace $\A^c(H_\Q)$
spanned by connected Jacobi diagrams.
The counit is given by $\varepsilon(D):=\delta_{D,\emptyset}$, and the
antipode is the unique algebra anti-automorphism satisfying 
$S(D) = -D$ if $D$ is connected and non-empty.

\begin{definition}
The Hopf algebra of \emph{symplectic Jacobi diagrams}\index{symplectic Jacobi diagram!algebra of} is the graded vector space
$\A(H_\Q)$ equipped with the multiplication $\star$ (with unit $\varnothing$),
the comultiplication $\Delta$ (with counit $\varepsilon$) and the antipode $S$.
The Lie algebra of \emph{symplectic Jacobi diagrams}\index{symplectic Jacobi diagram!Lie algebra of}
is $\A^c(H_\Q)$ equipped with the Lie bracket $[\centereddot,\centereddot]_\star$.
\end{definition}

\noindent
This Hopf algebra is introduced in \cite{HMass}.
Observe that all its operations  respect the action of $\Sp(H_\Q)$.
In particular, the Lie bracket $[\centereddot,\centereddot]_\star$ of $\A^c(H_\Q)$ is $\Sp(H_\Q)$-equivariant,
and this will play an important role in Section \ref{sec:Lie_ring_homology_cylinders}.
Another Lie bracket is defined in \cite{GL}, but this one is not $\Sp(H_\Q)$-equivariant.\\

We now define the algebra which will serve as a target for the LMO homomorphism in the closed case.
Let $I$ be the subspace of $\A(H_\Q)$ spanned by elements of the following form:

$$
\begin{array}{c}
\labellist \small \hair 2pt
\pinlabel {$\omega$} [t] at 12 0
\pinlabel {\scriptsize $x_1$} [t] at 42 0 
\pinlabel {\scriptsize $x_e$} [t] at 114 0
\pinlabel {\scriptsize $x_1$} [t] at 342 0
\pinlabel {\scriptsize $x_e$} [t] at 420 0
\pinlabel {\scriptsize $\cdots \cdots$} [t] at 78 0
\pinlabel {${\displaystyle -\frac{1}{4}\sum_{1\leq j <k \leq e} \omega(x_j,x_k)}$} at 221 54
\pinlabel {$D$} at 69 95
\pinlabel {$D$} at 373 95
\pinlabel {\scriptsize $j$} [tl] at 366 65
\pinlabel {\scriptsize $k$} [l] at 396 62
\pinlabel {\scriptsize $\cdots \widehat{x_j}\cdots\widehat{x_k}\cdots $} [t] at 381 4
\endlabellist
\centering
\includegraphics[scale=0.8]{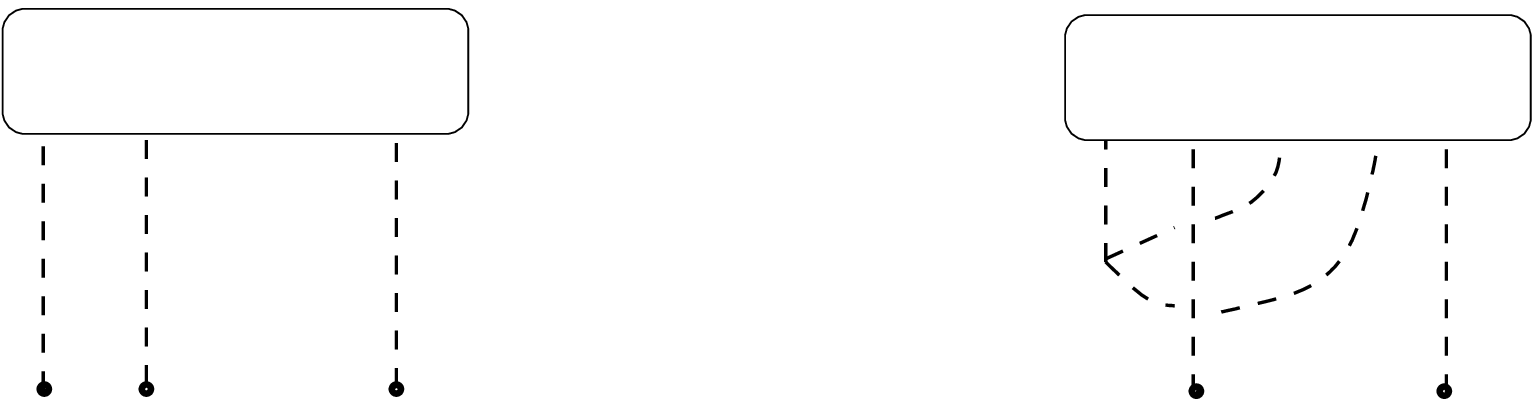}
\end{array}
$$
Here $x_1\dots,x_e$ are colors in $H_\Q$ and the $\omega$-colored vertex means
\begin{equation}
\label{eq:split_omega}
\begin{array}{c}
\labellist \small \hair 2pt
\pinlabel {${\displaystyle = \quad \sum_{i=1}^g}$} at 98 31 
\pinlabel {$\omega$} [t] at 3 0
\pinlabel {$\alpha_i$} [t] at 166 0 
\pinlabel {$\beta_i$} [t] at 202 0
\endlabellist
\centering
\includegraphics[scale=0.5]{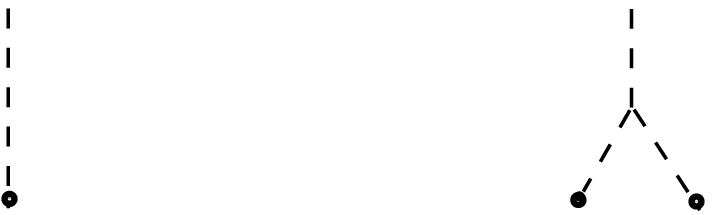}
\end{array}
\end{equation}
where $(\alpha_1,\beta_1,\dots,\alpha_g,\beta_g)$ is a symplectic basis of $H_\Q$ 
(defined, for instance, by a system of meridians and parallels as in Figure \ref{fig:surface}).
We also consider the subspace $I^c$ of $\A^c(H_\Q)$ 
spanned by elements of the above form  where $D$ is assumed to be connected.
It can be checked that $I$ is an ideal of the Hopf algebra $\A(H_\Q)$ 
and that $I^c$ is an ideal of the Lie algebra $\A^c(H_\Q)$. 
So, we can define the following quotient algebras:
$$
\clocase{\A}(H_\Q) := \A(H_\Q)/I
\quad \hbox{and} \quad  
\clocase{\A}^c(H_\Q) := \A^c(H_\Q)/I^c.
$$ 

\begin{example}[Degree $1$]
We have an isomorphism between $\clocase{\A}_1(H_\Q)= \clocase{\A}^c_1(H_\Q)$
and the quotient space $\Lambda^3 H_\Q/ \omega \wedge H_\Q$.
\end{example}

\subsection{Definition of the LMO homomorphism}

\label{subsec:definition_LMO}

We now sketch the construction of the LMO homomorphism \cite{HMass} starting from the LMO invariant.
The latter\index{LMO invariant} is an invariant of closed oriented $3$-manifolds $M$ introduced by Le, Murakami and Ohtsuki in \cite{LMO}.
Denoted by $\Omega(M)$, it takes values in the space of trivalent Jacobi diagrams:
$$
\A(\varnothing) :=
\frac{\Q\cdot \left\{ \hbox{Jacobi diagrams without external vertices}\right\}}{\hbox{AS, IHX}}.
$$
It can be regarded as a far-reaching generalization of the Casson--Walker--Lescop invariant $\lambda(M)$ \cite{Lescop}, since we have
$$
\Omega(M) = \varnothing + (-1)^{\beta_1(M)} \cdot
\frac{\lambda(M)}{2}\cdot \thetagraph + (\deg >2) \ \in \A(\varnothing).
$$
The reader is referred to Ohtsuki's book \cite{Ohtsuki} for an introduction to this invariant.

The LMO invariant can be extended to $3$-manifolds with boundary in several ways \cite{MO,CL,ABMP}.
Here, we need the \emph{LMO functor}\index{LMO functor}\index{functor!LMO} introduced in \cite{CHM}. 
Its source is the category of ``Lagrangian cobordisms'' whose objects are integers $g\geq 0$
and whose morphisms $g\to h$ are cobordisms (with corners) between $\Sigma_{g,1}$ and $\Sigma_{h,1}$,
which are required to satisfy certain homological conditions. The target of the LMO functor is a certain category of Jacobi diagrams. 
We refer to \cite{CHM} for the construction. It should be emphasized that the definition of the LMO functor requires two preliminary choices:
\begin{enumerate}
\item A Drinfeld associator must be specified. 
(Just like the LMO invariant, the LMO functor is constructed from the Kontsevich integral of tangles.)
\item For each $g\ge 0$, a system of meridians and parallels should be fixed on $\Sigma_{g,1}$ (as in Figure \ref{fig:surface}).
\end{enumerate}
Homology cylinders are ``Lagrangian'' in the sense of \cite{CHM}. Thus, the LMO functor restricts to a monoid homomorphism
$$
\widetilde{Z}^Y: \cyl_{g,1} \longrightarrow \A^Y\left(\set{g}^+ \cup \set{g}^-\right).
$$
Here, the target space is the space of Jacobi diagrams without strut component and
whose external vertices are colored by the finite set
$$
\set{g}^+ \cup \set{g}^- := \{1^+,\dots, g^+\} \cup \{1^-,\dots, g^-\}.
$$
The multiplication in this space is defined by
$$
D \star E :=
\left(\! \!  \begin{array}{c}
\hbox{sum of all ways of gluing \emph{some} of the $i^+$-colored vertices of $D$}\\
\hbox{to \emph{some} of the $i^-$-colored vertices of $E$, for all $i=1,\dots,g$}
\end{array}\! \!  \right) .
$$ 
This product was discovered in \cite{GL}. (See Section \ref{sec:Y}.)

Note that, in the definition of the LMO functor, the colors $1^+, \dots, g^+$ 
refer to the curves $\beta_1,\dots, \beta_g$ (in the top surface of the cobordism)
while the colors $1^-, \dots, g^-$ refer to the curves $\alpha_1,\dots, \alpha_g$ (in the bottom surface of the cobordism).
So, it is natural to  identify $\A^Y\left(\set{g}^+ \cup \set{g}^-\right)$ with $\A(H_\Q)$
by simply changing the colors with the rules $j^+\mapsto [\beta_j]$ and $j^- \mapsto [\alpha_j]$.
Unfortunately, the multiplication on $\A(H_\Q)$ corresponding to the multiplication $\star$ on $\A^Y\left(\set{g}^+ \cup \set{g}^-\right)$
by that identification is not $\Sp(H_\Q)$-equivariant. Instead of that identification, we consider the map
$$
\kappa: \A^Y\left(\set{g}^+ \cup \set{g}^-\right) \longrightarrow \A(H_\Q)
$$
defined by 
$$
\kappa(D) := (-1)^{\chi(D)}\cdot 
\left( \! \! \!  \left. \begin{array}{c} \hbox{sum of all ways of $(\times 1/2)$-gluing some $i^-$-colored}\\
\hbox{vertices of $D$ with some of its $i^+$-colored vertices} \end{array}
\! \right| \! \begin{array}{l} j^+ \mapsto \beta_j \\ j^- \mapsto \alpha_j \end{array}  \! \! \! 
\right).
$$
Here, $\chi(D)$ is the Euler characteristic of a Jacobi diagram $D$,
and a ``$(\times 1/2)$-gluing'' means the gluing of two vertices 
and the multiplication of the resulting diagram by $1/2$.
It can be proved that  $\kappa$ is an isomorphism and that it sends the multiplication 
$\star$ of $\A^Y\left(\set{g}^+ \cup \set{g}^-\right)$ to the multiplication $\star$ of $\A(H_\Q)$.

\begin{definition}
The \emph{LMO homomorphism}\index{LMO homomorphism}\index{homomorphism!LMO} is the monoid map 
$$
Z := \kappa \circ \widetilde{Z}^Y : \cyl_{g,1} \longrightarrow \A(H_\Q).
$$
\end{definition}

\noindent
This invariant of homology cylinders is universal among $\Q$-valued finite-type invariants \cite{CHM,HMass}
(see Section \ref{sec:Lie_ring_homology_cylinders} in this connection), and it takes group-like values.
Habegger defines from the LMO invariant another map $\cyl_{g,1} \to \A(H_\Q)$ with the same properties \cite{Habegger},
but he does not address the multiplicativity issue.\\

We now consider the case of a closed surface.
We think of $\Sigma_g$ as the union of $\Sigma_{g,1}$ with a closed disk $D$
as we did in \S \ref{subsec:bordered_to_closed}.
It can be shown from Lemma \ref{lem:closing} that, if the target of the LMO functor is quotiented
by some appropriate relations, then it induces a functor on the category of cobordisms between \emph{closed} surfaces.
Those relations in $\A^Y\left(\set{g}^+ \cup \set{g}^-\right)$ are sent by the isomorphism $\kappa$
to the Hopf ideal $I$ introduced at the end of \S \ref{subsec:symplectic_Jacobi_diagrams}.

\begin{definition}
The \emph{LMO homomorphism}\index{LMO homomorphism}\index{homomorphism!LMO} is the monoid map 
$$
Z := \kappa \circ \widetilde{Z}^Y : \cyl_{g} \longrightarrow \clocase{\A}(H_\Q).
$$
\end{definition}

\noindent
So, by construction, the following square is commutative:
\begin{equation}
\label{eq:Z_Z}
\xymatrix{
\cyl_{g,1} \ar[r]^-{Z} \ar@{->>}[d]_-{\clocase{\ \centereddot\ }} &  \ar@{->>}[d] \A(H_\Q)\\
\cyl_{g} \ar[r]_-{Z} &\clocase{\A}(H_\Q).
}
\end{equation}
Of course, the invariant of homology cylinders $Z: \cyl_g \to \clocase{\A}(H_\Q)$
depends on the way $\Sigma_{g,1}$ is embedded on $\Sigma_g$
(as well as on the choices of a Drinfeld associator and a system of meridians and parallels on $\Sigma_{g,1}$). 

\begin{example}[Genus $0$]
We saw in Example \ref{ex:genus_0_cobordisms}
that  the monoids $\cyl_{0,1}$ and $\cyl_0$ can be identified with the monoid of  homology $3$-spheres.
We have $\A(H_\Q)= \clocase{\A}(H_\Q)=\A(\varnothing)$ for $g=0$
and, by construction,  the LMO homomorphism $Z$ coincides in this case with the LMO invariant $\Omega$.
\end{example}

\subsection{The tree-reduction of the LMO homomorphism}

\label{subsec:LMO_tree}

A Jacobi diagram is \emph{looped} if at least one of its connected components is not contractible. 
The subspace of $\A(H_\Q)$ generated by looped Jacobi diagrams is an ideal, so that we can consider the quotient Hopf algebra
$$
\A^t(H_\Q) := \A(H_\Q)/\langle \hbox{looped Jacobi diagrams} \rangle.
$$
As a $\Q$-vector space, $\A^t(H_\Q)$ can be identified with the subspace of $\A(H_\Q)$ generated by tree-shaped Jacobi diagrams.
Thus, the composition
$$
\xymatrix{
\cyl_{g,1} \ar[r]^-Z  \ar@/_1.5pc/@{-->}[rr]_-{Z^t}  & \GLike \A(H_\Q) \ar@{->}[r]& \GLike \A^t(H_\Q)
}
$$
is called the \emph{tree-reduction} of the LMO homomorphism and is denoted by $Z^t$.
It takes values in the group-like part of the Hopf algebra $\A^t(H_\Q)$.

Besides, provided we are given a symplectic expansion $\theta$ of $\pi$,
we have the infinitesimal Dehn--Nielsen representation defined in \S \ref{subsec:total_Johnson}:
$$
\varrho^\theta: \cyl_{g,1} \longrightarrow \IAut_\omega(\Lie_\Q).
$$
Here, we have restricted $\varrho^\theta$ to the monoid of homology cylinders, so that
values are taken in the group $\IAut_\omega(\Lie_\Q)$ of automorphisms of the complete free Lie algebra
$\Lie_\Q$ that induce the identity at the graded level and fix $\omega$.
The logarithmic series defines a bijection
$$
\log_\circ:
\IAut_\omega(\Lie_\Q) \stackrel{\simeq}{\longrightarrow} \Der_\omega \left(\Lie_\Q, \Lie_{\Q,\geq 2}\right),
\ a \longmapsto \sum_{n \geq 1} \frac{(-1)^{n+1}}{n} \cdot \left( a - \Id \right)^n
$$
between $\IAut_\omega(\Lie_\Q)$ and  the Lie algebra $\Der_\omega \left(\Lie_\Q, \Lie_{\Q,\geq 2}\right)$
of derivations of $\Lie_\Q$ that vanish on $\omega$ and take values in $\Lie_{\Q,\geq 2}$.
This Lie algebra appears in Kontsevich's work \cite{Kontsevich_Gelfand,Kontsevich_ECM},
where it is implicitly identified with the Lie algebra
$$
\A^{t,c}(H_\Q) := \A^c(H_\Q)/\langle \hbox{connected looped Jacobi diagrams} \rangle.
$$
To recall that identification, let us observe that a derivation of $\Lie_\Q$ is determined by its restriction to $H_\Q$.
Hence we have an isomorphism
$$
\Der\left(\Lie_\Q, \Lie_{\Q,\geq 2}\right) \simeq \Hom(H_\Q,\Lie_{\Q,\geq 2}) \simeq H_\Q \otimes \Lie_{\Q,\geq 2}
$$ 
which restricts to an isomorphism
$$
\Der_\omega\left(\Lie_\Q, \Lie_{\Q,\geq 2}\right) \simeq
\operatorname{D}_{\geq 3}(H_\Q) = \Ker\left([\centereddot,\centereddot]: 
H_\Q \otimes \Lie_{\Q,\geq 2} \longrightarrow \Lie_{\Q,\geq 3} \right).
$$
The corresponding Lie algebra structure of $\operatorname{D}(H_\Q)$
appears (with integral coefficients) in Morita's work on the Johnson homomorphisms \cite{Morita_ICM,Morita_Abelian}.
It can be proved \cite{Levine_addendum} that the map
$$
\eta_k: \A^{t,c}_k(H_\Q) {\longrightarrow} \operatorname{D}_{k+2}(H_\Q),
\ T \longmapsto \sum_v \operatorname{color}(v) \otimes \operatorname{comm}(T_v)
$$
is an isomorphism for all $k\geq 1$. Here, the sum is over all external vertices $v$ of $T$
and $\operatorname{comm}(T_v)$ is the iterated Lie bracket encoded by $T$ ``rooted'' at $v$: we have, for instance,

$$
\operatorname{comm}\Bigg(\begin{array}{c}
\labellist \small \hair 2pt
\pinlabel {$h_1$} [b] at 2 162
\pinlabel {$h_2$} [b] at 111 162
\pinlabel {$h_3$} [b] at 167 162
\pinlabel {$h_4$} [b] at 222 162
\pinlabel {$v$} [t] at 108 0
\endlabellist
\includegraphics[scale=0.25]{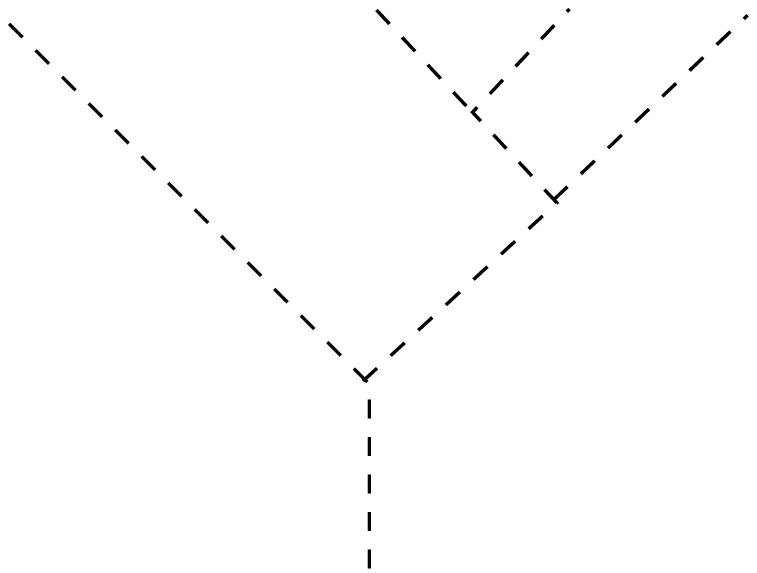}
\end{array}
\Bigg)= [h_1,[[h_2,h_3],h_4]].
$$
It is easily checked that the isomorphism
$$
\eta: \A^{t,c}(H_\Q) \stackrel{\simeq}{\longrightarrow} \operatorname{D}_{\geq 3}(H_\Q) 
\simeq \Der_\omega\left(\Lie_\Q, \Lie_{\Q,\geq 2}\right)
$$
is a Lie algebra map (which shifts the degree by $2$).

We can now state an algebraico-topological description of the tree-reduction of the LMO homomorphism.

\begin{theorem}[See \cite{Massuyeau_tree}]
\label{th:tree-reduction_LMO}
The LMO functor defines a symplectic expansion $\theta$ of $\pi$ 
such that the following diagram is commutative:
$$
\xymatrix @R=1em {
 & \GLike \A^t(H_\Q) \ar[r]^-{\log_\star}_-\simeq &  \A^{t,c}(H_\Q) \ar[dd]_-\simeq^-{\eta} \\
\cyl_{g,1} \ar[ru]^-{Z^t} \ar[rd]_-{\varrho^\theta} & & \\
& \IAut_\omega(\Lie_\Q) \ar[r]_-{\log_\circ}^-\simeq  & \Der_\omega \left(\Lie_\Q, \Lie_{\Q,\geq 2}\right).
}
$$
\end{theorem}

\noindent
Thus, the tree-reduction of the LMO homomorphism encodes the action of $\cyl_{g,1}$ on the Malcev Lie algebra $\MalcevLie(\pi)$.
Since the infinitesimal Dehn--Nielsen representation (\ref{eq:infinitesimal_DN}) is injective on the mapping class group, we deduce the following.

\begin{corollary}[See \cite{CHM}]
The LMO homomorphism is injective on  $\Torelli_{g,1}$.
\end{corollary}

\noindent
Theorem \ref{th:tree-reduction_LMO} is inspired by the work of Habegger and Masbaum \cite{HMasb} 
on the Kontsevich integral of string-links in $D^2 \times [-1,1]$.
With the following correspondence in mind,
one sees that Theorem \ref{th:tree-reduction_LMO} is very close in spirit to their ``global formula for Milnor's invariants'':

\begin{center}
\begin{tabular}{|r|l|}
\hline monoid of string-links & monoid of homology cylinders\\
\hline
pure braid group & Torelli group \\
Kontsevich integral & LMO homomorphism \\
Milnor's invariants & total Johnson map\\ \hline
\end{tabular}
\end{center}

We can also deduce the following from Theorem \ref{th:tree-reduction_LMO} and Theorem \ref{th:total_Johnson_to_Johnson}. 
A similar result was proved by Habegger in  \cite{Habegger} for his LMO-type map $\cyl_{g,1} \to \A(H_\Q)$.

\begin{corollary}[See \cite{CHM}]
\label{cor:LMO_to_Johnson}
Let $C \in \cyl_{g,1}$. The lowest degree non-trivial term of $Z^t(C)\in \A^t(H_\Q)$
coincides,  via the isomorphism $\eta$, with the first non-trivial Johnson homomorphism of $C$.
\end{corollary}

\noindent
More generally, it follows from Theorem \ref{th:tree-reduction_LMO}  and Theorem \ref{th:total_Johnson_to_Morita} that
the restriction to  $\cob_{g,1}[k]$ of the degree $[k,2k[$ truncation of $Z^t$ 
corresponds (by an explicit isomorphism) to the $k$-th Morita homomorphism \cite{Massuyeau_tree}.\\

We now outline the case of a closed surface. 
Let $I^{t}$ and $I^{t,c}$ be the images of the ideals $I$ and $I^c$ (introduced at the end of \S \ref{subsec:symplectic_Jacobi_diagrams})
in the Hopf algebra $\A^{t}(H_\Q)$ and in the Lie algebra $\A^{t,c}(H_\Q)$, respectively. 
Thus, if we think of $\A^{t,c}(H_\Q)$ as the subspace of $\A(H_\Q)$ generated by connected tree-shaped Jacobi diagrams,
then the subspace $I^{t,c}$ is generated by those diagrams having an $\omega$-colored vertex, as in (\ref{eq:split_omega}).
We denote
$$
\clocase{\A}^{t}(H_\Q) := \A^{t}(H_\Q)/I^t
\quad \hbox{and} \quad
\clocase{\A}^{t,c}(H_\Q) := \A^{t,c}(H_\Q)/I^{t,c}.
$$
The isomorphism $\eta$ sends $I^{t,c}$ to the ideal 
$$
\big(\Der(\Lie_\Q,\langle \omega \rangle_{\operatorname{ideal}}) + 
(\hbox{inner derivations of } \Lie_\Q)\big) \cap \Der_\omega(\Lie_\Q,\Lie_{\Q,\geq 2}).
$$
So, $\eta$ induces an isomorphism from $\clocase{\A}^{t,c}(H_\Q)$
to the Lie algebra $\ODer \left(\clocase{\Lie}_\Q, \clocase{\Lie}_{\Q,\geq 2}\right)$ of derivations
of $\clocase{\Lie}_\Q$ with values in $\clocase{\Lie}_{\Q,\geq 2}$, modulo inner derivations.
Let $\IOut\left(\clocase{\Lie}_\Q\right)$ be the group
of filtered automorphisms of $\clocase{\Lie}_\Q$ that induce the identity at the graded level, modulo inner automorphisms.
The following is an application of Theorem \ref{th:tree-reduction_LMO}
using the commutative squares (\ref{eq:Z_Z}) and (\ref{eq:rho_rho}).

\begin{theorem}
\label{th:tree-reduction_LMO_closed}
Let $\theta$ be the symplectic expansion of $\pi$ defined from the LMO functor in \cite{Massuyeau_tree}.
Then, the following diagram is commutative:
$$
\xymatrix  @R=1em {
 & \GLike\left(\clocase{\A}^t(H_\Q)\right) \ar[r]^-{\log_\star}_-\simeq &  \clocase{\A}^{t,c}(H_\Q) \ar[dd]_-\simeq^-\eta \\
\cyl_{g} \ar[ru]^-{Z^t} \ar[rd]_-{\varrho^\theta} & & \\
& \IOut\left(\clocase{\Lie}_\Q\right) \ar[r]_-{\log_\circ}^-\simeq  & \ODer \left(\clocase{\Lie}_\Q, \clocase{\Lie}_{\Q,\geq 2}\right).
}
$$
\end{theorem}

\noindent
Since the infinitesimal Dehn--Nielsen representation (\ref{eq:infinitesimal_DN_closed}) is injective on the mapping class group,
we have the following application of Theorem \ref{th:tree-reduction_LMO_closed}.

\begin{corollary}
The LMO homomorphism is injective on  $\Torelli_{g}$.
\end{corollary}

\noindent
One can also deduce from Theorem \ref{th:tree-reduction_LMO_closed}
the analogue of Corollary \ref{cor:LMO_to_Johnson} for $\Sigma_g$.

\section{The $Y$-filtration on the monoid of homology cylinders}

\label{sec:Y}

In this section, we consider the relation of $Y_k$-equivalence among homology cylinders, which is defined for every
$k\geq 1$ by surgery techniques.  The $Y$-filtration is the decreasing sequence of submonoids of $\cyl(\Sigma_{g,b})$
obtained by considering, for all $k\geq 1$, homology cylinders that are $Y_k$-equivalent to $\Sigma_{g,b} \times [-1,1]$.  
In some respects, this filtration is similar to the lower central series of the Torelli group $\Torelli (\Sigma_{g,b})$.

\subsection{The $Y_k$-equivalence relation}

\label{subsec:Y}

The lower central series is a fundamental tool in the study of groups.  Much of the structure of a residually nilpotent
group is contained in the associated graded Lie ring of the lower central series.  The family of $Y_k$-equivalence
relations  of $3$-manifolds ($k\ge 1$) plays the same role in $3$-dimensional topology as lower central series.
Those equivalence relations, which we shall now recall, 
have been defined and studied by Goussarov \cite{Goussarov,Goussarov_clovers} and the first author \cite{Habiro}.
We shall follow the terminology of \cite{Habiro}.

Let $M$ be a compact oriented $3$-manifold.
A {\em graph clasper}\index{graph clasper} in $M$ is a compact, connected surface $C$ embedded in the interior of $M$, 
which is equipped with a decomposition into the following types of subsurfaces: {\em leaves}, {\em nodes} and {\em edges}.  A
leaf is an annulus, a node is a disc, and an edge is a band.  Leaves and nodes are called {\em constituents}.  Each edge
connects either two distinct constituents or connects a node with itself.  Any two distinct constituents are disjoint
from each other.  Each leaf is incident with exactly one edge by one arc in the boundary.  Each node is either incident
with three distinct edges or incident with two distinct edges, one of which connects the node with itself.  
(Thus, in each case, the node is attached to the edges along exactly three arcs in the boundary circle.)  
See Figure \ref{fig:graph_clasper} for an example of a graph clasper.

\begin{figure}[h]
\begin{center}
{\labellist \small \hair 0pt 
\pinlabel {an edge} [l] at 94 21
\pinlabel {a node} [l] at 190 333
\pinlabel {a leaf} [l] at 608 165
\pinlabel {$=$} at 809 90
\endlabellist}
\includegraphics[scale=0.3]{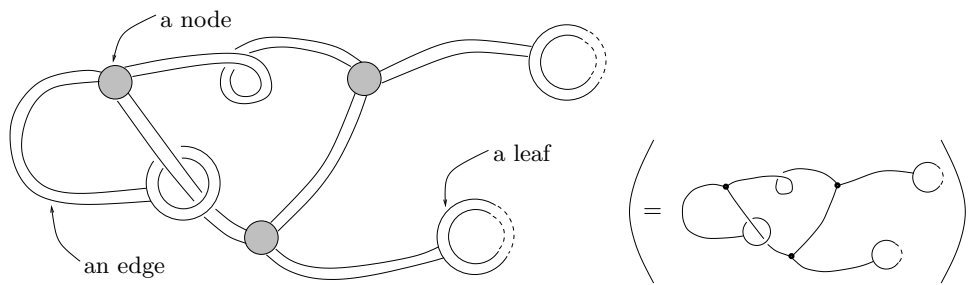}
\end{center}
\caption{An example of graph clasper $C \subset M$ with $3$ nodes, $3$ leaves and $6$ edges.
(And how it is drawn with the blackboard framing convention.)}
\label{fig:graph_clasper}
\end{figure}

A {\em tree clasper}\index{tree clasper} is a graph clasper $C$ such that $C\setminus \text{(the leaves of $C$)}$ is simply-connected.
The {\em degree} of a graph clasper $C$ is defined to be the number of nodes contained in $C$.  
A graph (respectively a tree) clasper of degree $k$ is called a {\em $Y_k$-graph} (respectively a {\em $Y_k$-tree}).
For instance, $Y_0$-graphs (which are also called ``basic claspers'' or ``I-claspers'' in the literature) consist of only one edge and two leaves:
$$
\includegraphics[scale=0.3]{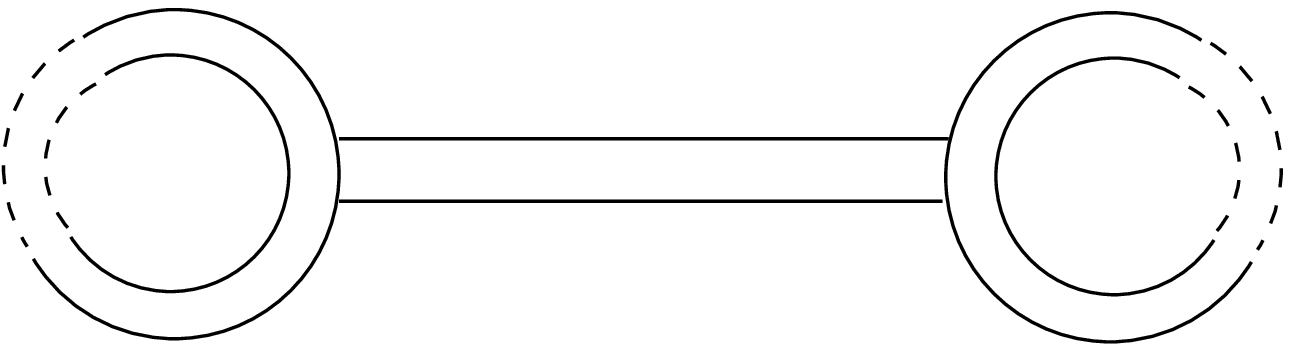}
$$

Surgery along a graph clasper $C\subset M$ is defined as follows.  
We first replace each node with three leaves linking like the Borromean rings in the following way: 
\begin{center}
{\labellist \small \hair 0pt 
\pinlabel {$\longrightarrow$} at 409 143
\endlabellist}
\includegraphics[scale=0.3]{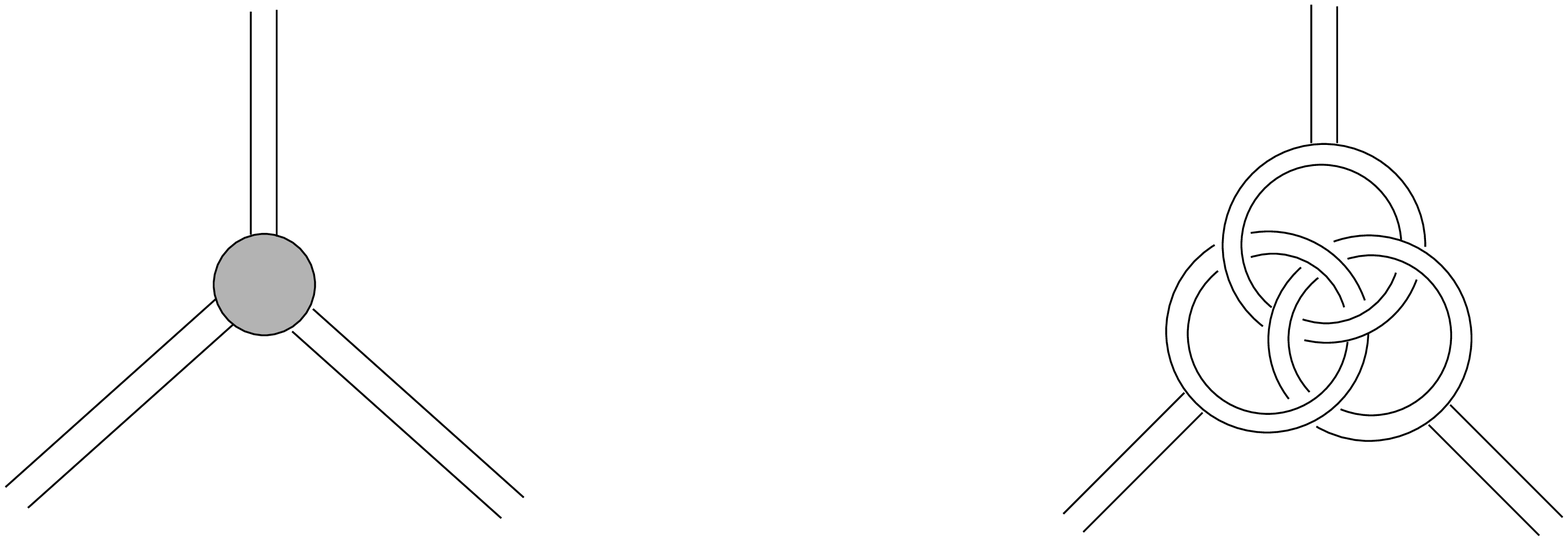}
\end{center}
Thus, we obtain a disjoint union of $Y_0$-trees.    
Next, we replace each $Y_0$-tree with a $2$-component framed link as follows:
\begin{center}
{\labellist \small \hair 0pt 
\pinlabel {$\longrightarrow$} at 460 50
\endlabellist}
\includegraphics[scale=0.3]{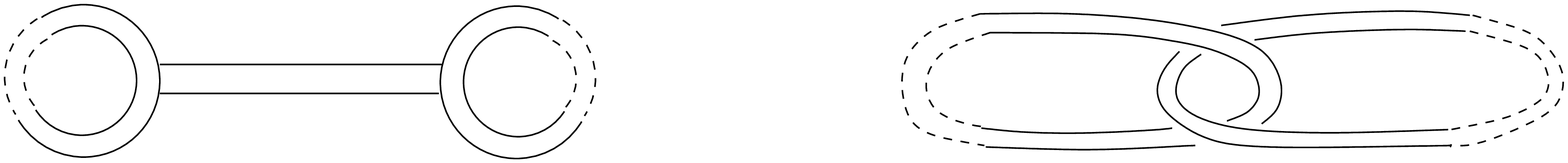}
\end{center}
See Figure \ref{fig:Y_to_framed_link} for an example.
Then, \emph{surgery} along the graph clasper $C$ is defined to be the surgery along the framed link thus obtained in $M$.
The resulting $3$-manifold is denoted by $M_C$.

\begin{figure}[h]
\begin{center}
{\labellist \small \hair 0pt 
\pinlabel {$\longrightarrow$}  at 384 142
\pinlabel {$\longrightarrow$}  at 820 142
\endlabellist}
\includegraphics[scale=0.25]{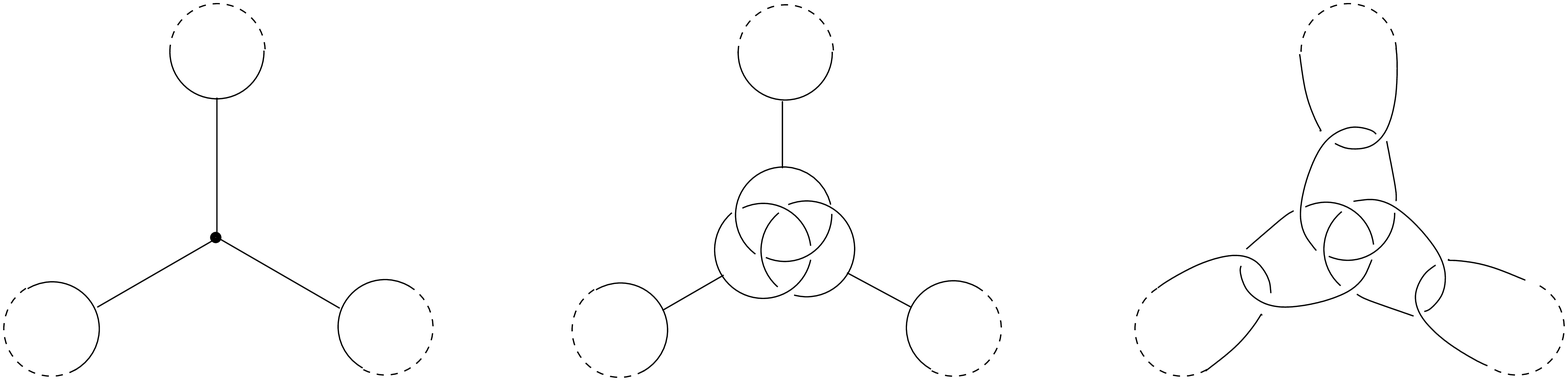}
\end{center}
\caption{The framed link associated with a $Y_1$-tree. (Blackboard framing convention is used here.)}
\label{fig:Y_to_framed_link}
\end{figure}

\index{calculus of claspers}
\emph{Calculus of claspers} is developed in \cite{Goussarov_clovers,Habiro,GGP},
in the sense that some specific moves between graph claspers are shown to produce  by surgery homeomorphic $3$-manifolds.
In fact, one can regard a node in a graph clasper as a commutator morphism for a 
Hopf algebra object in the category $\mathbf{Cob}$ of $3$-dimensional cobordisms introduced in \cite{CY,Kerler}. 
More general definition of claspers  involves another kind of constituents 
called ``boxes'' which corresponds to multiplication and comultiplication for that Hopf algebra object.
In this way, calculus of claspers can be regarded as a topological commutator calculus in $3$-manifolds.
We refer to \cite{Habiro} for precise statements.

\begin{example}
\label{ex:move_2}
One of the simplest moves in calculus of claspers consists in cutting an edge of a graph clasper
and inserting a Hopf link of two leaves:
\begin{equation}
\label{eq:move_2}
{\labellist \small \hair 0pt 
\pinlabel {$\longleftrightarrow$}  at 333 26
\pinlabel {\scriptsize edge}  [b] at 124 34
\endlabellist}
\includegraphics[scale=0.4]{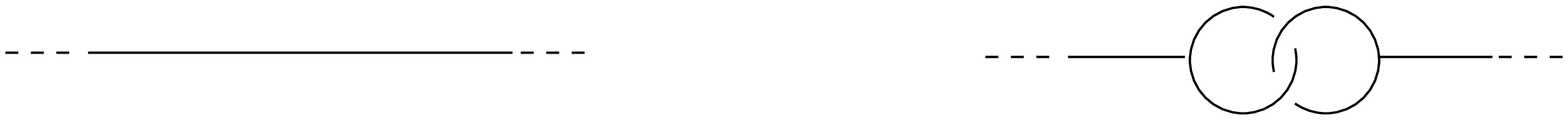}
\end{equation}
This (called ``Move 2'' in \cite{Habiro}) can be deduced from the ``slam dunk'' move (\ref{eq:slam_dunk}).
\end{example}

A {\em $Y_k$-surgery}\index{clasper surgery}
on a compact oriented $3$-manifold $M$ is defined to be the surgery along some $Y_k$-graph in $M$.  
The {\em   $Y_k$-equivalence}\index{Yk-equivalence relation@$Y_k$-equivalence relation}\index{relation!$Y_k$-equivalence}
is the equivalence relation generated 
by $Y_k$-surgeries and orientation-preserving homeomorphisms  preserving the boundary parameterization.
Considering $3$-manifolds up to $Y_k$-equivalence is something like considering elements in a group up to multiplication
by iterated commutators of class $k$. 

\begin{example}[Degree $1$]
\label{ex:Y_1}
A $Y_1$-surgery is equivalent to the ``Borromean surgery'' introduced by Matveev in \cite{Matveev}. 
It is proved there that two closed connected oriented $3$-manifolds are Borromean surgery equivalent, 
or $Y_1$-equivalent, if and only if there is an isomorphism of first homology groups 
for these two $3$-manifolds which induces an isomorphism of torsion linking pairings.  
The $Y_1$-surgeries are used in the definition of finite-type invariants of
$3$-manifolds in the sense of Goussarov and the first author \cite{Goussarov,Habiro}.
\end{example}

Calculus of claspers can be used to prove general properties for the $Y_k$-surgery and the $Y_k$-equivalence.
For instance, we deduce from  (\ref{eq:move_2}) that surgeries along $Y_k$-trees suffice to generate the $Y_k$-equivalence.
Let us give two more examples.

\begin{proposition}
\label{prop:Y_Y}
If $1\le k\le l$, then $Y_l$-equivalence implies $Y_k$-equivalence.
\end{proposition}

\noindent
This fact follows from ``Move 9'' in \cite{Habiro}.
Thus, the $Y_k$-equivalence gets finer and finer as $k$ increases.
It is conjectured that two $3$-manifolds are orientation-preserving homeomorphic 
if and only if they are \emph{$Y_\infty$-equivalent} (\ie $Y_k$-equivalent for all $k\ge 1$).

As mentioned in Example \ref{ex:Y_1},
a $Y_1$-surgery preserves the homology   (hence so do $Y_k$-surgeries for $k\ge 1$).
More generally, we have the following important property.

\begin{proposition}
\label{prop:nilpotent_quotients}
If two $3$-manifolds $M$ and $M'$ are related by a $Y_k$-surgery, then there is a canonical isomorphism
\begin{gather*}
\pi _1(M)/\Gamma _{k+1}\pi _1(M)\simeq \pi _1(M')/\Gamma _{k+1}\pi _1(M')
\end{gather*}
between the $k$-th nilpotent quotient of $\pi_1(M)$ and that of $\pi_1(M')$.
\end{proposition}

\noindent
This is another illustration of the fact that a $Y_k$-graph can be interpreted 
as an iterated commutator of length $k+1$.

\subsection{$Y_k$-equivalence and homology cylinders}

\label{subsec:Y_cylinders}

Let us now specialize the $Y_k$-equivalence relations to the class of homology cylinders.
The following statement gives a characterization of homology cylinders in terms of $Y_1$-equivalence.

\begin{proposition}[See \cite{Habiro}]
\label{prop:Y_1_cylinders}
A cobordism $M$ of the surface $\Sigma_{g,b}$ is $Y_1$-equivalent to $\Sigma_{g,b}\times [-1,1]$ 
if and only if $M$ is a homology cylinder.
\end{proposition}

\noindent
This surgery characterization of $\cyl(\Sigma_{g,b})$ is proved in \cite{Habegger,MM},
and it is deduced in \cite{Massuyeau_DSP} from a result by Matveev \cite{Matveev}. 
(See Example \ref{ex:Y_1} in this connection.)
It can be used to prove results about homology cylinders 
by considering, first, the case of the trivial cylinder  $\Sigma_{g,b}\times [-1,1]$ 
and by studying, next, how things change under $Y_1$-surgery.
For instance, Proposition \ref{prop:triple-cup} and Proposition \ref{prop:triple-cup_closed} can be proved in that way.

For all $k\ge 1$, we define a submonoid of $\cyl(\Sigma_{g,b})$ by
\begin{gather*}
 Y_k \cyl(\Sigma_{g,b}) :=\left\{M:\text{$M$ is $Y_k$-equivalent to $\Sigma_{g,b}\times [-1,1]$}\right\}.
\end{gather*}
Thus, by Proposition \ref{prop:Y_Y}, we get a decreasing sequence of monoids
\begin{gather*}
 \cyl(\Sigma_{g,b})=Y_1\cyl(\Sigma_{g,b})\supset Y_2\cyl(\Sigma_{g,b}) \supset Y_3\cyl(\Sigma_{g,b}) \supset \cdots 
\end{gather*}
which is called\index{Y-filtration@$Y$-filtration} the \emph{$Y$-filtration}. 
For $b=0$ and $b=1$, it is finer than the Johnson filtration:
\begin{equation}
\label{eq:Y_to_Johnson}
\forall k\geq 1, \ Y_k \cyl_{g,1} \subset \cob_{g,1}[k]
\quad \hbox{and} \quad
\forall k\geq 1, \ Y_k \cyl_{g} \subset \cob_{g}[k],
\end{equation}
as follows readily from Proposition \ref{prop:nilpotent_quotients}.

It is easy to see that, for all $k\geq 1$, the set $\cyl(\Sigma_{g,b})/Y_k$ forms a monoid,
with submonoid $Y_i \cyl(\Sigma_{g,b})/Y_k$ for all $1 \leq i \leq k$.
In fact we have the following.

\begin{theorem}[See  \cite{Goussarov,Habiro}]
\label{th:Y_groups}
For all $k\ge 1$, the monoid $\cyl(\Sigma_{g,b})/Y_k$ is a finitely-generated, nilpotent group.  
Moreover, the submonoids $Y_i\cyl(\Sigma_{g,b})/Y_k$ for $1\le i\le k$ are subgroups and  satisfy
\begin{gather*}
\big[Y_i\cyl(\Sigma_{g,b})/Y_k\thinspace,\thinspace Y_j\cyl(\Sigma_{g,b})/Y_k\big] \subset Y_{\min(i+j,k)}\cyl(\Sigma_{g,b})/Y_k.
\end{gather*}
In particular, $Y_i\cyl(\Sigma_{g,b})/Y_k$ is abelian if $k/2\le i\le k$.
\end{theorem}

\noindent
We use claspers in this proof.
The structure of the finitely-generated abelian group $Y_i\cyl(\Sigma_{g,b})/Y_{i+1}$
is discussed in Section \ref{sec:Lie_ring_homology_cylinders} for $b=0$ and $b=1$.

Theorem \ref{th:Y_groups} suggests  to complete the monoid $\cyl(\Sigma_{g,b})$ as follows:
\begin{gather*}
  \widehat{\cyl}(\Sigma_{g,b}):=\varprojlim_{k\rightarrow\infty }\cyl(\Sigma_{g,b})/Y_k.
\end{gather*}
This is a group by Theorem \ref{th:Y_groups}, 
which is called\index{homology cylinder!group of} the {\em group of homology cylinders}.  
For all $i\geq 1$, we also set
\begin{gather*}
  Y_i\widehat{\cyl}(\Sigma_{g,b}):=\varprojlim_{k\rightarrow\infty }Y_i\cyl(\Sigma_{g,b})/Y_k,
\end{gather*}
to obtain a decreasing sequence of subgroups 
\begin{gather*}
 \widehat{\cyl}(\Sigma_{g,b})=Y_1\widehat{\cyl}(\Sigma_{g,b})\supset 
Y_2\widehat{\cyl}(\Sigma_{g,b})\supset Y_3\widehat{\cyl}(\Sigma_{g,b})\supset \cdots 
\end{gather*}
This is an $N$-series by Theorem \ref{th:Y_groups}, \ie we have
$$
\forall i,j \geq 1, \quad
\left[Y_i\widehat{\cyl}(\Sigma_{g,b}), Y_j\widehat{\cyl}(\Sigma_{g,b})\right] \subset 
Y_{i+j}\widehat{\cyl}(\Sigma_{g,b}).
$$
It is conjectured that the canonical map $\cyl(\Sigma_{g,b}) \to \widehat{\cyl}(\Sigma_{g,b})$ is injective,
which would imply that the $Y$-filtration on $\cyl(\Sigma_{g,b})$ is separating.
According to \cite{Massuyeau_DSP}, this injectivity is equivalent to the fact that finite-type invariants
(in the sense of Goussarov and the first author) distinguish homology cylinders.

\subsection{$Y_k$-equivalence and the Torelli group}

\label{subsec:Y_Torelli}

The $Y_k$-equivalence can be defined also in terms of cut-and-paste operations.
More precisely, we call a {\em Torelli surgery of class $k$}\index{Torelli surgery} the operation which consists
in cutting a $3$-manifold $M$ along a compact, connected, oriented surface $\Sigma \subset M$,  
and in re-gluing with an element of the $k$-th lower central series subgroup $\Gamma _k\Torelli (\Sigma)$ of the Torelli group.
It is easily seen that we can restrict in that definition either to closed surfaces $\Sigma$ 
which bound handlebodies, or to surfaces  $\Sigma$ with a single boundary component.

\begin{theorem}[See \cite{Habiro}]
\label{th:Y_Torelli}
Let $k\ge 1$.  Two compact oriented $3$-manifolds are $Y_k$-equivalent 
if and only if they are related by a Torelli surgery of class $k$.
\end{theorem}

\noindent
We refer to \cite{Massuyeau_DSP} for a proof. Since the Torelli group of $\Sigma_{g,b}$ 
is residually nilpotent (at least for $b=0$ and $b=1$), 
Theorem \ref{th:Y_Torelli} somehow supports the conjecture that the $Y_\infty$-equivalence could separate $3$-manifolds.

\begin{example}[Degree $1$]
An element of $\Torelli (\Sigma_{g,b})$  is called a {\em BP (bounding pair) map}\index{bounding pair map}\index{map!bounding pair}
if it is represented by a pair of Dehn
twists along two co-bounding simple closed curves in $\Sigma_{g,b}$, where the direction of the Dehn twists are opposite to
each other.  Johnson  proved that if $g\ge3$ then $\Torelli (\Sigma _{g})$ is generated by genus $1$ BP maps \cite{Johnson_generation}.  
Besides, the Torelli surgery defined by a genus $1$ BP map is equivalent to a $Y_1$-surgery as shown in Figure \ref{fig:BP}.
(See \cite{Massuyeau_DSP} for instance.)
These two facts prove Theorem \ref{th:Y_Torelli} in the case $k=1$.
\end{example}

\begin{figure}[h]
\begin{center}
{\labellist \small \hair 0pt
\pinlabel {$\longleftrightarrow$} at 329 80
\pinlabel {\scriptsize right } [l] at 157 63
\pinlabel {\scriptsize Dehn twist} [l] at 163 50
\pinlabel {\scriptsize left} [l] at 12 30
\pinlabel {\scriptsize Dehn twist} [bl] at 12 12
\endlabellist}
\includegraphics[scale=0.5]{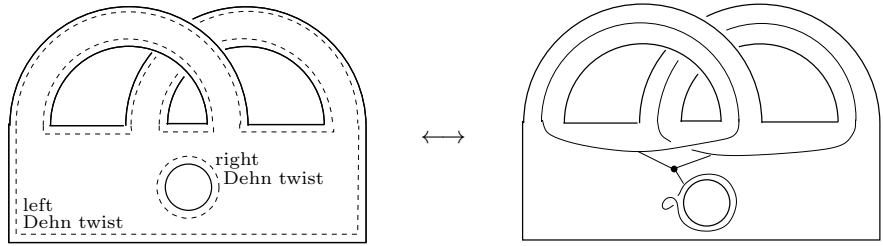}
\end{center}
\caption{The mapping cylinder of the BP map  defined by the boundary curves of $\Sigma_{1,2}$
can be obtained from $\Sigma_{1,2}\times [-1,1]$ by a $Y_1$-surgery.}
\label{fig:BP}
\end{figure}

By Theorem \ref{th:Y_Torelli}, we have the following.

\begin{corollary}
\label{cor:Y_lcs}
For all $k\ge 1$ and for all $f \in \Gamma _k\Torelli (\Sigma_{g,b})$, the mapping cylinder $\mcyl (f)$
is $Y_k$-equivalent to the trivial cylinder $\Sigma_{g,b} \times [-1,1]$.  
Therefore, the map $\mcyl \colon\thinspace\Torelli (\Sigma_{g,b})\to \cyl(\Sigma_{g,b})$ is a morphism of filtered monoids, i.e.,we have 
$$
\mcyl (\Gamma _k\Torelli (\Sigma_{g,b}))\subset Y_k\cyl(\Sigma_{g,b})
$$ 
for all integer $k\ge 1$.
\end{corollary}

Let us now assume that $b=0$ or $b=1$.
Then, we have three filtrations on the Torelli group: the lower central series, 
the restriction of the $Y$-filtration and the Johnson filtration. 
They are related in the following way:
\begin{equation}
\label{eq:hierarchy}
\forall k\geq 1, \quad
\Gamma_k \Torelli(\Sigma_{g,b}) \subset  \mcyl ^{-1}\left( Y_k\cyl(\Sigma_{g,b}) \right)
\subset \mcg(\Sigma_{g,b})[k].
\end{equation}
In the case $b=1$, a ``stable'' version of the lower central series can also be defined.
For this, we fix surface inclusions
$\Sigma_{0,1} \subset \Sigma_{1,1} \subset \Sigma_{2,1}\subset \cdots$, 
which induce a sequence of group injections
$$
\xymatrix{
\Torelli_{0,1}\ \ar@{>->}[r] & \Torelli_{1,1}\ \ar@{>->}[r] & \Torelli_{2,1}\ \ar@{>->}[r] & \cdots
}
$$
Thus, the Torelli group $\Torelli_{g,1}$ can be regarded as a subgroup of the direct limit
of groups $\varinjlim_{h} \Torelli_{h,1}$. Then, we define
$$
\forall k\geq 1, \quad \Gamma_k^{\operatorname{stab}} \Torelli_{g,1} 
:= \Torelli_{g,1} \cap \Gamma_k\!\!\! \varinjlim_{h\rightarrow\infty } \Torelli_{h,1}
= \Torelli_{g,1}   \cap \varinjlim_{h\rightarrow\infty } \Gamma_k \Torelli_{h,1}.
$$
We can similarly define a ``stable'' version of the $Y$-filtration on $\cyl_{g,1}$ by
$$
\forall k\geq 1, \quad Y_k^{\operatorname{stab}} \cyl_{g,1} 
:= \cyl_{g,1} \cap \varinjlim_{h\rightarrow\infty } Y_k \cyl_{h,1}
$$
where $\cyl_{g,1}$ is regarded as a submonoid of the monoid $\varinjlim_{h} \cyl_{h,1}$.
However, it can be proved that $Y_k^{\operatorname{stab}} \cyl_{g,1} = Y_k \cyl_{g,1}$ for all $k\geq 1$ \cite{HMass}.
Therefore, the stable lower central series of $\Torelli_{g,1}$ sits between the lower central series
and the restriction of the $Y$-filtration in the hierarchy (\ref{eq:hierarchy}).

\begin{conjecture}
The lower central series of the Torelli group $\Torelli_{g,1}$ stably coincides
with the restriction of the $Y$-filtration, \ie we have
$$
\forall k\geq 1, \quad \Gamma_k^{\operatorname{stab}} \Torelli_{g,1} =  \mcyl ^{-1}\left( Y_k\cyl_{g,1} \right).
$$
\end{conjecture}

\noindent
(See  also Problem \ref{pb:gr_mapping_cylinder} in this connection.)

\section{The Lie ring of homology cylinders}

\label{sec:Lie_ring_homology_cylinders}

In this section, we consider the graded Lie ring induced by the $Y$-filtration of $\cyl(\Sigma_{g,b})$,
and we discuss its relation with the graded Lie ring induced by the lower central series of $\Torelli(\Sigma_{g,b})$.
When $b=0$ or $b=1$, 
this graded Lie ring can be computed with rational coefficients,
using the LMO homomorphism (Section \ref{sec:LMO}) and claspers (Section \ref{sec:Y}).
Besides, the degree $1$ part  can be completely described in terms of a few invariants.

\subsection{Definition of the Lie ring of homology cylinders}

\label{subsec:Lie_ring_homology_cylinders}

Recall that an \emph{$N$-series}\index{N-series@$N$-series} $F$ on a group $G$ is a decreasing sequence
$G=F_1 G \supset F_2 G \supset F_3 G \supset \cdots$  of subgroups of $G$ such that
$\left[F_i G,F_j G\right] \subset F_{i+j} G$ for all $i,j \geq 1$.
Following Lazard \cite{Lazard}, we define the graded Lie ring \emph{induced} by $F$ as the graded abelian group
$$
\Gr^F G := \bigoplus_{i\geq 1} \frac{F_iG}{F_{i+1} G}
$$
with the Lie bracket defined by
$$
\big[\{g_i\},\{g_j\}\big] := \{\left[g_i,g_j\right]\} \in F_{i+j} G/ F_{i+j+1}G
$$
on homogeneous elements $\{g_i\} \in F_{i}G/F_{i+1} G$ and $\{g_j\} \in F_{j} G/F_{j+1} G$.
(One can check that the pairing is well-defined and satisfies the axioms of a Lie bracket using the Hall--Witt identities.)
For instance, this construction is well-known when $F=\Gamma$ is the lower central series of $G$,
and we have already met it in the previous sections for the free group.

Let us come back to the compact connected oriented surface $\Sigma_{g,b}$.
According to \S \ref{subsec:Y_cylinders}, we can apply the previous construction, for all $k\geq 1$, 
to the group $\cyl(\Sigma_{g,b})/Y_{k}$ equipped with the $Y$-filtration.
We then obtain the following construction  \cite{Habiro}.

\begin{definition}
The \emph{Lie ring of homology cylinders}\index{homology cylinder!Lie ring of}
is the graded Lie ring
$$
\Gr^Y \cyl(\Sigma_{g,b}) := \bigoplus_{i\geq 1} Y_i \cyl(\Sigma_{g,b})/Y_{i+1}.
$$
\end{definition}

Since the mapping cylinder construction sends the lower central series to the $Y$-filtration (Corollary \ref{cor:Y_lcs}),
it induces a map at the graded level:
$$
\Gr \mcyl: \Gr^\Gamma \Torelli(\Sigma_{g,b}) \longrightarrow \Gr^Y \cyl(\Sigma_{g,b}).
$$
The map $\Gr \mcyl$ is at the heart of the interactions between mapping class groups and finite-type invariants of $3$-manifolds.

\begin{problem}
\label{pb:gr_mapping_cylinder}
Determine whether $\Gr \mcyl$ is injective, and characterize its image. 
\end{problem}

\noindent
This problem is difficult to solve in general.
We shall review, in the following subsections, some pieces of answer.
As a preliminary remark,  let us observe that the representation theory of the symplectic group may help.
Indeed, we have the following lemma, which will be used in the sequel.

\begin{lemma}
\label{lem:Sp_action}
For $b=0$ and $b=1$, 
the conjugation action of $\mcg(\Sigma_{g,b})$ on itself and on $\cob(\Sigma_{g,b})$ 
induces an action of the symplectic group $\Sp(H)$ on  $\Gr^\Gamma \Torelli(\Sigma_{g,b})$
and on $\Gr^Y \cyl(\Sigma_{g,b})$, respectively. Moreover, the map $\Gr \mcyl$ is $\Sp(H)$-equivariant.
\end{lemma}

\begin{proof}
The conjugation action of $\mcg(\Sigma_{g,b})$ on $\cob(\Sigma_{g,b})$ is defined by
$$
\mcg(\Sigma_{g,b}) \times \cob(\Sigma_{g,b}) \ni \
(f,M) \longmapsto \mcyl(f) \circ M \circ \mcyl(f^{-1}) \ \in \cob(\Sigma_{g,b})
$$
and, clearly, it preserves the submonoid $\cyl_{g,b}$.
The $Y_i$-equivalence being generated by surgeries along $Y_i$-graphs,
this action also preserves the submonoid $Y_i \cyl(\Sigma_{g,b})$ of $\cyl(\Sigma_{g,b})$. 
Therefore, the group $\mcg(\Sigma_{g,b})$ acts on $\Gr^Y \cob(\Sigma_{g,b})$.
But, we also deduce from Theorem \ref{th:Y_groups} that 
$$
\forall f \in \Torelli(\Sigma_{g,b}), \ \forall M \in Y_i  \cyl(\Sigma_{g,b}),\ \mcyl(f) \circ M \circ \mcyl(f^{-1}) \sim_{Y_{i+1}} M. 
$$
So, the action of $\mcg(\Sigma_{g,b})$ on $\Gr^Y \cob(\Sigma_{g,b})$ factorizes to $\mcg(\Sigma_{g,b})/\Torelli(\Sigma_{g,b}) \simeq \Sp(H)$. 
The action of $\Sp(H)$ on $\Gr^\Gamma \Torelli(\Sigma_{g,b})$ is similarly defined, 
and the $\Sp(H)$-equivariance of $\Gr \mcyl$ is then obvious.
\end{proof}

One way to simplify the study of the Lie ring of homology cylinders 
is to consider it with rational coefficients:
$$
\Gr^Y \cyl(\Sigma_{g,b}) \otimes \Q = \bigoplus_{i\geq 1} \frac{Y_i \cyl(\Sigma_{g,b})}{Y_{i+1}} \otimes \Q.
$$ 
This Lie algebra is computed in the next subsection for $b=0$ and $b=1$.

\subsection{The Lie algebra of homology cylinders}

\label{subsec:Lie_algebra_homology_cylinders}

\index{homology cylinder!Lie algebra of}
We start with the case of a bordered surface $\Sigma_{g,1}$.
The following $\Q$-vector space was introduced in \cite{Habiro}. 
$$
\A^{<}(H_\Q) := 
\frac{\Q\cdot \left\{ \begin{array}{c} \hbox{Jacobi diagrams without strut component, and with}\\
\hbox{external vertices colored by } H_\Q \hbox{ and totally ordered}  \end{array} \right\}}
{\hbox{AS, IHX, multilinearity, STU-like}}.
$$
The AS, IHX and multilinearity relations are as defined in \S \ref{subsec:symplectic_Jacobi_diagrams}.
The ``STU-like'' relation is defined as follows:
$$
\begin{array}{c}
\labellist \small \hair 2pt
\pinlabel {$=\quad \quad \omega(x,y)\cdot$} at 395 50
\pinlabel {$x$} [t] at 2 0
\pinlabel {$y$} [t] at 58 0
\pinlabel {$y$} [t] at 203 0
\pinlabel {$x$} [t] at 256 0
\pinlabel {$-$} at 130 45
\pinlabel {$<$} at 30 4
\pinlabel {$<$} at 229 4
\pinlabel {$\cdots<$} [r] at 0 4
\pinlabel {$\cdots<$} [r] at 198 4
\pinlabel {$< \cdots$} [l] at 61 4
\pinlabel {$< \cdots$} [l] at 261 4
\endlabellist
\centering
\includegraphics[scale=0.5]{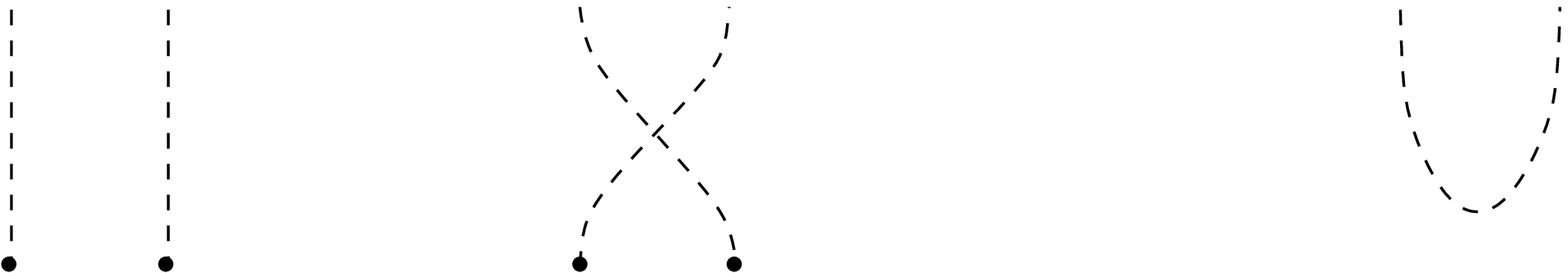}
\end{array}
$$
(Recall that $\omega:H_\Q \times H_\Q \to \Q$ denotes the intersection pairing.)
The space $\A^{<}(H_\Q)$ can be equipped with a multiplication $\osqcup$:
the product $D \osqcup E$ of two Jacobi diagrams $D$ and $E$ is defined as their disjoint union $D \sqcup E$,
the external vertices of $E$ being considered as ``higher'' than those of $D$. The identity element is the empty diagram $\varnothing$.
Similarly to  $\A(H_\Q)$, the space $\A^{<}(H_\Q)$ can also be given a comultiplication, a counit and an antipode
so that $\A^{<}(H_\Q)$ is a Hopf algebra. The obvious action of $\Sp(H_\Q)$ on $\A^{<}(H_\Q)$ preserves all this structure.
The Hopf algebra $A^{<}(H_\Q)$ can be regarded as a ``non-symmetric'' version of $\A(H_\Q)$.

\begin{proposition}[See \cite{HMass}]
The ``symmetrization'' map $\chi: \A(H_\Q) \to \A^{<}(H_\Q)$ defined, 
for all Jacobi diagrams $D$ with $e$ external vertices, by 
$$
\chi(D):= \frac{1}{e!} \cdot \left(\hbox{sum of all ways of ordering the $e$ external vertices of $D$} \right)
$$
is a Hopf algebra isomorphism and is $\Sp(H_\Q)$-equivariant.
\end{proposition}

\noindent
This is similar to the fact that, for a symplectic vector space $(V,s)$, the ``symmetrization'' map
defines a linear isomorphism between the symmetric algebra generated by $V$,
and the Weyl algebra defined by $(V,s)$. Thus, the multiplication $\star$ on $\A(H_\Q)$ defined in \S \ref{subsec:symplectic_Jacobi_diagrams}
can be regarded as a diagrammatic analogue of the Moyal--Weyl product on $S(V)$. 
This analogy can be made explicit by considering ``weight systems'' associated to quadratic Lie algebras \cite{HMass}.

The primitive part of $\A^<(H_\Q)$ is the subspace spanned by connected Jacobi diagrams, 
which we denote by $\A^{<,c}(H_\Q)$. Equipped with the Lie bracket $[\centereddot,\centereddot]_{\osqcup}$, this is a Lie algebra.
Of course, $\A^{<,c}(H_\Q)$ is isomorphic to $\A^{c}(H_\Q)$ but, from the point of view of calculus of claspers,
the former is more natural than the latter.

\begin{theorem}[See \cite{Habiro}]
\label{th:surgery}
\index{surgery map}\index{map!surgery}
Surgery along graph claspers defines a canonical map
$$
\psi: \A^{<,c}(H_\Q)  \longrightarrow \Gr^Y \cyl_{g,1} \otimes \Q
$$
of graded Lie algebras, which is surjective and $\Sp(H)$-equivariant.
\end{theorem}

\begin{proof}[Sketch of the proof]
The map $\psi$ sends each connected Jacobi diagram $D\in  \A^{<,c}(H_\Q)$ 
to the $3$-manifold obtained from the trivial cylinder $\Sigma_{g,1} \times [-1,1]$
by surgery along a graph clasper $C(D)$ obtained from $D$ as follows:
\begin{itemize}
\item Thicken $D$ to an oriented surface using the vertex-orientation of $D$
(vertices are thickened to disks, and edges to bands). 
Cut a smaller disk in the interior of each disk
that has been produced from an external vertex of $D$.
This leads to an oriented compact surface $S(D)$,
decomposed into disks, bands and annuli (corresponding to internal
vertices, edges and external vertices of $D$ respectively).
Use the induced orientation on $\partial S(D)$ to orient the cores of the annuli.
\item Next, embed $S(D)$ into the interior of $\Sigma_{g,1} \times [-1,1]$ in such a way that
each annulus of $S(D)$ represents in $H_\Q$ the color of the corresponding external vertex of $D$.
Moreover, the annuli should be in disjoint ``horizontal slices'' of $\Sigma_{g,1} \times [-1,1]$
and their ``vertical height'' along $[-1,1]$  should respect the total ordering of the external vertices of $D$.
Such an embedding defines a graph clasper $C(D)$ in $\Sigma_{g,1} \times [-1,1]$.
\end{itemize}
That $\psi$ is well-defined and surjective can be shown by using claspers \cite{Habiro,Goussarov_clovers,GGP}. 
By using the same techniques, one can also check that $\psi$ is a Lie algebra homomorphism.
See also \cite{GL}, \cite{Garoufalidis} and \cite{Habegger} for similar constructions.

To check the $\Sp(H)$-equivariance of $\psi$, consider an $F\in \Sp(H)$ and an $f \in \mcg_{g,1}$ such that $f_*=F$.
Then, the image of $C(D)$ by the homeomorphism $f\times \Id$ of $\Sigma_{g,1} \times [-1,1]$ 
can play the role of $C(F\cdot D)$. So, the class $\psi(F\cdot D)$ is represented by the cobordism
$$
 (\Sigma_{g,1} \times [-1,1])_{(f\times \Id)(C(D))} 
= \mcyl(f) \circ  (\Sigma_{g,1} \times [-1,1])_{C(D)} \circ \mcyl(f^{-1})
$$
and we conclude that $\psi(F\cdot D) = F \cdot \psi(D)$.
\end{proof}

\begin{remark}
\label{rem:integral_coefficients}
The algebra $\A^{<}(H_\Q)$ can also be defined with integral coefficients, which results in a ring $\A^{<}(H)$.
In degree $>1$, the surgery map $\psi$ also exists with integral coefficients, 
and the resulting map $\psi: \A^{<,c}_{\geq 2}(H) \to  \Gr^Y_{\geq 2} \cyl_{g,1}$ is surjective.
The degree $1$ case is special, and this is the subject of \S \ref{subsec:degree_one}.
\end{remark}

The Lie algebra of homology cylinders can now be described in a diagrammatic way.
A similar result is proved by Habegger in \cite{Habegger}, 
but just for vector spaces (disregarding Lie algebra structures).

\begin{theorem}[See \cite{CHM, HMass}]
\label{th:LMO_iso}
The LMO homomorphism $Z: \cyl_{g,1} \to \A(H_\Q)$ induces a map
$\Gr Z: \Gr^Y \cyl_{g,1} \otimes \Q \to  \A^c(H_\Q)$ at the graded level. 
Moreover, we have the following commutative triangle 
in the category  of graded Lie algebras with $\Sp(H_\Q)$-actions:
$$
\xymatrix{
\Gr^Y \cyl_{g,1} \otimes \Q \ar[r]^-{\Gr Z}_-\simeq & \A^c(H_\Q) \ar[d]^-\chi_-\simeq \\
& \A^{<,c}(H_\Q) \ar[lu]^-\psi_-\simeq 
}
$$
\end{theorem}

\noindent
Since the definition of the surgery map $\psi$ has not required prior choices,
it follows from this diagram that the map $\Gr Z$ is canonical.
In particular, it does not depend on the choice of the Drinfeld associator,
nor on the system of meridians and parallels $(\alpha, \beta)$ 
which has been fixed on the surface $\Sigma_{g,1}$ (see Figure \ref{fig:surface}).

\begin{proof}[Sketch of the proof]
Recall from \S \ref{subsec:definition_LMO} that, from the viewpoint of the LMO homomorphism $Z$,
the natural space to work with is $\A^Y\left(\set{g}^+ \cup \set{g}^-\right)$
since $Z$ was defined as $Z = \kappa \circ \widetilde{Z}^Y$ 
and $ \widetilde{Z}^Y$ takes values in this space.
An explicit formula for $\chi^{-1}$ (see \cite{HMass}) shows that
the isomorphism $\kappa: \A^Y\left(\set{g}^+ \cup \set{g}^-\right) \to \A(H_\Q)$ can be decomposed as
$$
\xymatrix{\A^Y\left(\set{g}^+ \cup \set{g}^-\right)  \ar[r]^-\varphi_-\simeq &
\A^<(-H_\Q) \ar[r]^-s_-\simeq & 
\A^<(H_\Q)  \ar[r]^-{\chi^{-1}}_-\simeq & 
\A(H_\Q).
}
$$
Here, the space $\A^<(-H_\Q)$ is defined exactly as $\A^<(H_\Q)$ except that the symplectic form $-\omega$
replaces $\omega$ in the STU-like relation; the map $\varphi$ sends a Jacobi diagram $E$ to the diagram $E'$
obtained by changing the colors with the rules $i^-\mapsto [\alpha_i]$ and $i^+ \mapsto [\beta_i]$
and by  declaring that every $i^-$-colored vertex is lower than any $i^+$-colored vertex;
finally the map $s$ is simply defined by $s(D)=(-1)^{\chi(D)}\cdot D$ for any Jacobi diagram $D$
with Euler characteristic $\chi(D)$.

Let $M\in \cyl_{g,1}$ and let $C$ be a graph clasper in $M$ of degree $v\geq 1$. 
A computation of the LMO invariants \cite{CHM}
shows  that the infinite sum of Jacobi diagrams $\widetilde{Z}^Y(M)-\widetilde{Z}^Y(M_C)$ starts in degree $v$. 
Therefore, $\widetilde{Z}^Y$ induces a map
$$
\Gr \widetilde{Z}^Y: \Gr^Y \cyl_{g,1} \otimes \Q \longrightarrow \A^{Y}\left(\set{g}^+ \cup \set{g}^-\right)
$$
which sends $\{M\} \in Y_v \cyl_{g,1}/Y_{v+1}$ to the degree $v$ part of $\widetilde{Z}^Y(M)$.
Thus, $Z$ induces a map $\Gr Z: \Gr^Y \cyl_{g,1} \otimes \Q \to \A(H_\Q)$ in the same way.

Let $E\in \A^Y\left(\set{g}^+ \cup \set{g}^-\right)$ 
be a connected Jacobi diagram having $v$ internal vertices, and $e$ internal edges.
This  diagram $E$ has a topological realization $C(E)$, 
which is the topological realization $C(E')$ of $E'$ defined in the proof of Theorem \ref{th:surgery}.
A more accurate computation of the LMO invariants \cite{CHM} shows that
$$
\widetilde{Z}^Y\left(\Sigma_{g,1}\times [-1,1]\right) - \widetilde{Z}^Y\left((\Sigma_{g,1}\times [-1,1])_{C(E)}\right)
= (-1)^{v+e+1} \cdot E + (\deg >v).
$$
Since $v-e$ is the Euler characteristic of $E$, this identity is equivalent to
$$
s\varphi \widetilde{Z}^Y\left( (\Sigma_{g,1}\times [-1,1])_{C(E')}\right)- \varnothing = E' + (\deg >v).
$$
We deduce that $\Gr Z \circ \psi = \chi^{-1}$. It follows from the surjectivity of $\psi$ and the bijectivity of $\chi$ 
that both $\psi$ and $\Gr Z$ are isomorphisms, and that $\Gr Z$ takes values in $\A^c(H_\Q)$.

Finally, since $\psi$ is $\Sp(H)$-equivariant,
the natural action of $\Sp(H_\Q)$ on $\A^{<,c}(H_\Q)$ can be transported by the isomorphism $\psi$
to an action on $\Gr^Y \cyl_{g,1} \otimes \Q$, which is compatible with the action of $\Sp(H)$ given by Lemma \ref{lem:Sp_action}.
The maps $\psi$ and $\chi$ are then $\Sp(H_\Q)$-equivariant, and so is $\Gr Z$.
\end{proof}

We now consider the mapping cylinder construction at the graded level and with rational coefficients:
$$
\Gr \mcyl \otimes \Q: \Gr^\Gamma \Torelli_{g,1} \otimes \Q \longrightarrow \Gr^Y \cyl_{g,1} \otimes \Q.
$$
On one hand, the target has the diagrammatic description given by Theorem \ref{th:LMO_iso}.
On the other hand, the source has been computed by Hain as we shall now recall.
For $g\geq 3$, the degree $1$ part of $\Gr^\Gamma \Torelli_{g,1} \otimes \Q$ has been computed by Johnson \cite{Johnson_abelianization}:
the first Johnson homomorphism $\tau_1$ induces an isomorphism between $\left(\Torelli_{g,1}/\Gamma_2 \Torelli_{g,1}\right) \otimes \Q$
and $\Lambda^3 H_\Q$. Since $\Gr^\Gamma \Torelli_{g,1} \otimes \Q$ is generated by its degree $1$ part, we get a surjective Lie algebra map
$$
J: \Lie\left(\Lambda^3 H_\Q\right) \longrightarrow \Gr^\Gamma \Torelli_{g,1} \otimes \Q
$$
defined on the free Lie algebra generated by $\Lambda^3 H_\Q$.
This map provides a quadratic or a cubic presentation of the Lie algebra  $\Gr^\Gamma \Torelli_{g,1} \otimes \Q$.

\begin{theorem}[Hain \cite{Hain}]
\label{th:Hain}
The ideal of relations
$$
\operatorname{R}(\Torelli_{g,1}) = \bigoplus_{i\geq 1} \operatorname{R}_i(\Torelli_{g,1}) := \Ker J
$$
is generated by $\operatorname{R}_2(\Torelli_{g,1})$ if $g\geq 6$, 
and by  $\operatorname{R}_2(\Torelli_{g,1}) + \operatorname{R}_3(\Torelli_{g,1})$ if $g=3,4,5$.
\end{theorem}

\noindent
Let $Y: \Lie(\Lambda^3 H_\Q) \to \A^c(H_\Q)$ be the Lie algebra map defined, in degree $1$, by 
$$
Y:\ x_1 \wedge x_2 \wedge x_3 \longmapsto \Ygraphbottoptop{x_1}{x_2}{x_3}.
$$ 
According to Corollary \ref{cor:LMO_to_Johnson}, the degree $1$ part of $\Gr Z$ 
coincides with $\tau_1$ by the identification $\Lambda^3 H_\Q \simeq \A_1^c(H_\Q)= \A_1^{t,c}(H_\Q)$ observed in Example \ref{ex:degree_one}.
Therefore, the following diagram commutes in degree $1$ and, so, commutes in any degree:
\begin{equation}
\label{eq:Hain_to_LMO}
\xymatrix{
\Gr^\Gamma \Torelli_{g,1} \otimes \Q \ar[r]^-{\Gr \mcyl \otimes \Q} & \Gr^Y \cyl_{g,1} \otimes \Q \ar[d]_-{\Gr Z} \\
\Lie(\Lambda^3 H_\Q)/ \operatorname{R}(\Torelli_{g,1}) \ar[u]^-{\overline{J}}_-\simeq \ar[r]_-{\overline{Y}} & 
\A^c(H_\Q) \ar@/_1pc/[u]_-{\psi \circ \chi}^-\simeq
}
\end{equation}
Here, the arrows $\overline{J}$ and $\overline{Y}$ denote quotients of the maps $J$ and $Y$ respectively.
All maps are Lie algebra homomorphisms and are $\Sp(H_\Q)$-equivariant.
Together with Theorem \ref{th:Hain}, the commutative square (\ref{eq:Hain_to_LMO}) 
gives an algebraic description of $\Gr \mcyl \otimes \Q$. This is a way of tackling Problem \ref{pb:gr_mapping_cylinder}.

For instance, Problem \ref{pb:gr_mapping_cylinder} can be solved in degree $2$ and with rational coefficients. 
On one hand, the quadratic relations of the Lie algebra  $\Gr^\Gamma \Torelli_{g,1} \otimes \Q$ can
be computed using the representation theory of the symplectic group.

\begin{proposition}[Hain \cite{Hain}, Habegger--Sorger \cite{HS}]
\label{prop:Hain}
If $g \geq 3$, then the $\Sp(H_\Q)$-module $\operatorname{R}_2(\Torelli_{g,1})$
is spanned by the following elements $r_1,r_2$ of $\Lie_2(\Lambda^3 H_\Q)$:
$$
\begin{array}{l}
r_1 :=\begin{cases}
\left[ \alpha_1 \wedge \alpha_2 \wedge \beta_2 , \alpha_3 \wedge
  \alpha_4 \wedge \beta_4  \right] 
&\text{if $g\ge4$},\\
0&\text{if $g=3$},
\end{cases}\\
r_2 := \left[ \alpha_1 \wedge \alpha_2 \wedge \beta_2 , \alpha_g \wedge \omega \right]
\quad \quad \quad \quad \;\;\text{if $g\ge3$}.
\end{array}
$$
\end{proposition}

\noindent
On the other hand, we have the following description of the Lie bracket of $\A^c(H_\Q)$ in degree $1+1$,
which is also obtained using the representation theory of $\Sp(H_\Q)$.

\begin{proposition}[See \cite{HMass}]
\label{prop:Y_2}
Let $g\geq 3$.
The image of $Y_2: \Lie_2(\Lambda^3 H_\Q)  \to \A_2^c(H_\Q)$
is the subspace spanned by the Theta graph and by the H graphs. 
Moreover, its kernel is $\Sp(H_\Q)$-spanned by $r_1,r_2$.
\end{proposition}

\noindent
It follows from those two propositions and from the commutative square (\ref{eq:Hain_to_LMO}) 
that $\Gr \mcyl \otimes \Q$ is injective in degree $2$. 
We also deduce that $\Gr \mcyl \otimes \Q$ is not surjective: Phi graphs are missing to the image.
This is a diagrammatic translation of Morita's results  \cite{Morita_Casson_1,Morita_Casson_2}: 
the quotient $(\Gamma_2 \Torelli_{g,1}/\Gamma_3 \Torelli_{g,1})\otimes \Q$
is determined by the second Johnson homomorphism (which corresponds to
the H graphs)
and by the Casson invariant (which corresponds to the Theta graph).\\

We now outline the case of a closed surface $\Sigma_g$.
Let $I^{<,c}$ be the subspace of $\A^{<,c}(H_\Q)$ spanned  by  sums of connected Jacobi diagrams of the form
$$
\begin{array}{c}
\labellist \small \hair 2pt
\pinlabel {${\displaystyle \quad \sum_{i=1}^g}$} [r] at -10 25
\pinlabel {$=:$} at 173 27 
\pinlabel {$\omega$} [t] at 249 0
\pinlabel {$\alpha_i$} [t] at 2 0
\pinlabel {$<$} at 18 5
\pinlabel {$< \cdots$} [l] at 42 4
\pinlabel {$< \cdots$} [l] at 254 4
\pinlabel {$\beta_i$} [t] at 36 0
\endlabellist
\centering
\includegraphics[scale=0.6]{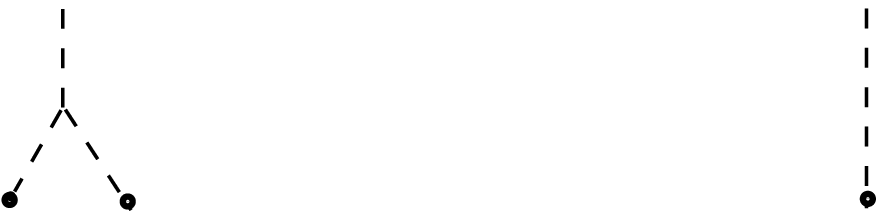}
\end{array}
$$ 
where only the lowest external vertices are shown.
This is an ideal of the Lie algebra $\A^{<,c}(H_\Q)$ which corresponds by $\chi^{-1}$ 
to the ideal $I^{c}$ introduced in \S \ref{subsec:symplectic_Jacobi_diagrams}.
Thus, we can consider the quotient Lie algebra
$$
\clocase{\A}^{<,c}(H_\Q) := \A^{<,c}(H_\Q)/I^{<,c}.
$$
Similarly to the bordered case, surgery along graph claspers defines a canonical map
$$
\psi: \clocase{\A}^{<,c}(H_\Q)   \longrightarrow \Gr^Y \cyl_{g} \otimes \Q
$$
of graded Lie algebras which is surjective and $\Sp(H)$-equivariant \cite{Habiro,HMass}.
Furthermore, if the surface $\Sigma_g$ is regarded as the union of $\Sigma_{g,1}$ with a closed disk $D$,
then there is the LMO homomorphism $Z: \cyl_g  \to \clocase{\A}(H_\Q)$.
The following is deduced from (\ref{eq:Z_Z}) and Theorem \ref{th:LMO_iso}.

\begin{theorem}[See \cite{HMass}]
\label{th:LMO_iso_closed}
The LMO homomorphism $Z: \cyl_{g} \to \clocase{\A}(H_\Q)$ induces a map
$\Gr Z: \Gr^Y \cyl_{g} \otimes \Q \to  \clocase{\A}^c(H_\Q)$ at the graded level. 
Moreover, the following prism in the category  of graded Lie algebras with $\Sp(H_\Q)$-actions is commutative:
$$
\xymatrix{
& \A^{<,c}(H_\Q) \ar@{->>}[rr]   \ar[ld]^-\simeq_-\psi & & \clocase{\A}^{<,c}(H_\Q) \ar[ld]^-\simeq_-\psi\\ 
\Gr^Y \cyl_{g,1} \otimes \Q  \ar[d]_-{\Gr Z}^-\simeq \ar@{->>}[rr]^-{\clocase{\centereddot}} & & \Gr^Y \cyl_{g} \otimes \Q \ar[d]_-{\Gr Z}^-\simeq & \\
\A^c(H_\Q) \ar@{->>}[rr]  \ar@/_1pc/[ruu]^(.4){\simeq}_(.4)\chi |!{[u];[urr]}\hole &  & \clocase{\A}^c(H_\Q) \ar@/_1pc/[ruu]^-\simeq_-\chi &
}
$$
\end{theorem}

\noindent
It follows from this diagram that the map $\clocase{\centereddot}: \Gr^Y \cyl_{g,1} \otimes \Q \to \Gr^Y \cyl_{g} \otimes \Q$ depends only 
on the map $H_1(\Sigma_{g,1};\Q) \to H_1(\Sigma_{g};\Q)$ induced by the inclusion $\Sigma_{g,1} \subset \Sigma_g$.
(By Remark \ref{rem:integral_coefficients}, the same observation holds with integral coefficients.)

In the closed case, too, the map $\Gr \mcyl \otimes \Q$ has an algebraic description obtained
by combining Hain's results \cite{Hain} to Theorem \ref{th:LMO_iso_closed}.
For $g\geq 3$, the first Johnson homomorphism induces an isomorphism between $\left(\Torelli_g/\Gamma_2 \Torelli_g\right)\otimes \Q$
and $\Lambda^3 H_\Q/ \omega \wedge H_\Q$, hence a Lie algebra surjective map
$$
J: \Lie\left(\Lambda^3 H_\Q/ \omega \wedge H_\Q\right) \longrightarrow \Gr^\Gamma \Torelli_g \otimes \Q.
$$
The analogue of Theorem \ref{th:Hain} holds in the closed case as well \cite{Hain}:
depending on the genus $g$, the ideal of relations $\operatorname{R}(\Torelli_g):= \Ker J$ is quadratic or cubic.
In this context,  we have the following commutative diagram in the category of Lie algebras with $\Sp(H_\Q)$-actions:
\begin{equation}
\label{eq:Hain_to_LMO_closed}
\xymatrix{
\Gr^\Gamma \Torelli_{g} \otimes \Q \ar[r]^-{\Gr \mcyl \otimes \Q} & \Gr^Y \cyl_{g} \otimes \Q \ar[d]_-{\Gr Z} \\
\Lie(\Lambda^3 H_\Q/ \omega \wedge H_\Q)/ \operatorname{R}(\Torelli_{g}) \ar[u]^-{\overline{J}}_-\simeq \ar[r]_-{\overline{Y}} & 
\clocase{\A}^c(H_\Q) \ar@/_1pc/[u]_-{\psi \circ \chi}^-\simeq
}
\end{equation}

Thanks to this diagram, Problem \ref{pb:gr_mapping_cylinder} can be solved in degree $2$ and for rational coefficients.
Again, representation theory of the symplectic group is the key tool for computations.
Thus, Hain proved in \cite{Hain} that, for $g\geq 3$,  the degree $2$ part of $\operatorname{R}(\Torelli_g)$
is $\Sp(H_\Q)$-spanned by the class of the element $r_1$ defined in Proposition \ref{prop:Hain}. At the diagrammatic level,
we need to consider the Lie bracket of $\clocase{\A}^c(H_\Q)$ in degree $1+1$ or, equivalently, the map
$$
Y_2: \Lie_2\left(\Lambda^3 H_\Q/ \omega \wedge H_\Q\right)
\longrightarrow \clocase{\A}^c_2(H_\Q)= \A_2^c(H_\Q) / I_2^c.
$$
Observe that the subspace $I_2^c$ is spanned by elements of the form
$$
\Ygraphbottoptop{\omega}{y}{x} -\frac{\omega(x,y)}{4} \cdot \thetagraph\ , \quad \quad \forall x,y \in H_\Q.
$$
It can be deduced  from Proposition \ref{prop:Y_2} that the image of $Y_2$ is the subspace generated by (the classes of)
the Theta graph and the H graphs, and that the kernel of $Y_2$ is $\Sp(H_\Q)$-spanned by the class of $r_1$ \cite{HMass}.
It follows that $\Gr \mcyl \otimes \Q$ is injective in degree $2$ but, like in the bordered case, it is not surjective.

\subsection{The degree $1$ part of the Lie ring of homology cylinders}

\label{subsec:degree_one}

We now compute the abelian group $\cyl(\Sigma_{g,b})/Y_2$ for $b=0$ and $b=1$ \cite{Habiro,MM},
and we relate this to the abelianization of $\Torelli(\Sigma_{g,b})$ due to Johnson \cite{Johnson_abelianization}.
For this, we need some homomorphisms introduced by Birman and Craggs \cite{BC}
which are derived from the Rochlin invariant of spin $3$-manifolds.

We consider first the case of $\Sigma_{g,1}$.
The reader is referred to \cite{Milnor} or \cite{Kirby} for an introduction to spin structures. 
The \emph{Rochlin invariant}\index{Rochlin invariant}
of a closed connected oriented $3$-manifold $M$ equipped with a spin structure $\sigma$ is defined as
$$
R_M(\sigma) := \sgn(W) \mod 16
$$
where  $\sgn(W)$ is the signature of  a compact connected oriented smooth $4$-manifold $W$
which is bounded by $M$ and to which $\sigma$ can be extended.
That $R_M(\sigma) \in \Z_{16}$ is well-defined follows from Rochlin's theorem and 
Novikov's additivity of the signature --- see \cite{Kirby}. 
We need two facts about the \emph{Rochlin function} of $M$
$$
R_M: \Spin(M) \longrightarrow \Z_{16}.
$$
First, it is trivial modulo $8$ if $H_1(M)$ is torsion-free \cite{BM} and, second,
it is a cubic function with respect to the affine action of $H^1(M;\Z_2)$ on $\Spin(M)$ \cite{Turaev}.
More precisely, Turaev proved that the $3$-rd derivative of  $R_M$ at a point  $\sigma\in \Spin(M)$ 
$$
\diff^3_\sigma R_M: H^1(M;\Z_2) \times H^1(M;\Z_2)  \times H^1(M;\Z_2)  \longrightarrow \Z_{16}
$$
which, by definition, sends $(x_1,x_2,x_3)$ to
$$
R_M(\sigma) - \sum_{1\leq i \leq 3} R_M(\sigma+ \vec{x_i}) +
\sum_{1 \leq i<j \leq 3} R_M(\sigma + \vec{x_i} + \vec{x_j}) - R_M(\sigma + \vec{x_1} + \vec{x_2} + \vec{x_3})
$$
does not depend on $\sigma$ and coincides with the mod $2$ triple-cup product form:
\begin{equation}
\label{eq:Rochlin_to_cohomology_mod_2}
\diff^3_\sigma R_M(x_1,x_2,x_3) = 8 \cdot \left\langle x_1 \cup x_2 \cup x_3, [M]\right\rangle \in \Z_{16}.
\end{equation}

Let $C \in \cyl_{g,1}$ and consider the closed oriented $3$-manifold $\closure{C}$ defined in \S \ref{subsec:closure}. 
The inclusion $c_\pm:\Sigma_{g,1} \to C \subset \overline{C}$ gives an isomorphism
$c^*: H^1(\overline{C};\Z_2) \to H^1(\Sigma_{g,1};\Z_2)$ in cohomology
and, so, it induces an affine isomorphism $c^*: \Spin(\overline{C}) \to \Spin(\Sigma_{g,1})$.
Spin structures on $\Sigma_{g,1}$ can be described in the following algebraic way. Let 
$$
\mathcal{Q} := \left\{H\otimes \Z_2 \stackrel{q}{\longrightarrow} \Z_2: \forall x,y \in H\otimes \Z_2, \ 
q(x+y) - q(x) - q(y) = \omega(x,  y) \mod 2 \right\}
$$
be the set of \emph{quadratic forms} whose polar form is the intersection pairing  mod $2$.
This is an affine space over the $\Z_2$-vector space $H\otimes  \Z_2$, the action being given by 
$$
\forall x \in H\otimes \Z_2, \forall q\in \mathcal{Q}, \ q + \vec{x} := q + \omega(x,\centereddot). 
$$
As observed by Atiyah \cite{Atiyah} and Johnson \cite{Johnson_quadratic}, there is a canonical bijection
\begin{equation}
\label{eq:Atiyah-Johnson}
\Spin(\Sigma_{g,1}) \stackrel{\simeq}{\longrightarrow} \mathcal{Q}, \ \sigma \longmapsto q_\sigma
\end{equation}
which is affine over the isomorphism $H^1(\Sigma_{g,1};\Z_2) \simeq H_1(\Sigma_{g,1};\Z_2)$ induced by $\omega$.
For any simple oriented closed curve $\gamma \subset \Sigma_{g,1}$,
the quadratic form $q_\sigma$ sends $[\gamma]$ to the cobordism class of $(\gamma, \sigma\vert_{\gamma})$ in $\Omega^{\Spin}_1 \simeq \Z_2$.
In the sequel, we denote by  $B_{\leq d}$ the space of polynomial functions $\Spin(\Sigma_{g,1}) \to \Z_2$ of degree $\leq d$. 

\begin{definition}
The \emph{Birman--Craggs homomorphism}\index{homomorphism!Birman--Craggs}\index{Birman--Craggs homomorphism} is the map
$$
\beta: \cyl_{g,1} \longrightarrow B_{\leq 3}
$$
which sends a $C \in \cyl_{g,1}$ to the cubic function $\frac{1}{8} R_{\overline{C}} \circ c^{*,-1}$.
\end{definition}

\noindent
The Birman--Craggs homomorphism was originally defined (for the Torelli group) 
as a collection of many homomorphisms derived from the Rochlin invariant of homology $3$-spheres \cite{BC}.
Those homomorphisms were further studied by Johnson, who noticed that they are naturally indexed
by spin structures of the surface  
and can be unified into a single map $\beta: \Torelli_{g,1} \to B_{\leq 3}$ \cite{Johnson_BC}. 
The idea of defining $\beta$ using the mapping torus construction is due to Turaev \cite{Turaev}.

The fact that  $\beta$ is indeed a monoid homomorphism can be deduced from the following variation formula.

\begin{proposition}
\label{prop:variation_Rochlin}
Let $N$ be a closed connected oriented $3$-manifold.
Given a cobordism $C \in \cyl_{g,1}$ and an embedding $j:\Sigma_{g,1} \to N$,
consider the $3$-manifold $N'$ obtained by ``cutting'' $N$ along $j$ and by ``inserting'' $C$.
Then, there is a canonical bijection $\sigma \mapsto \sigma'$ between $\Spin(N)$ and $\Spin(N')$,  such that
\begin{equation}
\label{eq:variation_Rochlin}
R_{N'}(\sigma') - R_N(\sigma) =  8 \cdot \beta(C)(j^*(\sigma)) \ \in \Z_{16}.
\end{equation}
\end{proposition}

\begin{proof}
As in the proof of Proposition \ref{prop:variation_Morita},
we denote by  $J$ the surface $j(\Sigma_{g,1})$ in $N$. The $3$-manifold $N'$ is then defined as 
\begin{equation}
\label{eq:gluing}
N' := \left(N \setminus \interior\left(J \times [-1,1]\right)\right) \cup_{j' \circ c^{-1}} C
\end{equation}
where $J \times [-1,1]$ denotes a closed regular neighborhood of $J$ in $N$
and $j'$ is the restriction to the boundary of $j \times \Id: \Sigma_{g,1} \times [-1,1] \to J \times [-1,1]$.
Given $\sigma \in \Spin(N)$, we can restrict it to the boundary of $J \times [-1,1]$ and pull-back by $j'$
to $j'^*(\sigma) \in \Spin\left(\partial (\Sigma_{g,1} \times [-1,1])\right)$. 
Next, by pull-back with the homeomorphism $c^{-1}: \partial C \to \partial (\Sigma_{g,1} \times [-1,1])$,
we get a spin structure  $c^{-1,*}j'^*(\sigma)$ on $\partial C$.

\begin{quote}
\textbf{Claim.}
The restriction map $\Spin(C) \to \Spin(\partial C)$ is injective  and, for all $\alpha \in \Spin(\partial C)$,
$\alpha$ comes from $\Spin(C)$ if and only if we have $c_+^*(\alpha) =c_-^*(\alpha) \in \Spin(\Sigma_{g,1})$.
\end{quote}

\noindent
The injectivity of $\Spin(C) \to \Spin(\partial C)$
follows from the injectivity of the homomorphism $H^1(C;\Z_2) \to H^1(\partial C;\Z_2)$ induced by the inclusion.
To prove the equivalence, we observe that $(c^*_+,c^*_-)$ defines an isomorphism
between $H^1(\partial C;\Z_2)$ and the direct sum of two copies of $H^1(\Sigma_{g,1};\Z_2)$, 
which sends the image of $H^1(C;\Z_2)$ to the diagonal.
So, thanks to the affine actions, it suffices to prove the necessary condition. Let $\eta \in \Spin( C)$. 
The quadratic forms $q_{c_{+}^*(\eta)}$ and $q_{c_{-}^*(\eta)}$ are equal since, 
for all simple oriented closed curve $\gamma \subset \Sigma_{g,1}$, we can find a compact oriented surface $S$ properly embedded in $C$
with boundary $c_+(\gamma)\cup \left(-c_-(\gamma)\right)$.
So, by the correspondence (\ref{eq:Atiyah-Johnson}), the two spin structures $c_+^*(\eta)$ and $c_-^*(\eta)$ of $\Sigma_{g,1}$ coincide.

The above claim shows that $c^{-1,*}j'^*(\sigma) \in \Spin(\partial C)$ extends to a unique spin structure on $C$.
Since the gluing locus in (\ref{eq:gluing}) is connected, we can define by gluing the following spin structure on $N'$:
$$
\sigma' := \sigma|_{N \setminus \interior\left(J \times [-1,1]\right)} \cup_{j' \circ c^{-1}} \left(\hbox{extension of }c^{-1,*}j'^*(\sigma)\right).
$$
The  map $\Spin(N) \to \Spin(N')$ defined by $\sigma \mapsto \sigma'$ is bijective, since it is affine over the
isomorphism $H^1(N;\Z_2) \to H^1(N';\Z_2)$ obtained by a Mayer--Vietoris argument. 

Let $W$ be a compact connected oriented smooth $4$-manifold with boundary $N$ and to which $\sigma$ can be extended.
Also, let $X$ be a compact connected oriented smooth $4$-manifold 
bounded by $\overline{C}$ and to which $c^{*,-1}j^*(\sigma)$ can be extended.
As in the proof of  Proposition \ref{prop:variation_Morita}, we think of $\overline{C}$ as the gluing
$$
\overline{C} = C \cup_c -(\Sigma_{g,1} \times [-1,1]).
$$
Then, regarding $X$ as a kind of ``generalized'' handle, we can glue it to $W$ using the attaching map 
$j\times [-1,1]: \Sigma_{g,1} \times [-1,1] \to J \times [-1,1] \subset \partial W$. 
The resulting $4$-manifold $W'$ is bounded by $M'$
and has a spin structure which restricts to $\sigma'$ on the boundary. 
Thus, we have
$$
\sgn(W') = \sgn(W) + \sgn(X) - \hbox{correcting term}.
$$
Since $C$ is a homology cylinder, the kernel of $c_*:H_1\left(\partial(\Sigma_{g,1} \times [-1,1])\right) \to H_1(C)$ 
is the same as that of the map 
$H_1\left(\partial(\Sigma_{g,1} \times [-1,1])\right) \to H_1(\Sigma_{g,1} \times [-1,1])$ induced by the inclusion.
Then, it follows from its description by Wall \cite{Wall} that the correcting term must be zero. 
Thus, reducing mod $16$, we get
$$
R_{N'}(\sigma') = R_N(\sigma) + R_{\overline{C}} \left(c^{*,-1} j^*(\sigma)\right)
$$
and the conclusion follows.
\end{proof}

Since the map $\beta$ is a monoid homomorphism, 
its restriction to $\Torelli_{g,1}$ vanishes on $\Gamma_2 \Torelli_{g,1}$.
By using the characterization of the $Y_2$-equivalence in terms of commutators in the Torelli group (Theorem \ref{th:Y_Torelli}),
we deduce from Proposition \ref{prop:variation_Rochlin} that $\beta$ is invariant under $Y_2$-surgery. 
Thus, we get a homomorphism $\beta: \cyl_{g,1}/Y_2 \to B_{\leq 3}$.

In the sequel, we denote $H_{(2)}:= H \otimes \Z_2$.
The third derivative $\diff^3\! c$ of a cubic function $c:\Spin(\Sigma_{g,1}) \to \Z_2$ is multilinear and (being in characteristic $2$)
it is also alternate: so, $\diff^3\! c$ lives in $\Hom( \Lambda^3H^1(\Sigma_{g,1};\Z_2), \Z_2)$
or, equivalently, in $\Lambda^3 H_1(\Sigma_{g,1}; \Z_2) \simeq \Lambda^3 H_{(2)}$.
According to formula (\ref{eq:Rochlin_to_cohomology_mod_2}) and to Proposition \ref{prop:triple-cup},
the composition $\diff^3 \circ \beta$ coincides with the mod $2$ reduction of the first Johnson homomorphism.
Thus, we have a group homomorphism
$$
(\tau_1, \beta): \cyl_{g,1}/Y_2 \longrightarrow \Lambda^3 H \times_{\Lambda^3 H_{(2)}} B_{\leq 3}
$$
with values in the pull-back of abelian groups
$$
\Lambda^3 H \times_{\Lambda^3 H_{(2)}} B_{\leq 3} 
:= \left\{ (t,c) \in \Lambda^3 H \times B_{\leq 3} : t \!\!\! \mod 2 = \diff^3\! c \right\}.
$$

To find an inverse to the map $(\tau_1,\beta)$, we need to define a third abelian group and, for this, 
it is convenient to introduce the following terminology \cite{MM}. An \emph{abelian group with special element}
is an abelian group $G$ with a distinguished element $s$ of order at most $2$ (the \emph{special} element).
For every abelian group with special element $(G,s)$, 
we consider the abelian group $\A_1(G,s)$ freely generated by $G$-colored Y graphs and subject to the ``multilinearity'' and ``slide'' relations:
$$
\begin{array}{ccc}
\Ygraphbottoptop{g_1+g'_1}{g_2}{g_3}
= \Ygraphbottoptop{g_1}{g_2}{g_3} +\Ygraphbottoptop{g'_1}{g_2}{g_3} & 
\hphantom{BBBBB} & \Ygraphbottoptop{g_2}{g_1}{g_1} = \Ygraphbottoptop{g_2}{g_1}{s} \\
\hbox{Multilinearity} & & \hbox{Slide}
\end{array}
$$
This defines a functor $\A_1$ from the category of abelian groups with special elements to the category of abelian groups.
To every compact oriented (smooth) $3$-manifold $M$, we can associate the abelian group with special element
$$
 \left(H_1(\operatorname{F}(M)),s\right)
$$
where $\operatorname{F}(M)$ is the total space of the bundle of oriented frames on $M$,
and where $s$ is the image of the generator of $H_1(\operatorname{GL}_+(3;\R))\simeq \Z_2$.
We consider the following square in the category of abelian groups with special elements:
\begin{equation}
\label{eq:square}
\xymatrix{ 
\left(H_1(\operatorname{F}(M)),s\right) 
\ar[r]^-{e} \ar[d]_-{p_*}  & \left(B_{\leq 1}(\Spin(M)), \overline{1} \right) \ar[d]^-{\diff^1} \\
(H_1(M),0) \ar[r]_-{\mod 2} & (H_1(M;\Z_2),0). 
}
\end{equation}
Here $B_{\leq 1}(\Spin(M))$ denotes the space of affine functions $\Spin(M) \to \Z_2$
(\ie polynomial functions of degree $\leq 1$), and we take as special element the constant function $\overline{1}:\sigma \mapsto 1$.
The map $\diff^1$ sends an affine function to its linear part, which lives in $\Hom(H^1(M;\Z_2),\Z_2)\simeq H_1(M;\Z_2)$.
The map $p_*$ in the above diagram is induced by the bundle projection $p:\operatorname{F}(M) \to M$, 
while the map $e$ sends all $x\in H_1(\operatorname{F}(M))$ to the evaluation at $x$.
(Here, we are regarding $\Spin(M)$ as the set of those $y \in H^1(\operatorname{F}(M);\Z_2)$ satisfying $\langle y,s \rangle \neq 0$.)
It turns out that (\ref{eq:square}) is a pull-back diagram \cite{MM} 
so, to define an element of $H_1\left(\operatorname{F}(M)\right)$, it is enough to specify its images by $p_*$ and $e$.
Those constructions apply in particular to the $3$-manifold $M:= \Sigma_{g,1} \times [-1,1]$ and we denote
$$
P:= \big(H_1(\operatorname{F}(\Sigma_{g,1}\!\! \times\!\! [-1,1])),s\big).
$$
By applying the functor $\A_1$ to the commutative square (\ref{eq:square}), we get a map
$$
\A_1(P) \longrightarrow \A_1(H,0) \times_{\A_1(H_{(2)},0)} \A_1\left(B_{\leq 1}, \overline{1} \right).
$$
But, $\A_1(H,0)$ is canonically isomorphic to $\Lambda^3 H$ by the map
$$
\Ygraphbottoptop{h_1}{h_2}{h_3} \longmapsto h_1 \wedge h_2 \wedge h_3
$$
which we have already met in Example \ref{ex:degree_one}. 
(Note that $\A_1(H,0)\otimes \Q = \A_1(H_\Q)$ with the notation of \S \ref{subsec:symplectic_Jacobi_diagrams}.)
Besides,  the group $\A_1\left(B_{\leq 1}, \overline{1} \right)$ is canonically isomorphic to the space of cubic functions $B_{\leq 3}$ by the map
$$
\Ygraphbottoptop{a_1}{a_2}{a_3} \longmapsto a_1 \cdot a_2 \cdot a_3.
$$
Thus, we have obtained for free a map
$$
\digamma: \A_1(P) \longrightarrow \Lambda^3 H \times_{\Lambda^3 H_{(2)}} B_{\leq 3}.
$$

\begin{theorem}[See \cite{Habiro,MM}]
\label{th:Y_2}
Surgery along graph claspers of degree $1$ defines a map $\psi_1: \A_1(P) \to \cyl_{g,1}/Y_2$,
and we have the following commutative triangle:
$$
\xymatrix{
\A_1(P) \ar[r]^-{\psi_1}_-\simeq  \ar[rd]_-\digamma^-\simeq & \cyl_{g,1}/Y_2 \ar[d]^-{(\tau_1,\beta)}_-\simeq \\
& \Lambda^3 H \times_{\Lambda^3 H_{(2)}} B_{\leq 3}
}
$$
\end{theorem}

\noindent
This is similar to Johnson's computation of the abelianization of the Torelli group.

\begin{theorem}[Johnson \cite{Johnson_abelianization}]
For $g\geq 3$, the map  
$$
(\tau_1,\beta): \Torelli_{g,1}/\Gamma_2 \Torelli_{g,1} \longrightarrow \Lambda^3 H \times_{\Lambda^3 H_{(2)}} B_{\leq 3}
$$
is an isomorphism.
\end{theorem}

\noindent
We deduce that, for $g \geq 3$,  the mapping cylinder construction induces an isomorphism
$\Gr_1 \mcyl: \Torelli_{g,1}/\Gamma_2 \Torelli_{g,1} \to \cyl_{g,1}/Y_2$. 
This solves Problem \ref{pb:gr_mapping_cylinder} in degree $1$.

\begin{proof}[Sketch of the proof of Theorem \ref{th:Y_2}]
The map $\psi_1$ is defined in a way similar to the surgery map $\psi$ of Theorem \ref{th:surgery}.
Thus, the map $\psi_1$ sends every $P$-colored $Y$ graph 
$$
D=\Ygraphbottoptop{p_2}{p_1}{p_3}
$$
to the 3-manifold obtained from  $\Sigma_{g,1} \times [-1,1]$ by surgery along a certain $Y_1$-tree $C(D)$.
To define it, first consider the oriented surface $S$ consisting of one disk, 
three annuli and three bands connecting the latter to the former.
Use the induced orientation on $\partial S$ to orient the cores of the annuli, 
and to give to the three annuli a cyclic ordering $1 < 2 < 3 <1$.
Then, embed $S$ into the interior of $\Sigma_{g,1} \times [-1,1]$ in such a way 
that the framed oriented knot corresponding to the $i$-th annuli of $S$  
lifts to an oriented curve in $\operatorname{F}\left(\Sigma_{g,1}\! \times\! [-1,1]\right)$ that represents $p_i + s$. 
We obtain in this way a graph clasper $C(D)$.

As in Theorem \ref{th:surgery}, it can be proved by using claspers that $\psi_1$ is well-defined and surjective.
Next, it is not difficult to compute how $\tau_1$ and  $\beta$ change under $Y_1$-surgery. 
For instance, such variation formulas can be derived from Proposition \ref{prop:variation_Morita} (with $k=1$) 
and from Proposition \ref{prop:variation_Rochlin}, respectively. (See also \cite{MM}.)
Such formulas show that $(\tau_1,\beta)\circ \psi_1=\digamma$.
Finally, it can be shown by working with a basis of $H$ that the map $\digamma$ is bijective. 
It follows that $\psi_1$ and $(\tau_1,\beta)$, too, are isomorphisms.
\end{proof}

We now deal with the closed case.
Let $C \in \cyl_g$ and let us consider the closure $\overline{C}$ of $C$. 
The inclusion $c_\pm: \Sigma_{g} \to C \subset \overline{C}$ 
induces an surjective affine map $c^*:\Spin(\overline{C}) \to \Spin(\Sigma_g)$.
Thus, each $\sigma \in \Spin(\Sigma_g)$ has two extensions $\overline{\sigma}_1$ and $\overline{\sigma}_2$
to $\overline{C}$, and we may have $R_{\overline{C}}(\overline{\sigma}_1) \neq R_{\overline{C}}(\overline{\sigma}_2)$.
Nevertheless, this can not occur if $(\Sigma_g,\sigma)$ is the boundary of a spin $3$-manifold, 
because one can then find a compact oriented smooth $4$-manifold with boundary $\overline{C}$
to which both $\overline{\sigma}_1$ and $\overline{\sigma}_2$ extend. 
It follows that, for any section $s$ of the surjective affine map $c^*:\Spin(\overline{C}) \to \Spin(\Sigma_g)$,
the cubic function $\frac{1}{8} \cdot R_{\overline{C}} \circ s : \Spin(\Sigma_g) \to \Z_2$ is independent of $s$ on the subset
$$
\Spin_\partial(\Sigma_g) := 
\left\{ \sigma \in \Spin(\Sigma_g): 
(\Sigma_g,\sigma) \hbox{  is the boundary of a spin $3$-manifold}\right\}.
$$
As in the bordered case, spin structures on $\Sigma_g$ can be identified with quadratic forms:
\begin{equation}
\label{eq:Atiyah-Johnson_closed}
\Spin(\Sigma_{g}) \stackrel{\simeq}{\longrightarrow} \mathcal{Q}, \ \sigma \longmapsto q_\sigma.
\end{equation}
The quadratic forms that correspond to bounding spin structures can be recognized 
thanks to the \emph{Arf invariant}, defined by
$$
\Arf(q) := \sum_{i=1}^g q\left(\alpha_i\right) \cdot q\left(\beta_i\right) \ \in \Z_2
$$
for all $q\in \mathcal{Q}$.
(More precisely, a spin structure $\sigma$ on $\Sigma_g$
belongs to $\Spin_\partial (\Sigma_g)$ if and only if $\Arf(q_{\sigma})=0$ \cite{Kirby}.)
Furthermore,  any two $f,f' \in B_{\leq 3}$ coincide on $\Spin_\partial(\Sigma_g)$ 
if and only if $f-f'$ is a multiple of $\Arf \in B_{\leq 2}$ \cite{Johnson_BC}.
(Here, we still denote by $B_{\leq d}$ the space of polynomial functions $\Spin(\Sigma_g) \to \Z_2$ of degree at most $d$,
so that the $\Arf$ invariant can be regarded as an element of $B_{\leq 2}$.)

\begin{definition}
The \emph{Birman--Craggs homomorphism}\index{homomorphism!Birman--Craggs}\index{Birman--Craggs homomorphism} is the map
$$
\beta: \cyl_{g} \longrightarrow B_{\leq 3}/\Arf\cdot B_{\leq 1}
$$
sending any $C \in \cyl_{g}$ to the class of the cubic function $\frac{1}{8} \cdot R_{\overline{C}} \circ s$,
where $s$ is a section of the surjective affine map $c^*:\Spin(\overline{C}) \to \Spin(\Sigma_g)$.
\end{definition}

Assume that $\Sigma_g$ is decomposed into a closed disk and a copy of $\Sigma_{g,1}$ as in \S \ref{subsec:bordered_to_closed}.
Then, for all $M \in \cyl_{g,1}$,
the closure of $\clocase{M}$ can be obtained from the closure of $M$
by $0$-framed surgery along a null-homologous knot. 
This allows us to bring their Rochlin functions into relation
and to see that the following square is commutative:
$$
\xymatrix{
\cyl_{g,1} \ar@{->>}[d]_-{\clocase{\ \centereddot\ }} \ar[r]^-\beta & B_{\leq 3} \ar@{->>}[d] \\
\cyl_{g} \ar[r]_-\beta & B_{\leq 3}/\Arf\cdot B_{\leq 1}
}
$$
We deduce that, in the closed case, too, $\beta$ is a monoid map and it is invariant under $Y_2$-surgery. 

For every $f\in B_{\leq 1}$, the third derivative of $\Arf \cdot f$ is $\omega \wedge \diff^1\! f \in \Lambda^3 H_{(2)}$. 
It follows from formula (\ref{eq:Rochlin_to_cohomology_mod_2}) and Proposition \ref{prop:triple-cup_closed}
that $\tau_1$ mod 2 coincides with $\diff^3 \circ \beta$. Thus, we obtain a group homomorphism
$$
(\tau_1, \beta): \cyl_{g}/Y_2 \longrightarrow 
\frac{\Lambda^3 H}{\omega\wedge H} \times_{\frac{\Lambda^3 H_{(2)}}{\omega \wedge H_{(2)}}} \frac{B_{\leq 3}}{\Arf \cdot B_{\leq 1}}.
$$
As in the bordered case, we consider the abelian group with special element
$$
P:= \big(H_1(\operatorname{F}(\Sigma_{g}\!\! \times\!\! [-1,1])),s\big)
$$
and, by applying the functor $\mathcal{A}_1$ to the commutative square (\ref{eq:square}), we get a map
$$
\digamma: \A_1(P) \longrightarrow \Lambda^3 H \times_{\Lambda ^3 H_{(2)}} B_{\leq 3}.
$$
Let $I_1$ be the subgroup of $\A_1(P)$ consisting of elements of the form
$$
\sum_{i=1}^g\ \Ygraphbottoptop{z}{\alpha_i^+}{\beta_i^+}
$$
where $z$ is an arbitrary element of $P$ and, for all $h \in H$,
$h^+$ is the unique element of $P$ which, in the pull-back diagram (\ref{eq:square}),
is sent to $h$ by $p_*$ and to the evaluation of quadratic forms at $h$ by $e$. 
We set
$$
\clocase{\A}_1(P) := \A_1(P) /I_1.
$$
Observe that the subgroup $\digamma(I_1)$ consists of elements of the form 
$(\omega \wedge h, \Arf \cdot f)$ where $h \in H$ and $f \in B_{\leq 1}$ are such that $\diff^1\! f = h \mod 2$.
As in the bordered case, the following can then be proved.  

\begin{theorem}[See \cite{Habiro,MM}]
\label{th:Y_2_closed}
Surgery along graph claspers of degree $1$ defines a map $\psi_1: \clocase{\A}_1(P) \to \cyl_{g}/Y_2$,
and we have the following commutative square:
$$
\xymatrix{
\clocase{\A}_1(P) \ar[r]^-{\psi_1}_-\simeq  \ar[d]_-\digamma^-\simeq & \cyl_{g}/Y_2 \ar[d]^-{(\tau_1,\beta)}_-\simeq   \\ 
\frac{\Lambda^3 H \mathop{\times}_{\Lambda^3 H_{(2)}} B_{\leq 3}}{(\omega\wedge H) \times_{\omega \wedge H_{(2)}} (\Arf \cdot B_{\leq 1})}
 \ar[r]_-\simeq & \frac{\Lambda^3 H}{\omega\wedge H} \times_{\frac{\Lambda^3 H_{(2)}}{\omega \wedge H_{(2)}}} \frac{B_{\leq 3}}{\Arf \cdot B_{\leq 1}}
}
$$
\end{theorem}

\noindent
This is similar to Johnson's computation of the abelianization of the Torelli group.

\begin{theorem}[Johnson \cite{Johnson_abelianization}]
For $g\geq 3$, the map  
$$
(\tau_1,\beta): \frac{\Torelli_{g}}{\Gamma_2 \Torelli_{g}} \longrightarrow
\frac{\Lambda^3 H}{\omega\wedge H} \times_{\frac{\Lambda^3 H_{(2)}}{\omega \wedge H_{(2)}}} \frac{B_{\leq 3}}{\Arf \cdot B_{\leq 1}}
$$
is an isomorphism.
\end{theorem}

\noindent
We deduce that, for $g \geq 3$,  the mapping cylinder construction induces an isomorphism 
$\Gr_1 \mcyl: \Torelli_{g}/\Gamma_2 \Torelli_{g} \to \cyl_{g}/Y_2$.
Thus, Problem \ref{pb:gr_mapping_cylinder} in degree $1$ is also solved in the closed case.

\section{The homology cobordism group}

\label{sec:group}

If homology cobordisms of $\Sigma_{g,b}$ are considered up to the \emph{relation} of homology cobordism, then the monoid
$\cob(\Sigma_{g,b})$ turns into a group. In this final section, 
we outline how the previous techniques apply to the study of this group,
and we mention a few recent results about it.

\subsection{Definition of the homology cobordism group}

\label{subsec:definition_group}

Two homology cobordisms $M$ and $N$ of $\Sigma_{g,b}$ are said to be 
\emph{homology cobordant}\index{homology cobordism!(relation)}\index{relation!homology cobordism}
if the closed oriented $3$-manifold $M \cup_{m \circ n^{-1}} (-N)$ bounds a compact oriented smooth $4$-manifold $W$,
in such a way that the inclusions $M \subset W$ and $N \subset W$ induce isomorphisms in homology.
In this situation, we denote $M \sim_H N$.
This defines an equivalence relation among homology cobordisms, which is compatible with their composition. 
The quotient monoid is denoted by 
$$
\gcob(\Sigma_{g,b}) := \cob(\Sigma_{g,b})/\sim_H.
$$ This monoid is actually a group. Indeed, let $\iota$ be the involution of $\partial(\Sigma_{g,b} \times [-1,1])$
defined by $\iota(x,t)=(x,-t)$.  For any $M \in \cob(\Sigma_{g,b})$, denote by $-M \in \cob(\Sigma_{g,b})$ the
$3$-manifold $M$ with opposite orientation and with boundary parameterized by $m\circ \iota$.  Then, $M \circ (-M)$ is
homology cobordant to $\Sigma_{g,b} \times [-1,1]$ via $M \times [-1,1]$.

\begin{definition}
The \emph{homology cobordism group}\index{homology cobordism group}\index{group!homology cobordism}
of homology cobordisms is $\gcob(\Sigma_{g,b})$.
\end{definition}

The following subgroup of $\gcob(\Sigma_{g,b})$ is of special interest:
$$
\gcyl(\Sigma_{g,b}) := \cyl(\Sigma_{g,b})/\sim_H.
$$
That group should not be confused with the ``group of homology cylinders'' introduced in \S \ref{subsec:Y_cylinders}.
The latter is defined as a completion and is conjectured to contain the monoid $\cyl(\Sigma_{g,b})$,
while the former is a quotient of the monoid $\cyl(\Sigma_{g,b})$.

The homology cobordism group has been introduced by Garoufalidis and Levine 
as an  ``enlargement'' of the mapping class group \cite{GL,Levine}.

\begin{proposition}[Garoufalidis--Levine \cite{Levine, GL}]
\label{prop:injectivity}
The group $\mcg(\Sigma_{g,b})$ embeds into $\gcob(\Sigma_{g,b})$ by the composition
$\xymatrix{
\mcg(\Sigma_{g,b}) \ar[r]^-{\mcyl} & \cob(\Sigma_{g,b}) \ar@{->>}[r] & \gcob(\Sigma_{g,b}).
}$
\end{proposition}

\begin{proof}[Sketch of the proof]
Let $f\in \mcg(\Sigma_{g,b})$ and assume that there exists a homology cobordism $W$ between $\mcyl(f)$ and $\Sigma_{g,b} \times [-1,1]$.
Let $\Sigma_{g,b}^+$ and $\Sigma_{g,b}^-$ be the two copies of $\Sigma_{g,b}$ in $W$ 
along which $\mcyl(f)$ and $\Sigma_{g,b} \times [-1,1]$ are glued. By assumption, the inclusion $\Sigma_{g,b}^\pm \subset W$
induces an isomorphism in homology and so, by Stallings' theorem \cite{Stallings}, 
it induces  an isomorphism at the level of the $k$-th nilpotent quotient of the fundamental group for any $k\geq 1$.
We deduce that $f_*: \pi_1(\Sigma_{g,b},\star) \to \pi_1(\Sigma_{g,b},\star)$ 
is the identity modulo $\Gamma_{k+1}  \pi_1(\Sigma_{g,b},\star)$ for all $k\geq 1$.
Here, the base point $\star$ is chosen on $\partial \Sigma_{g,b}$ if $b>0$ or on a small disk $D$ fixed by $f$ if $b=0$.
Since the group $\pi := \pi_1(\Sigma_{g,b},\star)$  is residually nilpotent, we deduce that $f_*:\pi \to \pi$ is the identity.

If $b=0$ or $b=1$, we conclude that $f$ is isotopic to the identity 
since the Dehn--Nielsen representation is injective (\S \ref{subsec:Dehn-Nielsen}).
If $b>1$, we number the boundary components of $\Sigma_{g,b}$ from $1$ to $b$,
and we assume that $\star$ belongs to the $k$-th component $\partial_k \Sigma_{g,b}$. 
The fact that $f_*:\pi \to \pi$ is the identity implies that
$f$ is a product of Dehn twists along curves parallel to $\partial_i \Sigma_{g,b}$ with $i\neq k$.
The conclusion follows by making $k$ range from $1$ to $b$: we refer to \cite{CFK} for details.
\end{proof}

The following example illustrates the richness of the group $\gcob(\Sigma_{g,b})$, 
which goes far beyond that of the mapping class group $\mcg(\Sigma_{g,b})$.

\begin{example}[Genus $0$]
\label{ex:group_genus_0}
For $b>1$, recall from Example \ref{ex:genus_0_cobordisms} that the monoid $\cob(\Sigma_{0,b})$
can be identified with the monoid of framed $(b-1)$-strand string-links in homology $3$-balls.  
The relation $\sim_H$ then corresponds to a cobordism relation 
which generalizes the concordance relation for string-links in $D^2 \times [-1,1]$ 
studied by Habegger and Lin \cite{HL_concordance,HL_link-homotopy}: see \cite[\S 3]{Levine}.  
For  $b=0$ and $b=1$, recall from Example \ref{ex:genus_0_cobordisms} 
that the monoid $\cob(\Sigma_{0,b})$ can be identified with the monoid
of homology $3$-spheres.  Thus, gauge-theoretic results of Furuta \cite{Furuta} and Fintushel--Stern \cite{FS} tell us
that the group $\gcob(\Sigma_{0,1}) \simeq \gcob(\Sigma_{0,0})$ has infinite rank.
\end{example}

\subsection{Representations of the homology cobordism group}

\label{subsec:group_homomorphisms}

As before, we are specially interested in the group $\gcob(\Sigma_{g,b})$ for $b=0$ and $b=1$,
in which case it is denoted by $\gcob_{g}$ and $\gcob_{g,1}$ respectively. 
Let us see how much the results from the previous sections apply to the study of those groups.
We start by indicating which invariants of homology cobordisms respect the relation $\sim_H$.  

As in the proof of Proposition \ref{prop:injectivity}, one can check that, for any $k\geq 1$,
the $k$-th nilpotent reduction of the Dehn--Nielsen representation (\S \ref{subsec:Dehn-Nielsen}) factorizes by the relation $\sim_H$.
Thus, we obtain group homomorphisms
$$
\gcob_{g,1} \stackrel{\rho_k}{\longrightarrow} \Aut(\pi/\Gamma_{k+1} \pi)
\quad \hbox{and} \quad
\gcob_{g}  \stackrel{\rho_k}{\longrightarrow}  \Out(\clocase{\pi}/\Gamma_{k+1} \clocase{\pi}).
$$
So, the Johnson filtration (\S \ref{subsec:Dehn-Nielsen}) can be defined on the groups $\gcob_{g,1}$ and $\gcob_g$:
$$
\gcob_{g,1}= \gcob_{g,1}[0] \supset \gcob_{g,1}[1] \supset \gcob_{g,1}[2] \supset \cdots
\quad \hbox{and} \quad
\gcob_{g}= \gcob_{g}[0] \supset \gcob_{g}[1] \supset \gcob_{g}[2] \supset \cdots 
$$
Note that $\gcyl_{g,1} =  \gcob_{g,1}[1]$ and $\gcyl_{g} =  \gcob_{g}[1]$.
The Johnson homomorphisms (\S \ref{subsec:Johnson}) can be defined on the subgroups of that filtration,
and similarly for the Morita homomorphisms (\S \ref{subsec:Morita}) in the bordered case.
 
Furthermore, the infinitesimal Dehn--Nielsen representation (\S \ref{subsec:total_Johnson}) 
is defined on the homology cobordism group, for any symplectic expansion $\theta$:
\vspace{-0.2cm}
$$
\gcob_{g,1} \stackrel{\varrho^\theta}{\longrightarrow} \Aut(\Lie_\Q)
\quad \hbox{and} \quad
\gcob_{g}  \stackrel{\varrho^{\theta}}{\longrightarrow}  \Out(\clocase{\Lie}_\Q).
$$
It follows from Theorem \ref{th:tree-reduction_LMO} and Theorem \ref{th:tree-reduction_LMO_closed} 
that the tree-reduction of the LMO homomorphism factorizes to the homology cobordism group of homology cylinders.

\begin{corollary}
The LMO homomorphism induces group homomorphisms
$$
Z^t: \gcyl_{g,1} \longrightarrow \GLike\left(\A^t(H_\Q)\right)
\quad \hbox{and} \quad
Z^t: \gcyl_{g} \longrightarrow \GLike\left(\clocase{\A}^t(H_\Q)\right).
$$
\end{corollary}
\noindent
However, the full LMO homomorphism is {\em not} invariant under the relation $\sim_H$.

Finally, by the $4$-dimensional definition of the Rochlin invariant, 
the Birman--Craggs homomorphism (\S \ref{subsec:degree_one}) also gives group homomorphisms
$$
\gcyl_{g,1} \stackrel{\beta}{\longrightarrow} B_{\leq 3}
\quad \hbox{and} \quad
\gcyl_{g} \stackrel{\beta}{\longrightarrow}  B_{\leq 3}/\Arf \cdot B_{\leq 1}.
$$

\subsection{The $Y$-filtration on the homology cobordism group}

For all $g,b\ge 0$, the $Y$-filtration on the monoid $\cyl(\Sigma_{g,b})$ 
induces a filtration on the group $\gcyl(\Sigma_{g,b})$ in the following way.
For all $k\geq 1$, we define $Y_k \gcyl(\Sigma_{g,b})$ to be the subgroup of $\gcyl(\Sigma_{g,b})$
consisting of homology cobordism classes of homology cylinders which are $Y_k$-equivalent to $\Sigma_{g,b} \times [-1,1]$.  
Thus, we get a decreasing sequence of subgroups\index{Y-filtration@$Y$-filtration}
$$
\gcyl(\Sigma_{g,b}) = Y_1\gcyl(\Sigma_{g,b}) \supset  Y_2\gcyl(\Sigma_{g,b})  \supset  Y_3\gcyl(\Sigma_{g,b}) \supset \cdots
$$
Two homology cobordism classes $\{M\},\{N\} \in \gcyl(\Sigma_{g,b})$ are said to be 
\emph{$Y_k$-equivalent}\index{relation!$Y_k$-equivalence}\index{Yk-equivalence relation@$Y_k$-equivalence relation}
if $M,N\in \cyl(\Sigma_{g,b})$ are related by a finite sequence of homology cobordisms and $Y_k$-surgeries.
Thus, we have an inclusion
$$
Y_k \gcyl(\Sigma_{g,b}) \subset 
\big\{ \{M\} \in \gcyl(\Sigma_{g,b}): \{M\} \hbox{ is $Y_k$-equivalent to } \{\Sigma_{g,b} \times [-1,1]\}\big\}.
$$ 

The $Y_k$-equivalence on $\gcyl(\Sigma_{g,b})$ is an equivalence relation,
and the quotient set $\gcyl(\Sigma_{g,b})/Y_k$ can be regarded either 
as a quotient of the group $\gcyl(\Sigma_{g,b})$, or as a quotient of the group $\cyl(\Sigma_{g,b})/Y_{k}$.
Thus, the  projection $\cyl(\Sigma_{g,b}) \to \gcyl(\Sigma_{g,b})$ 
induces a surjective Lie ring homomorphism
\begin{gather}
\label{eq:monoid_to_group}
\xymatrix{
\Gr^Y \cyl(\Sigma_{g,b}) \ar@{->>}[r] & \Gr^Y \gcyl(\Sigma_{g,b})
}
\end{gather}
where $\Gr^Y \cyl(\Sigma_{g,b})$ is the graded Lie ring studied in Section \ref{sec:Lie_ring_homology_cylinders} and where
\begin{gather*}
 \Gr^Y \gcyl(\Sigma_{g,b}) :=\bigoplus_{k\ge1} \frac{Y_k\gcyl(\Sigma_{g,b})}{Y_{k+1}}.
\end{gather*}

We now specialize to the cases $b=0,1$
and we  first deal with the degree $1$ case.

\begin{proposition}
\label{prop:groups_degree_1}
For all $g\geq 0$, we have the following isomorphisms:
$$
\frac{\gcyl_{g,1}}{Y_2}  \mathop{\longrightarrow}^{(\tau_1,\beta)}_\simeq \Lambda^3 H \times_{\Lambda^3 H_{(2)}} B_{\leq 3}
\quad \hbox{and} \quad
\frac{\gcyl_{g}}{Y_2}  \mathop{\longrightarrow}^{(\tau_1,\beta)}_\simeq 
\frac{\Lambda^3 H}{\omega\wedge H} \times_{\frac{\Lambda^3 H_{(2)}}{\omega \wedge H_{(2)}}} \frac{B_{\leq 3}}{\Arf \cdot B_{\leq 1}}.
$$
\end{proposition}

\begin{proof}
We saw in \S \ref{subsec:group_homomorphisms} that $\tau_1$ and $\beta$ factorize by the relation of homology cobordism.
Thus, we have the following commutative triangles:
$$
\xymatrix{
\cyl_{g,1}/Y_2 \ar@{->>}[r] \ar[rd]^-\simeq_-{(\tau_1,\beta)}
& \gcyl_{g,1}/Y_2  \ar[d]^-{(\tau_1,\beta)}\\
&  \Lambda^3 H \times_{\Lambda^3 H_{(2)}} B_{\leq 3}
}
\quad \hbox{and} \quad
\xymatrix{
\cyl_{g}/Y_2 \ar@{->>}[r] \ar[rd]^-\simeq_-{(\tau_1,\beta)}
& \gcyl_{g}/Y_2  \ar[d]^-{(\tau_1,\beta)}\\
&  \frac{\Lambda^3 H}{\omega\wedge H} \times_{\frac{\Lambda^3 H_{(2)}}{\omega \wedge H_{(2)}}} \frac{B_{\leq 3}}{\Arf \cdot B_{\leq 1}}
}
$$
The horizontal arrows correspond to the map (\ref{eq:monoid_to_group}) in degree $1$,
and the diagonal arrows are isomorphisms by Theorem \ref{th:Y_2} and Theorem \ref{th:Y_2_closed}.
\end{proof}

In higher degree, we need the following important fact.

\begin{lemma}[Levine \cite{Levine}]
\label{lem:looped}
Let $M$ be a cobordism of $\Sigma_{g,b}$ 
and let $C$ be a graph clasper in $M$. If $C$ is not a tree clasper, then $M_C$ is homology cobordant to $M$.
\end{lemma}

We  determine the graded Lie algebra $\Gr^Y \gcyl(\Sigma_{g,b})\otimes \Q$ for $b=0,1$ as follows.

\begin{theorem}[Levine \cite{Levine}]
\label{th:groups_rational}
For any $g\geq 0$, the family of all the Johnson homomorphisms $\tau :=\oplus_{k\geq 1} \tau_k$ induces the following isomorphisms:
$$
\Gr^Y \gcyl_{g,1} \otimes \Q  \mathop{\longrightarrow}^{\eta^{-1}  \tau}_\simeq \A^{t,c}(H_\Q)
\quad \hbox{and} \quad
\Gr^Y \gcyl_{g} \otimes \Q  \mathop{\longrightarrow}^{\eta^{-1}  \tau}_\simeq \clocase{\A}^{t,c}(H_\Q)
$$
\end{theorem}

\vspace{-0.3cm}

\begin{proof}
According to (\ref{eq:Y_to_Johnson}), the map $\tau$ is well-defined on $\Gr^Y \gcyl_{g,1}$ and $\Gr^Y \gcyl_g$.
It follows from Lemma \ref{lem:looped} that 
the  surgery maps defined in \S \ref{subsec:Lie_algebra_homology_cylinders} factorize to surjective maps
$\psi: \A^{t,c}(H_\Q) \to \Gr^Y \gcyl_{g,1}\otimes \Q$ and $\psi: \clocase{\A}^{t,c}(H_\Q) \to \Gr^Y \gcyl_{g}\otimes \Q$.
An explicit computation of the variation of $\tau_k$ under $Y_k$-surgery shows that the following triangles are commutative:
\vspace{-0.1cm}
$$
\xymatrix{
\A^{t,c}_k(H_\Q) \ar[rd]_{\eta_k}^-\simeq \ar@{->>}[r]^-{\psi_k} & \frac{Y_k \gcyl_{g,1}}{Y_{k+1}}\otimes \Q  \ar[d]^-{\tau_k}\\
& \operatorname{D}_{k+2}(H_\Q)
}
\quad \hbox{and} \quad
\xymatrix{
\clocase{\A}^{t,c}_k(H_\Q) \ar[rd]_{\eta_k}^-\simeq \ar@{->>}[r]^-{\psi_k} & \frac{Y_k \gcyl_{g}}{Y_{k+1}}\otimes \Q  \ar[d]^-{\tau_k} \\
& \clocase{\operatorname{D}}_{k+2}(H_\Q)
}
$$
\vspace{-0.1cm}
(Alternatively, we can deduce this from Corollary \ref{cor:LMO_to_Johnson} and Theorem \ref{th:LMO_iso}.)
\end{proof}

\begin{remark}
\label{rem:groups_integral_coefficients}
The structure of the abelian group $\Gr^Y \gcyl_{g,1}$ does not seem to be known in degree $>1$.
The surgery map $\psi_k:  \A^{t,c}_k(H) \to  Y_k \gcyl_{g,1}/Y_{k+1}$ exists with integral coefficients
and is surjective, for any $k\geq 2$. (See Remark \ref{rem:integral_coefficients}.)
However, the question of its injectivity seems to be open. 
See \cite{Levine_addendum,Levine_quasi-Lie} in this connection.
\end{remark}

The above results show that the $Y$-filtration is not very well suited to the study of the group $\gcob(\Sigma_{g,b})$,
since the associated graded Lie algebra is (for $b=0,1$ and for rational coefficients) 
completely determined by the Johnson homomorphisms.
This inadequacy is particularly apparent in genus $0$.

\begin{example}[Genus $0$]
Recall from Example \ref{ex:group_genus_0} that $\gcyl_{0,1} \simeq \gcyl_{0,0}$
is the homology cobordism group of homology $3$-spheres.
We deduce from Theorem \ref{th:Y_2} that a homology $3$-sphere $M$ is $Y_2$-equivalent to $S^3$
if and only if its Rochlin invariant is trivial. Assume this for $M$. 
Then, it can be proved that, for all $k\geq 2$, $M$ is obtained
from $S^3$ by surgery along a disjoint union of graph claspers which are not trees 
or which have degree $k$. We deduce from Lemma \ref{lem:looped} that, for all $k\geq 2$,
there is a homology $3$-sphere $M_k$ which is simultaneously $Y_k$-equivalent to $S^3$ and homology cobordant to $M$. 
This shows that the $Y$-filtration on the homology cobordism group of homology $3$-spheres stabilizes starting from the degree $2$.  
(This phenomenon is similar to the fact that the Arf invariant of knots is the only
Vassiliev invariant which is preserved by concordance \cite{Ng}.)
\end{example}

\subsection{Other aspects of the homology cobordism group}

The group $\gcob(\Sigma_{g,b})$ has been the subject of recent works in other directions.
The reader is  referred to Sakasai's chapter for other aspects of the homology cobordism group.
To conclude, we simply mention a few developments which tend to show that $\gcob(\Sigma_{g,b})$
has quite different properties from $\mcg(\Sigma_{g,b})$:

\begin{itemize}
\item By using his ``trace maps'' \cite{Morita_Abelian}, Morita shows that the abelianization of the group $\gcyl_{g,1}$ 
has infinite rank \cite{Morita_birthday}. In particular, the group $\gcyl_{g,1}$ is not finitely generated, 
which contrasts with the fact that $\Torelli_{g,1}$ is finitely generated for $g \geq 3$ \cite{Johnson_generation}.

\item By using Heegaard--Floer homology, Goda  and Sakasai show that the submonoid of $\cob_{g,1}$ 
consisting of irreducible homology cobordisms is not finitely generated \cite{GS}. 
This contrasts with the well-known fact that $\mcg_{g,1}$ is finitely generated.

\item By using  Reidemeister torsion, Cha, Friedl and Kim prove that the abelianization of $\gcob(\Sigma_{g,b})$
contains a direct summand isomorphic to $(\Z_2)^{\infty}$ if $2g+b>1$ 
and a direct summand isomorphic to $\Z^\infty \oplus (\Z_2)^{\infty}$ if $b>1$ \cite{CFK}.
Thus, in contrast with $\mcg(\Sigma_{g,b})$, the group $\gcob(\Sigma_{g,b})$ is not finitely generated (and neither is it perfect).
\end{itemize} 

\vspace{0.5cm}

\noindent
\textbf{Acknowledgement.}
The authors would like to thank Jean--Baptiste Meilhan and Takuya Sakasai for various comments and suggestions.
They are also grateful to Stefan Friedl for showing them a gap in a previous version of the proof of Proposition \ref{prop:injectivity}.
The first author was partially supported by Grant-in-Aid for Scientific Research (C) 21540078.
The second author was partially supported by the French ANR research project ANR-08-JCJC-0114-01.

\printindex


\frenchspacing
\begin{thebibliography}{99}

\bibitem{Alexander_open_book}
J. W. Alexander,
Note on Riemann spaces.
{\it Bull. Amer. Math. Soc.} 26 (1920), no. 8, 370--372. 

\bibitem{Alexander}
J. W. Alexander,
On the subdivision of $3$-space by a polyhedron.
{\it Nat. Acad. Proc.} 10 (1924), 6--8. 

\bibitem{ABMP}
J. E. Andersen, A. J. Bene, J.-B. Meilhan, R. C. Penner, 
Finite type invariants and fatgraphs.
Preprint (2009) \texttt{arXiv:0907$.$2827}, to appear in {\it Adv. Math.}

\bibitem{Atiyah}
M. F. Atiyah,
Riemann surfaces and spin structures.
{\it Ann. Sci. Ecole Norm. Sup. (4)} 4 (1971), 47--62.

\bibitem{Baer1}
R. Baer,
Kurventypen auf Fl\"achen.
{\it J. Reine Angew. Math.} 156 (1927), 231--246.

\bibitem{Baer2}
R. Baer,
Isotopie von Kurven auf orientierbaren, geschlossenen Fl\"achen und ihr Zusamenhang mit der topologischen Deformation der Fl\"achen.
{\it J. Reine Angew. Math.}
159 (1928), 101--116. 

\bibitem{Bar-Natan}
D. Bar-Natan, 
On the Vassiliev knot invariants.
{\it Topology} 34 (1995), no. 2, 423--472. 

\bibitem{BL}
H. Bass, A. Lubotzky, 
Linear-central filtrations on groups.  
{\it The mathematical legacy of Wilhelm Magnus: groups, geometry and special functions (Brooklyn, NY, 1992)}, 
45--98, Contemp. Math., 169, Amer. Math. Soc., Providence, RI, 1994.

\bibitem{Birman_exact_sequence}
J. S. Birman,
Mapping class groups and their relationship to braid groups.
{\it Comm. Pure Appl. Math.} 22 (1969), 213--238. 

\bibitem{Birman_Siegel}
J. S. Birman,
On Siegel's modular group.
{\it Math. Ann.} 191 (1971), 59--68. 

\bibitem{Birman_book}
J. S. Birman,
{\it Braids, links, and mapping class groups.}
Annals of Mathematics Studies, No. 82. Princeton University Press, Princeton, N.J.; University of Tokyo Press, Tokyo, 1974. 

\bibitem{BC}
J. S. Birman, R. Craggs,
The $\mu$-invariant of $3$-manifolds and certain structural properties of the group of homeomorphisms of a closed, oriented $2$-manifold.
{\it Trans. Amer. Math. Soc.} 237 (1978), 283--309.

\bibitem{Bourbaki}
N. Bourbaki, 
{\it \'El\'ements de math\'ematique. Fasc. XXXVII. Groupes et alg\`ebres de Lie. 
Chapitre II: Alg\`ebres de Lie libres. Chapitre III: Groupes de Lie.}
Actualit\'es Scientifiques et Industrielles, No. 1349. Hermann, Paris, 1972.

\bibitem{BM}
G. W. Brumfiel, J. W. Morgan,
Quadratic functions, the index modulo $8$, and a $Z/4$-Hirzebruch formula.
{\it Topology}  12  (1973), 105--122. 

\bibitem{CFK}
J. C. Cha, S. Friedl, T. Kim,
The cobordism group of homology cylinders.
Preprint (2009) \texttt{arXiv:0909$.$5580}, to appear in \emph{Compos. Math}.

\bibitem{CHM}
D. Cheptea, K. Habiro, G. Massuyeau,
A functorial LMO invariant for Lagrangian cobordisms.
{\it Geom. Topol.} 12 (2008), no. 2, 1091--1170. 

\bibitem{CL}
D. Cheptea, T. T. Q. Le,
A TQFT associated to the LMO invariant of three-dimensional manifolds.
{\it Comm. Math. Phys.} 272 (2007), no. 3, 601--634. 

\bibitem{CY}
L. Crane, D. Yetter,
On algebraic structures implicit in topological quantum field theories.
{\it J. Knot Theory Ramifications} 8 (1999),  no. 2, 125--163.

\bibitem{FS}
R. Fintushel, R. J. Stern,
Instanton homology of Seifert fibred homology three spheres.
{\it Proc. London Math. Soc. (3)} 61 (1990), no. 1, 109--137.

\bibitem{Furuta}
M. Furuta, 
Homology cobordism group of homology $3$-spheres.
{\it Invent. Math.} 100 (1990), no. 2, 339--355. 

\bibitem{Garoufalidis}
S. Garoufalidis,
The mystery of the brane relation.
{\it J. Knot Theory Ramifications} 11 (2002), no. 5, 725--737. 

\bibitem{GGP}
S. Garoufalidis, M. N. Goussarov, M. Polyak,
Calculus of clovers and finite type invariants of $3$-manifolds.
{\it Geom. Topol.} 5 (2001), 75--108.

\bibitem{GL_blinks}
S. Garoufalidis, J. Levine, 
Finite type $3$-manifold invariants, the mapping class group and blinks.
{\it J. Differential Geom.} 47 (1997), no. 2, 257--320. 

\bibitem{GL_FTI_Torelli}
S. Garoufalidis, J. Levine, 
Finite type $3$-manifold invariants and the structure of the Torelli group. I.  
{\it Invent. Math.}  131  (1998),  no. 3, 541--594.

\bibitem{GL}
S. Garoufalidis, J. Levine,
Tree-level invariants of 3-manifolds, Massey products and the Johnson homomorphism.
{\it Graphs and patterns in mathematics and theoretical physics}, 173--203, 
Proc. Sympos. Pure Math. 73, Amer. Math. Soc., Providence, RI, 2005.

\bibitem{GS}
H. Goda, T. Sakasai,
Abelian quotients of monoids of homology cylinders.
Preprint (2009) \texttt{arXiv:0905$.$4775}.

\bibitem{GA}
F. Gonz\'alez-Acu\~na,
$3$-dimensional open books.
Lectures, Univ. of Iowa Topology Seminar, 1974/75.

\bibitem{Goussarov}
M. N. Goussarov,
Finite type invariants and $n$-equivalence of 3-manifolds.
{\it Compt. Rend. Acad. Sc. Paris S\'erie I Math.} 329  (1999), no. 6, 517--522.

\bibitem{Goussarov_clovers}
M. N. Goussarov,
Variations of knotted graphs. The geometric technique of $n$-equivalence.
In Russian: {\it Algebra i Analiz}  12  (2000),  no. 4, 79--125.
English translation: {\it St. Petersburg Math. J.} 12 (2001), no. 4, 569--604.

\bibitem{Habegger}
N. Habegger,
Milnor, Johnson and tree-level perturbative invariants.
Preprint (2000).

\bibitem{HL_link-homotopy}
N. Habegger, X.-S. Lin,
The classification of links up to link-homotopy.
{\it J. Amer. Math. Soc.} 3 (1990), no. 2, 389--419. 

\bibitem{HL_concordance}
N. Habegger, X.-S. Lin,
On link concordance and Milnor's $\overline {\mu}$ invariants.  
{\it Bull. London Math. Soc.}  30  (1998),  no. 4, 419--428. 

\bibitem{HMasb}
N. Habegger, G. Masbaum,
The Kontsevich integral and Milnor's invariants.
{\it Topology} 39 (2000), no. 6, 1253--1289. 

\bibitem{HS}
N. Habegger, C. Sorger, 
An infinitesimal presentation of the Torelli group of a surface with boundary.
Preprint (2000).

\bibitem{Habiro}
K. Habiro,
Claspers and finite type invariants of links.
{\it Geom. Topol.} 4 (2000), 1--83.

\bibitem{HMass}
K. Habiro, G. Massuyeau,
Symplectic Jacobi diagrams and the Lie algebra of homology cylinders.
{\it J. Topology} 2 (2009), no. 3, 527--569.

\bibitem{Hain}
R. M. Hain,
Infinitesimal presentations of the Torelli groups.
{\it J. Amer. Math. Soc.} 10 (1997), no. 3, 597--651. 

\bibitem{Heap}
A. Heap, 
Bordism invariants of the mapping class group.
{\it Topology} 45 (2006), no. 5, 851--886. 

\bibitem{Hempel}
J. Hempel, 
{\it $3$-Manifolds.}
Ann. of Math. Studies, No. 86. Princeton University Press, Princeton, N. J.; University of Tokyo Press, Tokyo, 1976.

\bibitem{Hopf}
H. Hopf,
Beitr\"age zur Klassifizierung der Fl\"achenabbildungen.
{\it J. Reine Angew. Math.} 165 (1931), 225--236.

\bibitem{IO}
K. Igusa, K. E. Orr,
Links, pictures and the homology of nilpotent groups.
{\it Topology} 40 (2001), no. 6, 1125--1166. 

\bibitem{Jennings}
S. A. Jennings,
The group ring of a class of infinite nilpotent groups.
{\it Canad. J. Math.} 7 (1955), 169--187. 

\bibitem{Johnson_generation}
D. L. Johnson, 
Homeomorphisms of a surface which act trivially on homology.  
{\it Proc. Amer. Math. Soc.}  75  (1979), no. 1, 119--125.

\bibitem{Johnson_first_homomorphism}
D. L. Johnson,
An abelian quotient of the mapping class group $\mathcal{I}_g$.
{\it Math. Ann.} 249 (1980), no. 3, 225--242.

\bibitem{Johnson_quadratic}
D. L. Johnson,
Spin structures and quadratic forms on surfaces.
{\it J. London Math. Soc. (2)} 22 (1980), no. 2, 365--373. 

\bibitem{Johnson_BC}
D. L. Johnson,
Quadratic forms and the Birman--Craggs homomorphisms.
{\it Trans. Amer. Math. Soc.} 261 (1980), no. 1, 235--254. 

\bibitem{Johnson_survey}
D. L. Johnson,
A survey of the Torelli group.
\emph{Low-dimensional topology} (San Francisco, Calif., 1981), 165--179,
Contemp. Math. 20, Amer. Math. Soc., Providence, RI, 1983. 

\bibitem{Johnson_finite_generation}
D. L. Johnson,
The structure of the Torelli group. I. A finite set of generators for $\mathcal{I}$.
{\it Ann. of Math. (2)} 118 (1983), no. 3, 423--442.

\bibitem{Johnson_abelianization}
D. L. Johnson, 
The structure of the Torelli group. III. The abelianization of $\mathcal{I}$.
{\it Topology}  24  (1985),  no. 2, 127--144. 

\bibitem{Kawazumi}
N. Kawazumi,
Cohomological aspects of Magnus expansions.
Preprint (2005) \texttt{arXiv:math/0505497}.

\bibitem{Kerler}
T. Kerler,
Bridged links and tangle presentations of cobordism categories.
{\it Adv. Math.} 141 (1999),  no. 2, 207--281. 

\bibitem{Kirby}
R. C. Kirby,
{\it The topology of $4$-manifolds.}
Lecture Notes in Mathematics, 1374. Springer-Verlag, Berlin, 1989.

\bibitem{Kitano}
T. Kitano,
Johnson's homomorphisms of subgroups of the mapping class group, the Magnus expansion and Massey higher products of mapping tori.
{\it Topology Appl.}  69  (1996),  no. 2, 165--172.

\bibitem{Kontsevich_Gelfand}
M. Kontsevich,
Formal (non)commutative symplectic geometry. 
{\it The Gel'fand Mathematical Seminars}, 1990--1992, 173--187, 
Birkh\"auser Boston, Boston, MA, 1993. 

\bibitem{Kontsevich_ECM}
M. Kontsevich,
Feynman diagrams and low-dimensional topology. 
{\it First European Congress of Mathematics, Vol. II (Paris, 1992)}, 97--121,
Progr. Math., 120, Birkh\"auser, Basel, 1994. 

\bibitem{Labute}
J. P. Labute,
On the descending central series of groups with a single defining relation.
{\it J. Algebra} 14 (1970), 16--23. 

\bibitem{Lazard}
M. Lazard,
Sur les groupes nilpotents et les anneaux de Lie.
{\it Ann. Sci. Ecole Norm. Sup. (3)}  71 (1954), 101--190. 

\bibitem{LMO}
T. T. Q. Le, J. Murakami, T. Ohtsuki,
On a universal perturbative invariant of $3$-manifolds.
{\it Topology} 37 (1998), no. 3, 539--574.

\bibitem{Lescop}
C. Lescop,
{\it Global surgery formula for the Casson-Walker invariant.}
Annals of Mathematics Studies, 140. Princeton University Press, Princeton, NJ, 1996.

\bibitem{Levine}
J. Levine, 
Homology cylinders: an enlargement of the mapping class group.
{\it Algebr. Geom. Topol.} 1 (2001), 243--270.

\bibitem{Levine_addendum}
J. Levine,
Addendum and correction to: ``Homology cylinders: an enlargement of the mapping class group''.
{\it Algebr. Geom. Topol.} 2 (2002), 1197--1204.

\bibitem{Levine_quasi-Lie}
J. Levine,
Labeled binary planar trees and quasi-Lie algebras.
{\it Algebr. Geom. Topol.} 6 (2006), 935--948. 

\bibitem{Lickorish}
W. B. R. Lickorish,
A representation of orientable combinatorial $3$-manifolds.
{\it Ann. of Math. (2)} 76 (1962), 531--540. 

\bibitem{Lin}
X.-S. Lin,
Power series expansions and invariants of links. 
{\it Geometric topology (Athens, GA, 1993)}, 184--202,
AMS/IP Stud. Adv. Math., 2.1, Amer. Math. Soc., Providence, RI, 1997. 

\bibitem{MKS}
W. Magnus, A. Karrass, D. Solitar, 
{\it Combinatorial group theory: Presentations of groups in terms of generators and relations}. 
Interscience Publishers [John Wiley \& Sons, Inc.], New York-London-Sydney 1966.

\bibitem{Masbaum}
G. Masbaum,
Quantum representations of mapping class groups. 
{\it Groupes et g\'eom\'etrie}, 19--36,
SMF Journ. Annu., 2003, Soc. Math. France, Paris, 2003. 

\bibitem{Massuyeau_DSP}
G. Massuyeau,
Finite-type invariants of 3-manifolds and the dimension subgroup problem.
{\it J. Lond. Math. Soc. (2)} 75 (2007), no. 3, 791--811. 

\bibitem{Massuyeau_tree}
G. Massuyeau,
Infinitesimal Morita homomorphisms and the tree-level of the LMO invariant.
Preprint (2008) \texttt{arXiv:0809$.$4629}.

\bibitem{MM}
G. Massuyeau, J.-B. Meilhan,
Characterization of $Y_2$-equivalence for homology cylinders.
{\it J. Knot Theory Ramifications} 12 (2003), 493--522. 

\bibitem{Matveev}
S. V. Matveev,
Generalized surgeries of three-dimensional manifolds and representations of homology spheres.
In Russian: {\it Mat. Zametki} 42 (1987), no. 2, 268--278, 345. 
English translation: {\it Math. Notes} 42 (1987), no. 1-2, 651--656.

\bibitem{Milnor}
J. Milnor,
Spin structures on manifolds.
{\it Enseignement Math.}  9 (1963), no. 2, 198--203. 

\bibitem{Morita_Casson_1}
S. Morita,
Casson's invariant for homology $3$-spheres and characteristic classes of surface bundles. I.
{\it Topology} 28 (1989), no. 3, 305--323. 

\bibitem{Morita_Casson_2}
S. Morita,
On the structure of the Torelli group and the Casson invariant.
{\it Topology} 30 (1991), no. 4, 603--621. 

\bibitem{Morita_ICM}
S. Morita,
Mapping class groups of surfaces and three-dimensional manifolds. 
{\it Proceedings of the International Congress of Mathematicians, Vol. I, II (Kyoto, 1990)}, 
665--674, Math. Soc. Japan, Tokyo, 1991.

\bibitem{Morita_Abelian}
S. Morita,
Abelian quotients of subgroups of the mapping class group of surfaces.
{\it Duke Math. J.} 70 (1993), 699--726.

\bibitem{Morita_linear}
S. Morita,
A linear representation of the mapping class group of orientable surfaces and characteristic classes of surface bundles. 
{\it Topology and Teichm\"uller spaces (Katinkulta, 1995)},  159--186, World Sci. Publ., River Edge, NJ, 1996. 

\bibitem{Morita_Casson_3}
S. Morita,
Casson invariant, signature defect of framed manifolds and the secondary characteristic classes of surface bundles.
{\it J. Differential Geom.} 47 (1997), no. 3, 560--599. 

\bibitem{Morita_prospect}
S. Morita,
Structure of the mapping class groups of surfaces: a survey and a prospect.
{\it Proceedings of the Kirbyfest (Berkeley, CA, 1998)},  349--406,
Geom. Topol. Monogr., 2, Geom. Topol. Publ., Coventry, 1999. 

\bibitem{Morita_chapter}
S. Morita,
Introduction to mapping class groups of surfaces and related groups.  
{\it Handbook of Teichm\"uller theory. Vol. I},  353--386,
IRMA Lect. Math. Theor. Phys., 11, Eur. Math. Soc., Z\"urich, 2007. 

\bibitem{Morita_birthday}
S. Morita,
Symplectic automorphism groups of nilpotent quotients of fundamental groups of surfaces.
{\it Groups of diffeomorphisms: In honor of Shigeyuki Morita on the occasion of his 60th birthday},
Advanced Studies in Pure Mathematics 52 (2008), 443--468.

\bibitem{MO}
J. Murakami, T. Ohtsuki,
Topological quantum field theory for the universal quantum invariant.
{\it Comm. Math. Phys.} 188 (1997), no. 3, 501--520. 

\bibitem{Myers}
R. Myers,
Open book decompositions of $3$-manifolds.
{\it Proc. Amer. Math. Soc.} 72 (1978), no. 2, 397--402.

\bibitem{Ng}
K. Y. Ng,
Groups of ribbon knots.
{\it Topology} 37 (1998), no. 2, 441--458. 

\bibitem{Ohtsuki}
T. Ohtsuki,
{\it Quantum invariants. A study of knots, 3-manifolds, and their sets.}
Series on Knots and Everything, 29. World Scientific Publishing Co., Inc., River Edge, NJ, 2002.

\bibitem{Perron}
B. Perron,
Homomorphic extensions of Johnson homomorphisms via Fox calculus.
{\it Ann. Inst. Fourier (Grenoble)} 54 (2004), no. 4, 1073--1106.

\bibitem{Pickel}
P. F. Pickel,
Rational cohomology of nilpotent groups and Lie algebras.
{\it Comm. Algebra} 6 (1978), no. 4, 409--419.

\bibitem{Quillen}
D. G. Quillen,
Rational homotopy theory.
{\it Ann. of Math.}  90 (1969), no. 2, 205--295. 

\bibitem{Quillen_graded}
D. G. Quillen,
On the associated graded ring of a group ring.
{\it J. Algebra} 10 (1968), 411--418. 

\bibitem{Sakasai} 
T. Sakasai,
Homology cylinders and the acyclic closure of a free group.
{\it Algebr. Geom. Topol.} 6 (2006), 603--631.

\bibitem{Stallings}
J. Stallings,
Homology and central series of groups.
{\it J. Algebra} 2 (1965), 170--181. 

\bibitem{Sullivan}
D. Sullivan,
On the intersection ring of compact three manifolds.
{\it Topology} 14 (1975), no. 3, 275--277. 

\bibitem{Swarup}
G. A. Swarup,
On a theorem of C. B. Thomas.
{\it J. London Math. Soc. (2)} 8 (1974), 13--21. 

\bibitem{Thomas}
C. B. Thomas, 
The oriented homotopy type of compact $3$-manifolds.  
{\it Proc. London Math. Soc. (3)}  19  (1969), 31--44.

\bibitem{Turaev}
V. G. Turaev, 
Cohomology rings, linking coefficient forms and invariants of spin structures in three-dimensional manifolds.
In Russian: {\it Mat. Sb. (N.S.)} 120(162) (1983), no. 1, 68--83, 143. 
English translation: {\it Math. USSR Sb.} 48 (1984), 65--79.

\bibitem{Turaev_nilpotent}
V. G. Turaev,
Nilpotent homotopy types of closed $3$-manifolds. 
{\it Topology (Leningrad, 1982)}, 355--366,
Lecture Notes in Math., 1060, Springer, Berlin, 1984. 

\bibitem{Wall}
C. T. C. Wall,
Non-additivity of the signature.
{\it Invent. Math.} 7 (1969), 269--274. 

\bibitem{ZVC}
H. Zieschang, E. Vogt, H.-D. Coldewey,  
{\it Surfaces and planar discontinuous groups.} 
Translated from the German by John Stillwell. Lecture Notes in Mathematics, 835. Springer, Berlin, 1980

\end{thebibliography}
\end{document}